\documentclass[reqno]{amsart}
\usepackage[T1]{fontenc}
\usepackage[utf8]{inputenc}
\usepackage{amssymb}
\usepackage[mathscr]{eucal}
\usepackage{mathtools}
\usepackage{microtype}
\usepackage{enumerate}
\usepackage{xcolor}
\usepackage{graphicx}
\usepackage[percent]{overpic}
\usepackage[colorlinks=true]{hyperref}
\hypersetup{urlcolor=blue,citecolor=red,linkcolor=blue}
\usepackage[initials]{amsrefs}
\graphicspath{{genus3-views/}}
\numberwithin{equation}{section}
\theoremstyle{plain}
\newtheorem{theorem}{Theorem}[section]
\newtheorem{proposition}[theorem]{Proposition}
\newtheorem{lemma}[theorem]{Lemma}

\newtheorem{claim}{Claim}
\newtheorem*{theorem*}{Theorem}
\theoremstyle{remark}

\newtheorem{remark}[theorem]{Remark}
\newtheorem*{comment*}{Comment}
\newtheorem*{remark*}{Remark}
\theoremstyle{definition}

\newtheorem*{acknowledgements*}{Acknowledgements}
\newtheorem*{assumption*}{Assumption}
\newtheorem*{assumptions*}{Assumptions}
\newtheorem*{definition*}{Definition}
\newtheorem*{notation*}{Notation}
\newtheorem*{notations*}{Notations}
\newtheorem*{rh-pb*}{Basic RH-problem}
\providecommand{\BS}[1]{\boldsymbol{#1}}
\providecommand{\C}[1]{\mathcal{#1}}
\providecommand{\D}[1]{\mathbb{#1}}

\newcommand{\dd}{\mathrm{d}}
\newcommand{\eul}{\mathrm{e}}
\newcommand{\ii}{\mathrm{i}}
\newcommand{\vece}{\mathrm{e}}
\providecommand{\abs}[1]{\lvert#1\rvert}
\providecommand{\accol}[1]{\lbrace#1\rbrace}
\providecommand{\croch}[1]{\lbrack#1\rbrack}
\providecommand{\norm}[1]{\lVert#1\rVert}
\renewcommand{\Im}{\operatorname{Im}}
\newcommand{\model}{\operatorname{mod}}
\newcommand{\modulo}{\mathrm{mod}}
\newcommand{\appr}{\mathrm{app}}
\newcommand{\asympt}{\mathrm{as}}
\newcommand{\diff}{\mathrm{dif}}

\newcommand{\ord}{\mathrm{O}}
\newcommand{\osmall}{\mathrm{o}}
\renewcommand{\Re}{\operatorname{Re}}

\newcommand{\Airy}{\mathrm{Ai}}
\newcommand{\inv}{\mathrm{inv}}

\DeclareMathOperator{\GL}{GL}

\begin{document}
%---------------------------------------------------------%
\title[NLS with step-like oscillating background: genus 3 sector]{The focusing NLS equation with step-like oscillating background: the genus 3 sector}
%---------------------------------------------------------%
\author[A. Boutet de Monvel]{Anne Boutet de Monvel}
\address{AB: Institut de Math\'ematiques de Jussieu-Paris Rive Gauche, Universit\'e de Paris, 75205 Paris Cedex 13, France.}
\email{anne.boutet-de-monvel@imj-prg.fr}
%---------------------------------------------------------%
\author[J. Lenells]{Jonatan Lenells}
\address{JL: Department of Mathematics, KTH Royal Institute of Technology, 100 44 Stockholm, Sweden.}
\email{jlenells@kth.se}
%---------------------------------------------------------%
\author[D. Shepelsky]{Dmitry Shepelsky}
\address{DS: B.~Verkin Institute for Low Temperature Physics and Engineering, 61103 Kharkiv, Ukraine.} 
\email{shepelsky@yahoo.com}
%---------------------------------------------------------%
\date{}
%---------------------------------------------------------%
\begin{abstract}
We consider the Cauchy problem for the focusing nonlinear Schr\"odinger equation with initial data approaching different plane waves $A_j\eul^{\ii\phi_j}\eul^{-2\ii B_jx}$, $j=1,2$ as $x\to\pm\infty$. The goal is to determine the long-time asymptotics of the solution, according to the value of $\xi=x/t$. The general situation is analyzed in a recent paper where we develop the Riemann--Hilbert approach and detect different asymptotic scenarios, depending on the relationships between the parameters $A_1$, $A_2$, $B_1$, and $B_2$. In particular, in the shock case $B_1<B_2$, some scenarios include genus $3$ sectors, i.e., ranges of values of $\xi$ where the leading term of the asymptotics is given in terms of hyperelliptic functions attached to a Riemann surface $M(\xi)$ of genus three. The present paper is devoted to the complete asymptotic analysis in such a sector.
\end{abstract}
%---------------------------------------------------------%
\maketitle
\tableofcontents
%---------------------------------------------------------%
%:s.1
%---------------------------------------------------------%
\section{Introduction}  \label{sec:intro}

We consider the Cauchy problem for the focusing nonlinear Schr\"odinger (NLS) equation
\begin{subequations}\label{nlsic}
\begin{alignat}{2}  \label{nls}
&\ii q_t+q_{xx}+2\abs{q}^2q=0,&\qquad&x\in\D{R},\quad t\geq 0,\\
\label{ic}
&q(x,0)=q_0(x),&&x\in\D{R},
\end{alignat}
\end{subequations}
when the initial data behave like oscillatory waves as $x\to\pm\infty$:
\begin{equation}\label{q0-limits}
q_0(x)\sim 
\begin{cases}
A_1\eul^{\ii\phi_1}\eul^{-2\ii B_1 x},&x\to -\infty,\\
A_2\eul^{\ii\phi_2}\eul^{-2\ii B_2 x},&x\to +\infty,
\end{cases}
\end{equation}
where $\accol{A_j,B_j,\phi_j}_1^2$ are real constants such that $A_j>0$. Our general objective is the study of the long-time behavior of the solution $q(x,t)$ of \eqref{nlsic}--\eqref{q0-limits}, according to the value of $\xi\coloneqq x/t$. The Cauchy problem \eqref{nlsic}-\eqref{q0-limits} has to be supplemented with boundary conditions for $t>0$. These boundary conditions are the natural extensions of \eqref{q0-limits} to $t>0$, namely
\begin{subequations}\label{cauchy-asbg}
\begin{equation}   \label{cauchy-asy}
\int_0^{(-1)^j\infty}\abs{q(x,t)-q_{0j}(x,t)}\dd x<\infty\quad\text{for all }t\geq 0,\qquad j=1,2,
\end{equation}
where $q_{0j}(x,t)$, $j=1,2$ are the plane wave solutions of the NLS equation satisfying the initial conditions $q_{0j}(x,0)=A_j\eul^{\ii\phi_j}\eul^{-2\ii B_jx}$, that is,
\begin{equation}  \label{bg}
q_{0j}(x,t)=A_j\eul^{\ii\phi_j}\eul^{-2\ii B_jx+2\ii\omega_jt},\quad\omega_j\coloneqq A_j^2-2B_j^2.
\end{equation}
\end{subequations}
%---------------------------------------------------------%
%:s.1.1
%---------------------------------------------------------%
\subsection{Step-like boundary conditions}

Nonlinear integrable equations with nonzero boundary (as $x\to\pm\infty$) conditions and, particularly, with conditions that are different at different infinities (the so-called step-like conditions) have attracted the interest of physicists and mathematicians since the 1970s \cites{Bik89,GP73,GP74,Khr75,Khr76,Ven86}, due to remarkable features of their solutions, in particular, a multitude of patterns for the large time behavior. The application of the inverse scattering transform method to the Cauchy problem for nonlinear integrable equations with step-like initial data rests on the analysis of direct and inverse scattering problems for associated linear equations (the Lax pair) with step-like potentials originally developed in \cite{BF62} with later contributions in \cite{CK85}. 

Recent developments for the NLS equation have been reported in 
\cites{DPV13,BP14,BFP16} (for the defocusing equation), in \cites{BK14,BM17} (for the focusing equation \eqref{nlsic}-\eqref{q0-limits} with $A_1=A_2\neq 0$ and $B_1=B_2=0$), in \cite{DPV14}
(for the focusing equation with $A_2>A_1>0$ and $B_1=B_2=0$), and in \cite{BKS11} (for the focusing equation with $A_1=0$, $A_2\neq 0$ and $B_2\neq 0$). In \cite{BLS20a} we have developed the Riemann--Hilbert approach for the focusing NLS equation in the case $B_1\neq B_2$ with $A_1\neq 0$ and $A_2\neq 0$, reported the possible long-time scenarios and proved the existence of a sector in the $(x,t)$-half plane where the asymptotics is given in terms of hyperelliptic functions of genus $2$. 

For other works on equations with nonzero boundary conditions we refer to \cites{EGK13,EMT18,GM20,LN08,LN14,Mi16,No05}.

%---------------------------------------------------------%
%:s.1.2
%---------------------------------------------------------%
\subsection{Rarefaction and shock}

Regarding the long time behavior of solutions, two cases are usually distinguished: expansive (rarefaction) \cites{AEL16,DKZ96,LN08} and compressive (shock) \cites{EMT18,LN14}. For the Korteweg--de Vries equation, the rarefaction and shock cases are associated with step up and step down boundary conditions, respectively, whereas in the case of the focusing NLS equation, the difference between these two cases is characterized by the relationship between $B_1$ and $B_2$: $B_1>B_2$ leads to rarefaction while $B_1<B_2$ leads to the development of dispersive shock waves \cite{BLS20a}. 

It turns out that the solutions of equations with shock step-like conditions exhibit a richer large-time behavior: depending on the asymptotic sector under consideration, one can see a number of different large time patterns \cite{BLS20a}. In some cases, they can be qualitatively caught using the Whitham modulated equations \cite{Bio18}.

%---------------------------------------------------------%
%:s.1.3
%---------------------------------------------------------%
\subsection{The focusing NLS equation}

In \cite{BV07}, the long-time asymptotic analysis was presented for the symmetric case where $A_1=A_2=1$ and $B_1=-B_2$ (which is not a restriction provided $B_1\neq B_2$). The analysis was based on the nonlinear steepest descent method introduced by Deift and Zhou \cite{DZ93} and extended by Deift, Venakides, and Zhou \cites{DVZ94,DVZ97}, which was applied to the representation of the solution of the problem \eqref{nlsic}--\eqref{q0-limits} in terms of an associated Riemann--Hilbert (RH) problem. The asymptotic picture presented there consists of five sectors in the $(x,t)$ half-plane ($t>0$) with different qualitative behaviors. Namely, the central sector (the \emph{residual} region, containing the half-axis $x=0$, $t>0$) was described as a modulated elliptic (genus $1$) wave \cite{BV07}*{Theorem 1.2}, two contiguous sectors (the \emph{transition} region) as modulated hyperelliptic (genus 2) waves \cite{BV07}*{Theorem 1.3}, and two sectors (the \emph{plane-wave} region, adjacent to the $x$-axis), as modulated plane (genus $0$) waves.

Two important features of \cite{BV07} are (i) the development of the $g$-function mechanism used in the transformations of the original RH problem (leading to explicitly solvable model problems) and (ii) rigorous error estimates based on the analysis of associated local (parametrix) RH problems.

A closer look at the conditions characterizing various asymptotic regions for the problem \eqref{nlsic}--\eqref{q0-limits} has revealed \cite{BLS20a} that the asymptotic scenario presented in \cite{BV07} is valid only for a particular range of the parameters involved in the boundary conditions \eqref{q0-limits}. In \cite{BLS20a} we analyzed the general situation, developing the RH formulation of the problem. This led us to detect several asymptotic scenarios, depending on the relationships between the parameters $A_1$, $A_2$, $B_1$, and $B_2$. Each scenario corresponds to a division of the $(x,t)$ half-plane into sectors $\xi_1<\xi<\xi_2$ where the long-time behavior of the solution $q$ is qualitatively different. Each sector is characterized by a $g$-function of a particular type, associated with a Riemann surface of genus $0$, $1$, $2$, or $3$. In such a sector the solution can be slowly decaying as $t\to+\infty$, or its long time behavior can be expressed in terms of functions attached to Riemann surfaces of different genera.

%---------------------------------------------------------%
%:s.1.5
%---------------------------------------------------------%
\subsection{Summary of results}

In the shock case, i.e., when $B_1<B_2$, some of these scenarios include genus $3$ sectors. More precisely, a genus $3$ sector arises for $\xi$ close (but not equal) to $0$ as stated in the first theorem below.

The asymptotic analysis in sectors of genus $1$ or $2$ has already been performed, see, e.g., \cite{BV07}.  Our goal here is to perform the long time asymptotic analysis of $q(x,t)$ in a sector of genus $3$.

Our main result gives precisely the leading and subleading terms of the long time asymptotics of $q$ in such a sector $\xi_1<\xi<\xi_2$. The leading term is expressed by means of hyperelliptic theta functions (attached to Riemann surfaces of genus $3$) while the subleading term is expressed in terms of parabolic cylinder and Airy functions. These two terms manifest the universality of the asymptotic behavior.

%---------------------------------------------------------%
%:s.1.5
%---------------------------------------------------------%
\subsection{Contents}

In Section~\ref{sec:main} we state the main result that gives the long time asymptotics of $q$ in a genus $3$ sector. We also give a statement on the existence of such a sector. In Section~\ref{sec:rhp} we present the RH formulation of the problem together with the Riemann surface and the $g$-function that are specific to the genus $3$ sector under consideration. Sections~\ref{sec:transfos} to \ref{sec:final} present the steps in the proof of the main result. In Section~\ref{sec:transfos} we introduce a series of transformations of the original RH problem that leads to a RH problem whose asymptotic analysis will be performed using a ``model problem'' supplemented by ``local problems''. In Section~\ref{sec:model} we give the solution of the ``model problem'' in terms of Riemann theta functions attached to a hyperelliptic Riemann surface $M(\xi)$. This solution contributes to the leading term of the asymptotics in the main theorem. In Section~\ref{sec:local} local problems are introduced to correct the non-uniform approximation by the model problem near some points of the complex $k$-plane. Their solutions contribute to the subleading term in the asymptotics. They are obtained from the solutions of standard models described and solved in appendices \ref{sec:A} and \ref{sec:B}, in terms of parabolic cylinder and Airy functions. Finally, in Section~\ref{sec:final} we paste together the solution of the model problem with those of the local problems to obtain an appropriate approximation giving us the main result.

%---------------------------------------------------------%
%:s.2
%---------------------------------------------------------%
\section{Main result}  \label{sec:main}

Henceforth, we assume we are in the shock case, i.e., $B_1<B_2$.

%---------------------------------------------------------%
%:s.2.1
%---------------------------------------------------------%
\subsection{Notation}

Let $E_j\coloneqq B_j+\ii A_j$ in the complex $k$-plane $\D{C}$. We denote by $\Sigma_j$, $j=1,2$ the vertical segment $\croch{\bar E_j,E_j}$ oriented upward. See Figure~\ref{fig:basic-contour}. 

Let $\D{C}^+=\accol{\Im k>0}$ and $\D{C}^-=\accol{\Im k<0}$ denote the open upper and lower halves of the complex $k$-plane. The Riemann sphere will be denoted by $\hat{\D{C}}=\D{C}\cup\accol{\infty}$. We write $\ln k$ for the logarithm with the principal branch, that is, $\ln k=\ln\abs{k}+\ii\arg k$ where $\arg k\in(-\pi,\pi\rbrack$. Unless specified otherwise, all complex powers will be defined using the principal branch, i.e., $k^{\alpha}=\eul^{\alpha\ln k}$. We let $f^*(k)\coloneqq\overline{f(\bar k)}$ denote the Schwarz conjugate of a complex-valued function $f(k)$.

Given an open subset $D\subset\hat{\D{C}}$ bounded by a piecewise smooth contour $\Sigma$, we let $\dot E^2(D)$ denote the Smirnoff class consisting of all functions $f(k)$ analytic in $D$ with the property that for each connected component $D_j$ of $D$ there exist curves $\accol{C_n}_1^{\infty}$ in $D_j$ such that the $C_n$ eventually surround each compact subset of $D_j$ and $\sup_{n\geq 1}\norm{f}_{L^2(C_n)}<\infty$. We let $E^{\infty}(D)$ denote the space of bounded analytic functions $D\to\D{C}$. RH problems in the paper are generally $2\times 2$ matrix-valued. They are formulated in the $L^2$-sense using Smirnoff classes (see \cites{Le17,Le18}): 
\begin{equation}   \label{rhp}
\begin{cases}
m\in I+\dot E^2(\D{C}\setminus\Sigma),&\\
m_+(k)=m_-(k)J(k)&\text{for a.e. }k\in\Sigma,
\end{cases}
\end{equation}
where $m_+$ and $m_-$ denote the boundary values of $m$ from the left and right sides of the contour $\Sigma$. Further on, the contours are invariant under complex conjugation and for all $2\times 2$ matrix-valued RH problems the jump matrix $J\equiv J(k)$ satisfies
\begin{equation}  \label{jump-symm}
J=
\begin{cases}
\sigma_3\sigma_1J^*\sigma_1\sigma_3,&k\in\Sigma\setminus\D{R},\\
\sigma_3\sigma_1(J^*)^{-1}\sigma_1\sigma_3,&k\in\Sigma\cap\D{R},
\end{cases}
\end{equation}
where $\sigma_1\coloneqq\left(\begin{smallmatrix}0&1\\1&0\end{smallmatrix}\right)$ and $\sigma_3\coloneqq\left(\begin{smallmatrix}1&0\\0&-1\end{smallmatrix}\right)$. Together with uniqueness of the solution of the RH problem \eqref{rhp}, this implies the symmetry
\begin{equation}  \label{m-symm}
m=\sigma_3\sigma_1m^*\sigma_1\sigma_3,\quad k\in\D{C}\setminus\Sigma.
\end{equation}

%---------------------------------------------------------%
%:s.2.2
%---------------------------------------------------------%
\subsection{Asymptotics in a genus 3 sector}  \label{sec:thm}

According to \cite{BLS20a}*{Section 2.2}, since $B_1<B_2$, we may assume, without loss of generality, that 
\[
B_2=-B_1=1,\quad\phi_2=0,\ \phi_1=\phi.
\]
We thus have $E_1=-1+\ii A_1$ and $E_2=1+\ii A_2$. 

%-------------------%
\begin{assumptions*}[on $q_0$ and $q$]
In what follows, $q\colon\D{R}\times\lbrack 0,\infty)\to\D{C}$ is a smooth solution of \eqref{nls}-\eqref{ic} with initial data
\begin{equation}  \label{id}
q_0(x)=
\begin{cases}
A_1\eul^{\ii\phi}\eul^{2\ii x},&x<-C,\\
A_2\eul^{-2\ii x},&x>C,
\end{cases}
\end{equation}
where $C$, $A_1$, $A_2$ are positive constants and $\phi\in\D{R}$. We also assume that $q$ satisfies
\begin{equation}   \label{cauchy-as}
q(x,t)\sim q_{0j}(x,t)\quad\text{as }x\to (-1)^j\infty
\end{equation}
for all $t\geq 0$, where $q_{0j}(x,t)=A_j\eul^{\ii\phi_j}\eul^{-2\ii B_jx+2\ii\omega_jt}$, and $\omega_j=A_j^2-2B_j^2$.
\end{assumptions*}
%-------------------%

The main result deals with the long time behavior of $q$ in a genus $3$ sector $\xi_1<\xi<\xi_2$, $\xi=x/t$ in the $(x,t)$ half plane. That means a sector where the leading term of the asymptotics can be expressed in terms of quantities associated to a genus $3$ Riemann surface $M\equiv M(\xi)$. 

To perform the long time asymptotics we first translate the Cauchy problem into a RH problem (see Section~\ref{sec:rhp-0}), then apply the nonlinear descent method which is the main tool to study the long time behavior of solutions of integrable equations \cites{DZ93,DVZ94}. Its application to problems with nonzero backgrounds goes through the so-called $g$-function mechanism. This mechanism is introduced in situations where the jump matrix of the RH problem has some exponentially growing or oscillating entries as $t\to+\infty$. Its core idea is to replace (for some ranges of values of $\xi=x/t$) the phase function $\theta(\xi,k)=2k^2+\xi k$ in the original jump \eqref{basic-jump} by an appropriate function $g(\xi,k)$ (analytic in $k$, up to jumps across some contour) in such a way that after series of transformations 
of the original RH problem we get RH problems with jumps that are constant or decay to the identity as $t\to+\infty$.

%-------------------%
\begin{assumption*}[existence of a genus $3$ sector]
We henceforth assume that there exists a genus $3$ sector $\xi_1<\xi<\xi_2$, a sector where the system of equations \eqref{dg-conditions} has a unique solution $\dd g(\xi,k)$ and that $\xi$ lies in this sector.
\end{assumption*}
%-------------------%

%-------------------%
\begin{comment*}
By \eqref{dg}, the differential form $\dd g(\xi,k)$ is entirely characterized by its zeros, that is, by $\mu(\xi)\in\D{R}$, $\alpha(\xi)\in\D{C}^+$, and $\beta(\xi)\in\D{C}^+$, with $\alpha\neq\beta$ and $\alpha,\beta\neq E_1,E_2$. Thus, solving the system of equations \eqref{dg-conditions} means finding $\alpha(\xi)\in\D{C}^+$, $\beta(\xi)\in\D{C}^+$, and $\mu(\xi)\in\D{R}$ with the above constraints. See Figures~\ref{fig:Img-contour-plot}, \ref{fig:canonical-basis}, and \ref{fig:curves-aj-bj}.
\end{comment*}
%-------------------%

According to \cite{BLS20a}*{Section 5.3} where scenarios for the symmetric case $A_1=A_2=A$ are listed this should happen when $\frac{A}{B}>1$ (here $B\coloneqq B_2=-B_1=1$). This condition corresponds to the 3rd, 4th, and 5th scenarios. In these cases a genus $3$ sector should exist for $\xi\neq 0$ close to $0$. This is indeed confirmed by the following existence result. This result can be proven following the ideas used in \cite{BLS20a}*{Section 6} to establish the existence of a genus $2$ sector.

%-------------------%
%:thm
%-------------------%
\begin{theorem*}[existence result]
Suppose we are in the symmetric case with $\frac{A}{B}>1$. Then there exists a $\xi_m>0$ and a smooth curve 
\[
\xi\mapsto(\alpha(\xi),\beta(\xi),\mu(\xi))\in(\D{C}^+)^2\times\D{R}
\]
defined for $\xi\in(0,\xi_m)$ such that the following hold:
\begin{enumerate}[\rm(a)]
\item 
For each $\xi\in(0,\xi_m)$, $(\xi,\alpha(\xi),\beta(\xi),\mu(\xi))$ is a solution of the system of equations~\eqref{dg-conditions}, which means that a $g$-function appropriate for this range of $\xi$ is given by \eqref{dg}.
\item 
$\mu(\xi)$ is a smooth real-valued function of $\xi\in(0,\xi_m)$.
\item 
The curves $\xi\mapsto\alpha(\xi)$ and $\xi\mapsto\beta(\xi)$ are smooth maps $(0,\xi_m)\to\D{C}^+\setminus\{E_1,E_2\}$ and $\alpha(\xi)\neq\beta(\xi)$ for all $\xi\in(0,\xi_m)$.
\item 
As $\xi\downarrow 0$, we have
\begin{equation}\label{behavioratxiE1}
\alpha(\xi)\to\ii\alpha_0,\quad\beta(\xi)\to\ii\alpha_0,\quad\mu(\xi)\to 0,
\end{equation}
where $\alpha_0\coloneqq\sqrt{A^2-B^2}$. 
\item 
As $\xi\uparrow\xi_m$, at least one of the following occurs: 
\begin{enumerate}[\rm(i)]
\item 
$\beta(\xi)$ and $\overline{\beta(\xi)}$ merge at a point on the real axis.
\item 
$\alpha(\xi)$ hits $E_1$.
\item 
Two or more of the four branch points $\alpha(\xi)$, $\beta(\xi)$, $E_1$, $E_2$ collide. 
\item 
$\xi_m=+\infty$.
\end{enumerate}
\end{enumerate}
\end{theorem*}
%-------------------%

%-------------------%
\begin{assumption*}[on $a(k)$ and $b(k)$]
In what follows we assume that the scattering coefficients $a(k)$ and $b(k)$ defined in \cite{BLS20a}*{Section 2.4} are nonzero for any $k\in\D{C}^+\cup\D{R}$.
\end{assumption*}
%-------------------%

%-------------------%
\begin{remark*}
The assumption that $a(k)$ is nonzero for any $k\in\D{C}^+\cup\D{R}$, $k\neq E_1, E_2$ means that we deal with the solitonless case. On the other hand, the assumption that $a(E_1)\neq 0$ and $a(E_2)\neq 0$ signifies that we focus, for simplicity, on the generic case 
(see the discussion of the behavior of $a(k)$ near $E_1$ and $E_2$ below, in Section~\ref{sec:rhp}).
\end{remark*}
%-------------------%

The next theorem is our main result. In this theorem, where we return to the general (not necessarily symmetric) case, the long time asymptotics is expressed in terms of quantities associated to the genus $3$ Riemann surface $M\equiv M(\xi)$. This hyperelliptic Riemann surface is defined by \eqref{M-def} with distinct branch points $\alpha\equiv\alpha(\xi)$ and $\beta\equiv\beta(\xi)$ determined by the system of equations \eqref{dg-conditions}.

%-------------------%
%:thm 2.1
%-------------------%
\begin{theorem}[asymptotics in a genus $3$ sector]  \label{main-thm}
Let $\C{I}$ denote a compact subset of a genus $3$ sector $(\xi_1,\xi_2)$. The asymptotics in this sector is given by
\begin{equation}  \label{main-asymptotics}
q(x,t)=Q_0(\xi,t)+\frac{Q_1(\xi,t)}{\sqrt{t}}+\ord\left(\frac{\ln t}{t}\right),\quad t\to\infty,\ \ \xi\in\C{I},
\end{equation}
where the error term is uniform with respect to $\xi\in\C{I}$, the leading order term $Q_0$ is given by
\begin{equation}  \label{leading}
Q_0(\xi,t)=\eul^{2\ii(tg^{(0)}+h(\infty))}\Im(E_1+E_2+\alpha+\beta)\frac{\Theta(\varphi(\infty^+)+d)\Theta(\varphi(\infty^+)-v(t)-d)}{\Theta(\varphi(\infty^+)+v(t)+d)\Theta(\varphi(\infty^+)-d)}\,,
\end{equation}
and the coefficient $Q_1$ of the subleading term is given by \begin{equation}  \label{subleading}
Q_1(\xi,t)=-2\ii\eul^{2\ii(tg^{(0)}+h(\infty))}\frac{(Y_{\mu}(x,t,\mu)m_1^XY_{\mu}(x,t,\mu)^{-1})_{12}}{\psi_{\mu}(\mu)}.
\end{equation}
$\Theta$ is the Riemann theta function associated with the genus $3$ Riemann surface $M$ and defined in \eqref{Theta-def}. The Abel map $\varphi$ is defined in \eqref{Abel-map}. The constants $g^{(0)}\equiv g^{(0)}(\xi)$ and $h(\infty)\equiv h(\xi,\infty)$ are defined in \eqref{g-asymptotics} and \eqref{h(infty)}, respectively. The vector-valued function $v(t)$ is defined in \eqref{v-def}. The vector $d\equiv d(\xi)$ is defined in \eqref{d-def}. The constant $\psi_{\mu}(\mu)\equiv\psi_{\mu}(\xi,\mu)>0$ is defined in \eqref{psimu-def}. The matrices $Y_{\mu}$ and $m_1^X\equiv m_1^X(\xi)$ are defined in \eqref{Ymu-def} and \eqref{m1X-def}, respectively.
\end{theorem}
%-------------------%

This theorem deals with a region where the two branch points $\alpha(\xi)$ and $\beta(\xi)$ are distinct. In \cite{BLS20c} we perform the asymptotic analysis in a transition zone where $\xi=\xi(t)$ is asymptotically close (as $t\to+\infty$) to a critical value $\xi_0$ at which $\alpha(\xi)$ and $\beta(\xi)$ merge.

On the other hand, according to our analysis of asymptotic scenarios \cite{BLS20a}, the genus $3$ sector is also adjacent to a genus $2$ sector, or to a genus $1$ sector, or to a genus $0$ (plane wave) sector. In the last two cases, the transition between the sectors is associated with the situation when a characteristic curve (the infinite branch of $\Im g = 0$ for the corresponding ``$g$-function'' \cite{BLS20c}) hits the endpoints of a spectral arc. With this respect we notice that a transition mechanism has been recently reported \cites{BM19,KM19} that involves the so-called ``asymptotic solitons'' and is associated with endpoints of a spectral arc. Whether it (or a similar mechanism) is applicable to transitions from the genus $3$ sector is an open problem.

The remainder of the paper is devoted to proving Theorem~\ref{main-thm}. 

%---------------------------------------------------------%
%:s.3
%---------------------------------------------------------%
\section{RH problem, Riemann surface, and $g$-function}\label{sec:rhp}
%---------------------------------------------------------%
%:s.3.1
%---------------------------------------------------------%
\subsection{The basic RH problem}\label{sec:rhp-0}

We briefly recall the description of the basic RH problem introduced in \cite{BLS20a}*{Section~2.5}. Its formulation involves two spectral functions $a(k)$ and $b(k)$ which are defined through the scattering relation between the Jost solutions $\Phi_1$ and $\Phi_2$:
\begin{equation}  \label{scattering}
\Phi_2(x,t,k)=\Phi_1(x,t,k)
\begin{pmatrix}
a^*(k)&b(k)\\
-b^*(k)&a(k)
\end{pmatrix},\quad k\in\D{R},\quad k\neq B_1,B_2.
\end{equation}
See \cite{BLS20a}*{Section~2.4}. The Jost solutions $\Phi_1$ and $\Phi_2$ are defined through Volterra integral equations \cite{BLS20a}*{(2.15)} which, for $t=0$, are defined only from the initial data $q_0(x)$. In what follows $\mu^{(i)}$ denotes the $i$-th column of a matrix $\mu$. Then, under conditions \eqref{cauchy-asbg}, the column $\Phi_1^{(1)}$ is analytic in $\D{C}^+\setminus\Sigma_1$ with a jump across $\Sigma_1$, $\Phi_2^{(2)}$ is analytic in $\D{C}^+\setminus\Sigma_2$, $\Phi_1^{(2)}$ is analytic in $\D{C}^-\setminus\Sigma_1$,  and $\Phi_2^{(1)}$ is analytic in $\D{C}^-\setminus\Sigma_2$. See \cite{BLS20a}*{Proposition~2.1}. Setting $t=0$ in \eqref{scattering} it follows that $a(k)$ and $b(k)$ are uniquely determined by $q_0(x)$, and similarly for the reflection coefficient which is defined by
\begin{equation}   \label{reflec-anal}
r(k)\coloneqq\frac{b^*(k)}{a(k)}\,.
\end{equation}
If, in addition, the initial data $q_0(x)$ satisfy \eqref{id}, then $a(k)$ and $b(k)$ are both analytic functions of $k\in\hat{\D{C}}\setminus(\Sigma_1\cup\Sigma_2)$. See \cite{BLS20a}*{Section 2.5.3}.

The smoothness of $q_0$ implies that
\begin{equation}  \label{ab-at-infty}
a(k)=1+\ord(k^{-1}),\qquad b(k)=\ord(k^{-N}),\quad k\in\D{R},\ k\to\infty,
\end{equation}
for every $N\geq 1$. Moreover, under assumption \eqref{id} on $q_0$,
\begin{equation}  \label{reflec-infty}
a(k)=1+\ord\left(\tfrac{\eul^{4C\abs{\Im k}}+1}{k}\right),\qquad b(k)= \ord\left(\tfrac{\eul^{4C\abs{\Im k}}+1}{k}\right),\quad k\to\infty.
\end{equation}
See \cite{BLS20a}*{Lemma 2.4}.

On $\Sigma_1$ and $\Sigma_2$, the spectral functions $a(k)$ and $b(k)$ satisfy the relations 
\begin{subequations}  \label{ab12}
\begin{alignat}{2}   \label{ab1}
&\begin{cases}
a_+=-\ii\eul^{-\ii\phi_1}b_-,&\\
b_+=-\ii\eul^{\ii\phi_1}a_-,&
\end{cases}&\quad&k\in\Sigma_1,\\
\label{ab2}
&\begin{cases}
a_+=-\ii\eul^{\ii\phi_2}b_-^*,&\\
b_+=\ii\eul^{\ii\phi_2}a_-^*,&
\end{cases}&&k\in\Sigma_2.
\end{alignat}
\end{subequations}
See \cite{BLS20a}*{Section 2.5.3}. Their behavior near $E_1$, $E_2$, $\bar E_1$, and $\bar E_2$ is described in \cite{BLS20a}*{Section 2.5.5}. We have one of the two possibilities: 
\begin{enumerate}[(i)]
\item
(generically) $a(k)$ has singularities of order $\frac{1}{4}$ at $E_j$ and $\bar E_j$, for $j=1,2$; precisely, $a(k)\sim c_j(k-E_j)^{-\frac{1}{4}}$ near $E_j$ and $a(k)\sim c_j'(k-\bar E_j)^{-\frac{1}{4}}$ near $\bar E_j$ with nonzero complex constants $c_j$, $c_j'$. Simultaneously, $b(k)\sim c_1''(k-E_1)^{-\frac{1}{4}}$ near $E_1$ and $b^*(k)\sim c_2''(k-E_2)^{-\frac{1}{4}}$ near $E_2$ with nonzero complex constants $c_j''$. In this case $r(k)$ is bounded near $E_j$ and $\bar E_j$.
\item
(virtual level cases) $a(k)$ vanishes at $E_j$, for $j=1$ and/or $j=2$; precisely, $a(k)\sim\gamma_j(k-E_j)^{\frac{1}{4}}$ for $k$ near to $E_j$ with nonzero complex constants $\gamma_j$. Moreover, if this happens for $j=1$, then $b(k)\sim\gamma_1'(k-E_1)^{1/4}$ near $E_1$, and, if this happens for $j=2$, then $b^*(k)\sim\gamma_2'(k-E_2)^{1/4}$ near $E_2$, with nonzero complex constants $\gamma_j'$.
\end{enumerate}    
Moreover, the behavior of $\Phi_j^{(l)}$, $l=1,2$ at $k=E_j$ or $k=\bar E_j$, $j=1,2$ is $\ord((k-E_j)^{-\frac{1}{4}})$ or $\ord((k-\bar E_j)^{-\frac{1}{4}})$ (generically, they all have singularities of order $-1/4$).

In \cite{BLS20a}*{Section 2.5} we introduced the basic RH problem as follows. We first consider the $2\times 2$ matrix-valued function $m(x,t,k)$ by
\begin{equation}  \label{m}
m(x,t,k)\coloneqq\begin{cases}
\begin{pmatrix}\frac{\Phi_1^{(1)}}{a}&\Phi_2^{(2)}\end{pmatrix}\eul^{(\ii kx+2\ii k^2t)\sigma_3},& k\in\D{C}^+,\\
\begin{pmatrix}\Phi_2^{(1)}&\frac{\Phi_1^{(2)}}{a^*}\end{pmatrix}\eul^{(\ii kx+2\ii k^2t)\sigma_3}, 
&k\in\D{C}^-,
\end{cases}
\end{equation}
and then observe that $m(x,t,k)$ can be characterized as the solution of a RH problem whose data are uniquely determined by $q_0(x)$ through the associated scattering functions $a(k)$ and $b(k)$. This approach is a direct analogue of the case of Cauchy problems with decaying initial data. It turns out that for various reasons it is necessary to regularize this problem. We will do this by means of the function $\nu_1$ which is defined as in \cite{BLS20a}*{Section 2.3} by
\begin{equation}  \label{nu1}
\nu_1(k)\coloneqq\left(\frac{k-E_1}{k-\bar E_1}\right)^{\frac{1}{4}},\qquad\nu_1(\infty)=1,
\end{equation}
and has a jump $\nu_{1+}(k)=\ii\nu_{1-}(k)$ across $\Sigma_1$.

In the virtual level (non-generic) case with $a(E_1)=0$ the function $m$ can have a singularity of order $\frac{1}{2}$ at $k=E_1$. Then it is necessary to regularize $m$ to get a RH problem in the $L^2$ setting \eqref{rhp}, e.g., by replacing $m$ with
\begin{equation}  \label{tildem}
\tilde m(x,t,k)\coloneqq m(x,t,k)\nu_1^{\sigma_3}(k),
\end{equation}
so that the singularities (if any) of $\tilde m$ are all of order $\leq\frac{1}{4}$. This regularization aims to characterize $m(x,t,k)$
from the solution of a RH problem \eqref{rhp-0} whose data \eqref{basic-jump} are uniquely determined by $q_0(x)$ through the associated scattering functions $a(k)$ and $b(k)$.

In the generic case $m$ leads directly to a RH problem in the $L^2$ sense. However we must regularize $m$ to keep the $L^2$ setting \eqref{rhp-0} in the chain of subsequent deformations of the original RH problem. Lemmas \ref{rhp-deformation} and \ref{rhp-deformation-bis} indeed require that the transformation factors be bounded at $k=E_1$ and $k=\bar E_1$. To perform this regularization we replace $m$ with
\begin{equation}  \label{hat-m}
\hat m(x,t,k)\coloneqq m(x,t,k)\nu_1^{-\sigma_3}(k)
\end{equation}
and introduce related versions of $a$, $b$, and $r$:
\begin{equation}  \label{hat-abr}
\hat a\coloneqq a\nu_1,\quad\hat b\coloneqq b\nu_1,\quad\hat r\coloneqq\frac{\hat b^*}{\hat a}=r\nu_1^{-2}.
\end{equation}
Note that $\hat a\hat a^*=aa^*$, $\hat b\hat b^*=bb^*$, and $\hat r\hat r^*=rr^*=\abs{r}^2$. Note also that $\hat a$, $\hat b$, and $\frac{\hat r^*}{1+\hat r\hat r^*}=\hat a\hat b$ are bounded near $E_1$ whereas $\hat r$ is bounded near $E_2$.

From now on, for simplicity, we assume (see Assumption on $a(k)$ in Section~\ref{sec:thm}) that we are in the generic case, that is, that $a(k)$ actually has singularities at $E_1$ (which affects the construction of the RH problem) and $E_2$.

%---------------------------------------%
%:fig 1 = 3.1
%---------------------------------------%
\begin{figure}[ht]
\centering\includegraphics[scale=1]{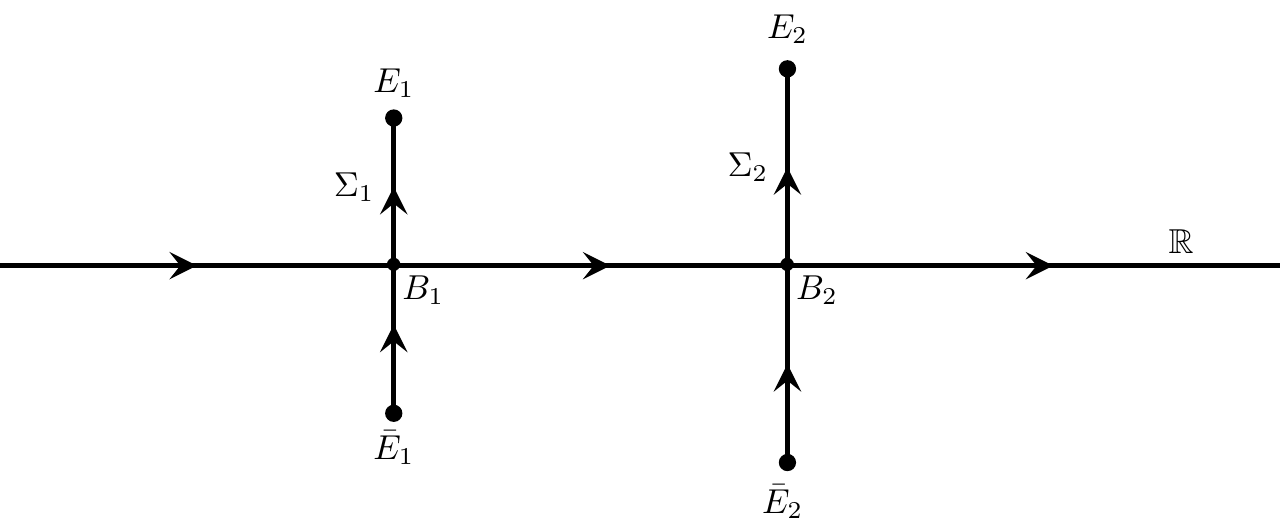}
\caption{The contour $\Sigma=\D{R}\cup\Sigma_1\cup\Sigma_2$ for the basic RH problem.} 
\label{fig:basic-contour}
\end{figure}
%---------------------------------------%
Let $\hat m$ be defined by \eqref{hat-m} and \eqref{m}. The RH problem satisfied by $\hat m$ derives from that satisfied by $m$, see \cite{BLS20a}*{Section 2.5}. Its contour is $\Sigma\coloneqq\D{R}\cup\Sigma_1\cup\Sigma_2$, see Figure~\ref{fig:basic-contour}. Its jump matrix is
\begin{subequations}  \label{basic-jump}
\begin{equation}  \label{jump}
\hat J(x,t,k)\coloneqq\eul^{-\ii t\theta(k)\sigma_3}\hat J_0(k)\eul^{\ii t\theta(k)\sigma_3},\quad k\in\Sigma,
\end{equation}
where the phase $\theta(k)$ is
\begin{equation}   \label{phase}
\theta(k)\equiv\theta(\xi,k)\coloneqq 2k^2+\xi k,\qquad\xi\coloneqq\frac{x}{t}\,.
\end{equation}
and 
\begin{equation}  \label{jump0}
\hat J_0(k)=\begin{cases}
\begin{pmatrix}1&\hat r^*\\
0&1\end{pmatrix}\begin{pmatrix}1&0\\
\hat r&1\end{pmatrix}=\begin{pmatrix}1+\hat r\hat r^*&\hat r^*\\
\hat r&1\end{pmatrix},&k\in\D{R},\\
\begin{pmatrix}-\ii&0\\\frac{\ii\eul^{-\ii\phi}}{\hat a_+\hat a_-}&\ii\end{pmatrix},&k\in\Sigma_1\cap\D{C}^+,\\
\begin{pmatrix}\frac{\hat a_-}{\hat a_+}&\ii\nu_1^2\\0&\frac{\hat a_+}{\hat a_-}\end{pmatrix},&k\in\Sigma_2\cap\D{C}^+,\\
\begin{pmatrix}-\ii&\frac{\ii\eul^{\ii\phi}}{\hat a_+^*\hat a_-^*}\\0&\ii\end{pmatrix},&k\in\Sigma_1\cap\D{C}^-,\\
\begin{pmatrix}\frac{\hat a_+^*}{\hat a_-^*}&0\\\ii\nu_1^{-2}&-\frac{\hat a_-^*}{\hat a_+^*}\end{pmatrix},&k\in\Sigma_2\cap\D{C}^-,
\end{cases}
\end{equation}
\end{subequations}
with $\hat a$ and $\hat r$ defined in \eqref{hat-abr}. Note that $\hat J_0=\nu_{1-}^{\sigma_3}J_0\nu_{1+}^{-\sigma_3}$ where $J_0$ is as in \cite{BLS20a}*{Section 2.5}.

%-------------------%
\begin{rh-pb*}    \label{basic-rhp-0}
Given $r(k)$ for $k\in\D{R}$, $a_+(k)$ and $a_-(k)$ for $k\in(\Sigma_1\cup\Sigma_2)\cap\D{C}^+$, find $\hat m(x,t,k)$ analytic in $k\in\D{C}\setminus\Sigma$ that satisfies 
\begin{equation}   \label{rhp-0}
\begin{cases}
\hat m\in I+\dot E^2(\D{C}\setminus\Sigma),&\\
\hat m_+(k)=\hat m_-(k)\hat J(k)&\text{for a.e. }k\in\Sigma,
\end{cases}
\end{equation}
where $\Sigma=\D{R}\cup\Sigma_1\cup\Sigma_2$ and $\hat J$ is defined by \eqref{basic-jump}.
\end{rh-pb*}
%-------------------%

Its unique solution $\hat m(x,t,k)$ provides a representation of the solution $q(x,t)$ of the Cauchy problem as follows:

%-------------------%
%:prop 3.1
%-------------------%
\begin{proposition}  \label{prop:basic-rhp}
Let $\hat m(x,t,k)$ be the solution of the basic RH problem. Then, the solution $q(x,t)$ of the Cauchy problem \eqref{nlsic}-\eqref{q0-limits} is given by
\begin{equation}  \label{m-to-q}
q(x,t)=2\ii\lim_{k\to\infty}k\hat m_{12}(x,t,k).
\end{equation}
\end{proposition}
%-------------------%

The jump matrix $\hat J$ satisfies
\begin{equation}  \label{J-symmetry}
\hat J(x,t,k)=
\begin{cases}
\sigma_3\sigma_1\overline{\hat J(x,t,\bar k)}\sigma_1\sigma_3,&k\in\Sigma_1\cup\Sigma_2,\\
\sigma_3\sigma_1\overline{\hat J(x,t,\bar k)}^{-1}\sigma_1\sigma_3,&k\in\D{R}.
\end{cases}
\end{equation}
Together with uniqueness of the solution of the RH problem, this implies the important symmetry
\begin{equation}  \label{m-symmetry}
\hat m(x,t,k)=\sigma_3\sigma_1\overline{\hat m(x,t,\bar k)}\sigma_1\sigma_3,\quad k\in\D{C}\setminus\Sigma.
\end{equation}

%---------------------------------------------------------%
%:s.3.2
%---------------------------------------------------------%
\subsection{The Riemann surface}  \label{sec:riemann-surface}

A genus $3$ sector is characterized by the fact that the leading order asymptotics of $q(x,t)$ can be expressed in terms of theta functions defined on a family of genus $3$ Riemann surfaces parametrized by $\xi\coloneqq x/t$. The genus $3$ structure arises because the derivative of the $g$-function has five zeros: one real zero, which we denote by $\mu\equiv\mu(\xi)$, and two complex conjugate pairs of zeros, which we denote by $\accol{\alpha,\bar\alpha}$ and $\accol{\beta,\bar\beta}$, where $\alpha\equiv\alpha(\xi)$ and $\beta\equiv\beta(\xi)$ lie in $\D{C}^+\setminus(\Sigma_1\cup\Sigma_2)$. Assuming for the moment that the level set $\Im g=0$ connects the points $E_1$, $\bar E_1$, $E_2$, $\bar E_2$, $\alpha$, $\bar\alpha$, $\beta$, $\bar\beta$, $\mu$, and $\infty$ as in Figure~\ref{fig:Img-contour-plot} (see also Figure~\ref{fig:jump-contour-1}), we next introduce the Riemann surface $M\equiv M(\xi)$ associated to $\xi$; the precise definition of the $g$-function will be given in the following subsection.
%-------------------%
%:fig 2 = 3.2
%-------------------%
\begin{figure}[ht]
\centering\includegraphics[scale=.75]{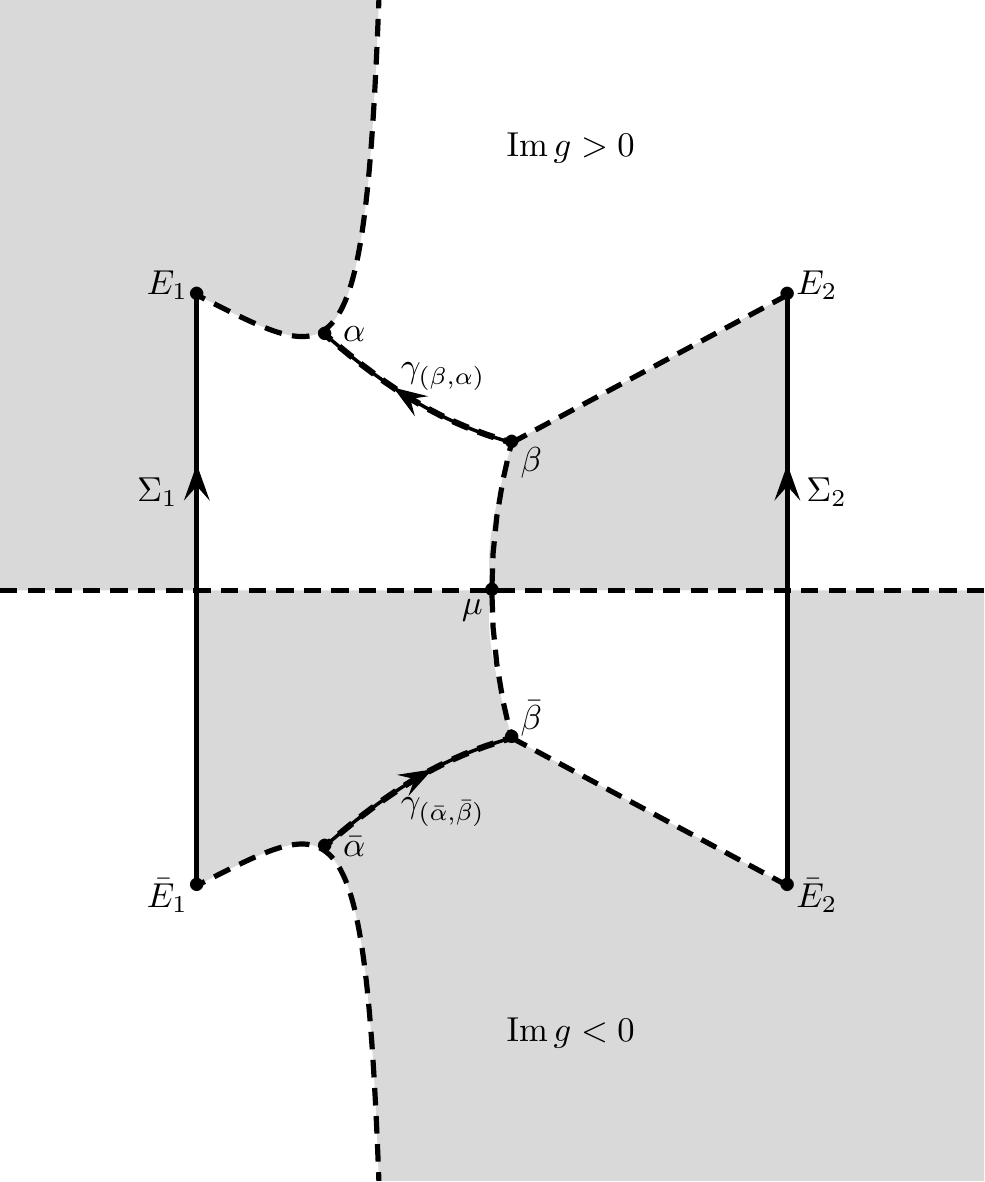}
\caption{The four branch cuts (solid black) and the level set $\Im g=0$ in the genus $3$ sector in the case of $A_1=A_2=3/2$ and $\xi=2$. The region where $\Im g<0$ is shaded and the region where $\Im g>0$ is white.} 
\label{fig:Img-contour-plot}
\end{figure}
%-------------------%

Let $\gamma_{(\beta,\alpha)}$ and $\gamma_{(\bar\alpha,\bar\beta)}$ denote the oriented contours from $\beta$ to $\alpha$ and from $\bar\alpha$ to $\bar\beta$, respectively, on which $\Im g=0$, see Figure~\ref{fig:Img-contour-plot}. Similarly, let $\gamma_{(\bar\beta,\beta)}$ denote the oriented contour from $\bar\beta$ to $\beta$ on which $\Im g=0$.

We define the hyperelliptic surface $M\equiv M_{\alpha\beta}$ as the set of all points $P\coloneqq(w,k)\in\D{C}^2$ such that
\begin{equation}  \label{M-def}
w^2=(k-E_1)(k-\bar E_1)(k-E_2)(k-\bar E_2)(k-\alpha)(k-\bar\alpha)(k-\beta)(k-\bar\beta)
\end{equation}
together with two points $\infty^+$ and $\infty^-$ at infinity which make the surface compact. We view $M$ as a two-sheeted cover of the complex plane by introducing a set of branch cuts $\C{C}$ which connect the eight branch points. For $k\in\hat{\D{C}}\setminus\C{C}$, we let $k^+$ and $k^-$ denote the corresponding points on the upper and lower sheet of $M$, respectively. By definition, the upper (resp.\ lower) sheet is characterized by $w=k^4+\ord(k^3)$ (resp.\ $w=-k^4+\ord(k^3)$) as $k\to\infty$. For two complex numbers $z_1$ and $z_2$, we let $\croch{z_1,z_2}^+$ and $\croch{z_1,z_2}^-$ denote the covers of $\croch{z_1,z_2}$ in the upper and lower sheets, respectively.
%-------------------%
%:fig 3 = 3.3
%-------------------%
\begin{figure}[ht]
\centering\includegraphics[scale=.8]{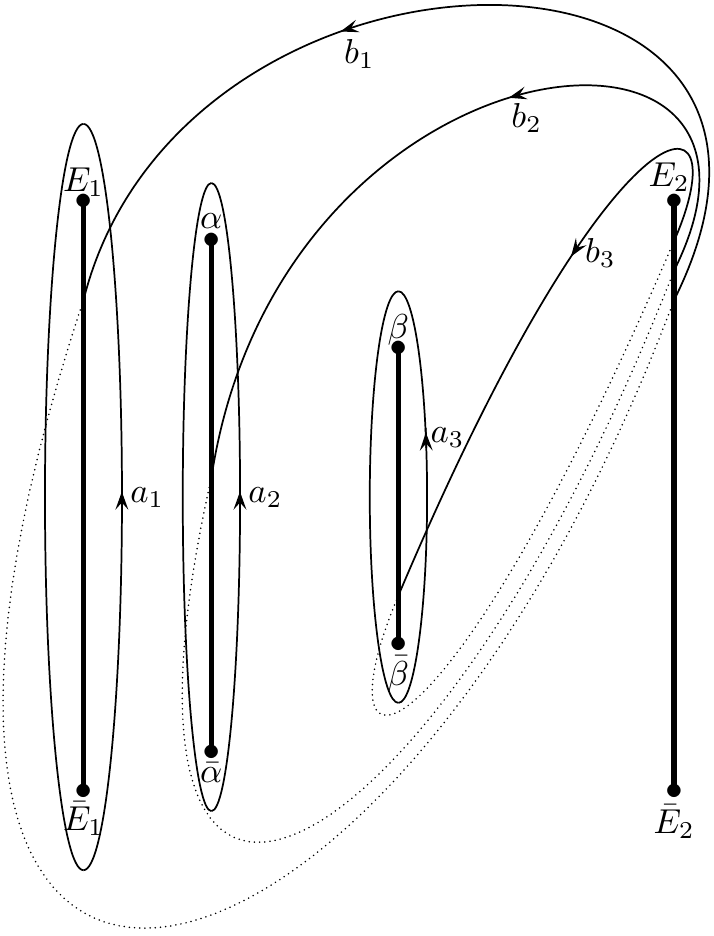}
\caption{A possible, but inappropriate, choice of branch cuts and canonical homology basis for the genus three surface $M$.} 
\label{fig:homology-basis}
\end{figure}
%-------------------%

We next discuss the choice of branch cuts and the choice of a canonical homology basis $\accol{a_j,b_j}_1^3$ for $M$. One possibility is to let all branch cuts be vertical segments and adopt the standard choice of canonical homology basis $\accol{a_j,b_j}_1^3$ displayed in Figure~\ref{fig:homology-basis}. However, for the purposes of the nonlinear steepest descent method, the cuts connecting the branch points $\alpha$ and $\beta$ and their complex conjugates must lie in the level set where the imaginary part of the $g$-function vanishes. Thus, we deform the vertical cut from $\bar\beta$ to $\beta$ into the curve segment $\gamma_{(\bar\beta,\beta)}$ which passes through $\mu$. Similarly, we deform the vertical cut connecting $\bar\alpha$ to $\alpha$ into the curve segment $\gamma_{(\bar\alpha,\alpha)}\coloneqq\gamma_{(\bar\alpha,\bar\beta)}\cup\gamma_{(\bar\beta,\beta)}\cup\gamma_{(\beta,\alpha)}$. Since both cuts then proceed along the curve segment $\gamma_{(\bar\beta,\beta)}$, this segment ceases to be a cut and the homology basis of Figure~\ref{fig:homology-basis} turns into the canonical homology basis $\accol{a_j,b_j}_1^3$ displayed in Figure~\ref{fig:canonical-basis}. We will henceforth adopt the four branch cuts in the complex $k$-plane shown in Figure~\ref{fig:canonical-basis} and let $\C{C}\subset\D{C}$ denote the union of these cuts, i.e.,
\begin{equation}\label{cuts}
\C{C}\coloneqq\Sigma_1\cup\Sigma_2\cup\gamma_{(\beta,\alpha)}\cup\gamma_{(\bar\alpha,\bar\beta)}.
\end{equation}
%-------------------%
%:fig 4 = 3.4
%-------------------%
\begin{figure}[ht]
\centering\includegraphics[scale=.8]{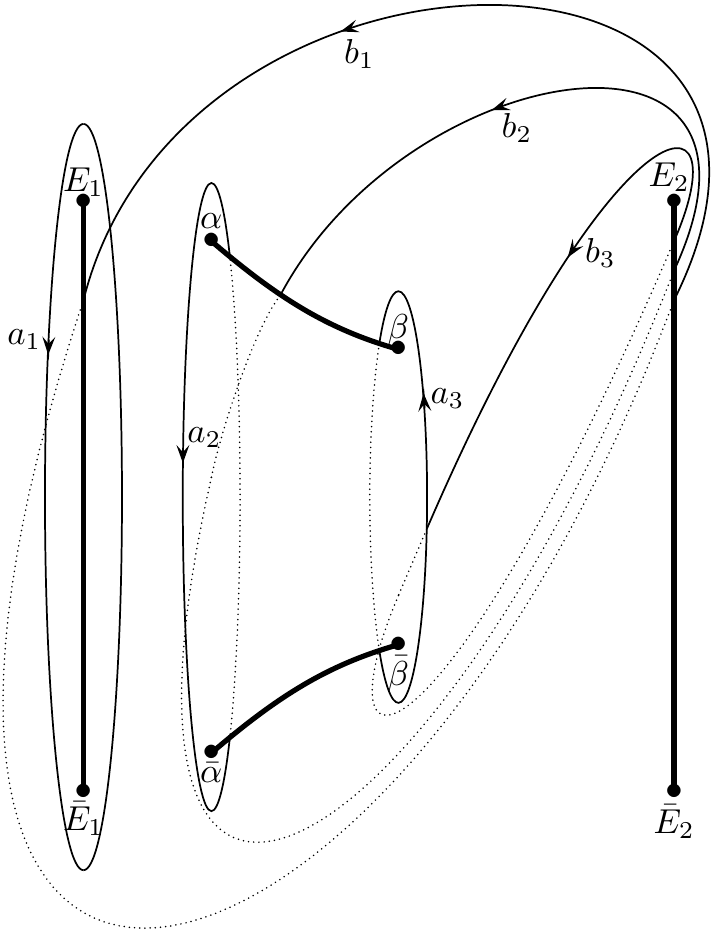}
\caption{The canonical homology basis $\accol{a_j,b_j}_1^3$ for the genus three surface $M$.} 
\label{fig:canonical-basis}
\end{figure}
%-------------------%

We will also adopt the homology basis for $M$ shown in Figure~\ref{fig:canonical-basis}, but with the following important specification: Unless stated otherwise, all integrals along paths on $M$ for which only the endpoints are specified are assumed to lie within the fundamental polygon obtained by cutting the Riemann surface along the basis $\accol{a_j,b_j}_1^3$. This implies that some integrals will depend on the particular choices of the $a_j$ and $b_j$ within their respective homology classes. We therefore henceforth fix the curves $a_j$ and $b_j$ so that they are invariant under the involution $\varsigma\colon k^\pm\to k^\mp$, more precisely $\varsigma(a_j)=-a_j$ and $\varsigma(b_j)=-b_j$. For the present case, this is accomplished by letting $a_1$, $a_2$, $a_3$ be the paths in the homology classes specified by Figure~\ref{fig:canonical-basis} which as point sets consist of the points of $M$ lying directly above $\Sigma_1$, $\gamma_{(\bar\alpha,\alpha)}$, and $\gamma_{(\bar\beta,\beta)}$, respectively, and by letting $b_3$, $b_2$, $b_1$ be the paths in the homology classes specified by Figure~\ref{fig:canonical-basis} which as point sets consist of the points of $M$ lying directly above
\[
\croch{E_2,\bar\beta},\quad\croch{E_2,\bar\beta}\cup\gamma_{(\bar\alpha,\bar\beta)}\cup\gamma_{(\bar\beta,\beta)},\text{ and }\croch{E_2,\bar\beta}\cup\gamma_{(\bar\alpha,\bar\beta)}\cup\gamma_{(\bar\beta,\beta)}\cup\croch{\alpha,\bar E_1},
\]
respectively, see Figure~\ref{fig:curves-aj-bj}. 

%-------------------%
%:fig 5 = 3.5
%-------------------%
\begin{figure}[ht]
\centering\includegraphics[scale=1]{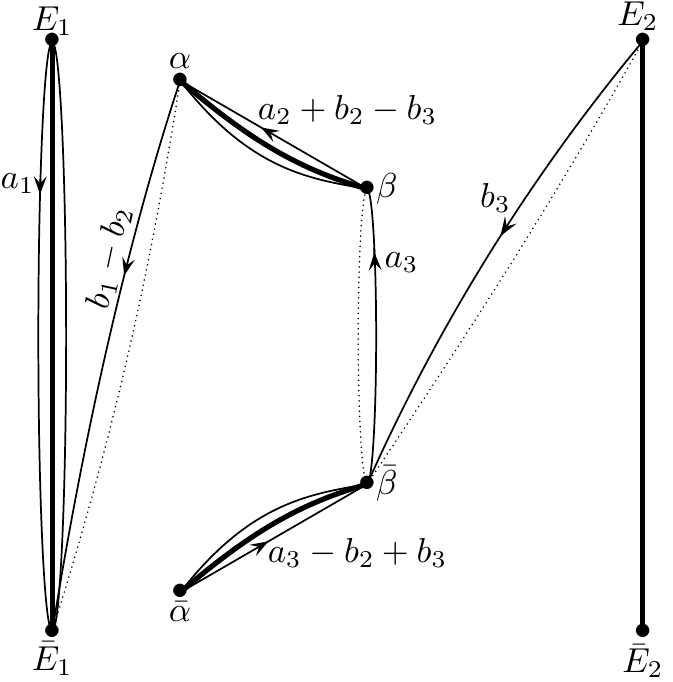}
\caption{The particular choices of the curves $a_j$ and $b_j$ within their homology classes.} 
\label{fig:curves-aj-bj}
\end{figure}
%-------------------%

For $k\in\hat{\D{C}}\setminus\C{C}$, we let $w(k)\equiv w(\xi,k)$ denote the value of $w$ corresponding to the point $k^+\in M$, that is,
\[
w(k)=\sqrt{(k-E_1)(k-\bar E_1)(k-E_2)(k-\bar E_2)(k-\alpha)(k-\bar\alpha)(k-\beta)(k-\bar\beta)},\quad k\in\hat{\D{C}}\setminus\C{C},
\]
where the branch of the square root is such that $w(k)>0$ for $k\in\D{R}$, $k\gg 0$. We let $\zeta=(\zeta_1,\zeta_2,\zeta_3)$ be the normalized basis of $\C{H}^1(M)$ dual to the canonical homology basis $\accol{a_j,b_j}_1^3$, in the sense that $\accol{\zeta_j}_1^3$ are holomorphic differentials such that
\[
\int_{a_i}\zeta_j=\delta_{ij}.
\]
The basis $\accol{\zeta_j}_1^3$ is explicitly given by $\zeta_j=\sum_{l=1}^3A_{jl}\hat\zeta_l$, where $\hat\zeta_l\coloneqq\frac{k^{l-1}}{w}\,\dd k$ and $\left(A^{-1}\right)_{jl}=\int_{a_j}\hat\zeta_l$. Since $w$ satisfies $w(k^+)=-w(k^-)$, we have the same symmetry for $\zeta$:
\begin{equation}  \label{zeta-symmetry}
\zeta(k^+)=-\zeta(k^-).
\end{equation}
We let $\tau=\begin{pmatrix}\tau_{jl}\end{pmatrix}$ denote the $3\times 3$ period matrix, defined by
\[
\tau_{jl}\coloneqq\int_{b_j}\zeta_l.
\]
It is a standard result that $\tau$ is symmetric and has positive definite imaginary part, see \cite{FK92}*{Proposition III.2.8}. The symmetries $\varsigma(a_j)=-a_j$, $\varsigma(b_j)=-b_j$, together with \eqref{zeta-symmetry} imply
\[
\int_{a_j^+}\zeta=\frac{1}{2}\int_{a_j}\zeta=\frac{1}{2}\vece^{(j)}\quad\text{and}\quad\int_{b_j^+}\zeta=\frac{1}{2}\int_{b_j}\zeta=\frac{1}{2}\tau^{(j)}, 
\]
where $a_j^+$, $b_j^+$ denote the restrictions of $a_j$, $b_j$ to the upper sheet, $\vece^{(j)}$ denotes the $j$th column of the $3\times 3$ identity matrix $I$, and $\tau^{(j)}$ denotes the $j$th column of $\tau=(\tau_{jl})$.  The theta function $\Theta$ associated to $\tau$ is defined by
\begin{equation}  \label{Theta-def}
\Theta(z)\coloneqq\sum_{N\in\D{Z}^3}\eul^{2\pi\ii(\frac{1}{2}N^T\tau N+N^Tz)},\quad z\in\D{C}^3.
\end{equation}
We have
\begin{equation}  \label{Theta-relat}
\Theta(z+\vece^{(j)})=\Theta(z),\quad\Theta(z+\tau^{(j)})=\eul^{2\pi\ii(-z_j-\frac{\tau_{jj}}{2})}\Theta(z),\quad j=1,2,3,\ \ z\in\D{C}^3.
\end{equation}
We also have $\Theta(z)=\Theta(-z)$ for all $z\in\D{C}^3$. 

We let $\varphi\colon M\to\D{C}^3$ denote the Abel-type map with base point $\bar E_2$, i.e.,  
\begin{equation}  \label{Abel-map}
\varphi(P)=\int_{\bar E_2}^P\zeta,\quad P\in M,
\end{equation}
where the contour is chosen to lie within the fundamental polygon specified by $\accol{a_j,b_j}$. We see from equation \eqref{zeta-symmetry} and Figures~\ref{fig:canonical-basis} and \ref{fig:curves-aj-bj} that
\begin{equation}  \label{varphi-jump-1}
\varphi_+(k^+)+\varphi_-(k^+)=-\varphi_+(k^-)-\varphi_-(k^-)=\begin{cases}
\tau^{(1)},&k\in\Sigma_1,\\
\tau^{(2)},&k\in\gamma_{(\beta,\alpha)},\\
\tau^{(2)}+\vece^{(1)}+\vece^{(2)},&k\in\gamma_{(\bar\alpha,\bar\beta)},\\
0,&k\in\Sigma_2,
\end{cases}
\end{equation}
and
\begin{equation}  \label{varphi-jump-2}
\varphi_+(k^+)-\varphi_-(k^+)=\tau^{(2)}-\tau^{(3)}+\vece^{(1)}+\vece^{(2)},\quad k\in\gamma_{(\bar\beta,\beta)},
\end{equation}
where $\varphi_+(k^\pm)$ and $\varphi_-(k^\pm)$ denote the boundary values of $\varphi(l^\pm)$ as $l\in\D{C}$ approaches $k$ from the left and right sides of the contour, respectively.
    
%---------------------------------------------------------%
%:s.3.3
%---------------------------------------------------------%
\subsection{The $\BS{g}$-function}   \label{sec:g-function}

We let $\dd g(k)\equiv\dd g(\xi,k)$ denote the differential on $\D{C}\setminus\C{C}$ defined by
\begin{equation}  \label{dg}
\dd g\coloneqq 4\frac{(k-\mu)(k-\alpha)(k-\bar\alpha)(k-\beta)(k-\bar\beta)}{w(k)}\dd k,\quad k\in\D{C}\setminus\C{C},
\end{equation}
where $\mu\equiv\mu(\xi)$ is a real number. Let $\widehat{\dd g}$ denote the differential on $M$ given by $\dd g$ on the upper sheet and by $-\dd g$ on the lower sheet, i.e., $\widehat{\dd g}(k^\pm)=\pm\dd g(k)$ for $k\in\D{C}\setminus\C{C}$. Then $\widehat{\dd g}$ is a meromorphic differential on $M$ which is holomorphic everywhere on $M$ except for two poles at $\infty^\pm$.
  
The definition of $\dd g$ depends on the five real numbers $\mu$, $\alpha_1$, $\alpha_2$, $\beta_1$, $\beta_2$, where $\alpha_1$, $\beta_1$ and $\alpha_2$, $\beta_2$ denote the real and imaginary parts of $\alpha$ and $\beta$:
\[
\alpha=\alpha_1+\ii\alpha_2,\quad\beta=\beta_1+\ii\beta_2.
\]
These five real numbers are determined by the five conditions 
\begin{subequations}  \label{dg-conditions}
\begin{align}  \label{dg-first-conditions}
&\int_{a_1}\widehat{\dd g}=\int_{a_2}\widehat{\dd g}=\int_{a_3}\widehat{\dd g}=0,\\
\label{dg-last-conditions}
&\lim_{k\to\infty}\left(\frac{\dd g}{\dd k}-4k\right)=\xi,\quad       
\lim_{k\to\infty}k\left(\frac{\dd g}{\dd k}-4k-\xi\right)=0.
\end{align}
\end{subequations}
The solvability of this system of equations characterizes a genus $3$ sector. Since
\begin{equation}  \label{dg-symmetry}
\frac{\dd g}{\dd k}(k)=\overline{\frac{\dd g}{\dd k}(\bar k)},\quad k\in\D{C}\setminus\C{C},
\end{equation}
we have $\int_A^B\dd g=\overline{\int_{\bar A}^{\bar B}\dd g}$ where the contour in the second integral is the complex conjugate of the contour in the first integral. This implies that
\[
\int_{a_j}\widehat{\dd g}\in\ii\,\D{R},\quad j=1,2,3,
\]
so the conditions in \eqref{dg-first-conditions} are three real conditions. 

Let $\Sigma^{\model}$ denote the union of the four branch cuts and the curve $\gamma_{(\bar\beta,\beta)}$, oriented as in Figure~\ref{fig:jump-contour-model}:
\[
\Sigma^{\model}\coloneqq\C{C}\cup\gamma_{(\bar\beta,\beta)}=\Sigma_1\cup\Sigma_2\cup\gamma_{(\beta,\alpha)}\cup\gamma_{(\bar\alpha,\bar\beta)}\cup\gamma_{(\bar\beta,\beta)}.
\]
The $g$-function $g(k)\equiv g(\xi,k)$ is defined by
\begin{equation}  \label{g-def}
g(k)\coloneqq\int_{\bar E_2}^k\dd g,\quad k\in\D{C}\setminus\Sigma^{\model},
\end{equation}
where the contour is required to lie in $\D{C}\setminus\Sigma^{\model}$. For convenience, we choose the same base point for the integral in \eqref{g-def} as for the integral in \eqref{Abel-map} defining $\varphi$.

The two conditions in \eqref{dg-last-conditions} ensure that
\[
g'(k)=4k+\xi+\ord(k^{-2}),\quad k\to\infty,
\]
and hence
\begin{equation}  \label{g-asymptotics}
g(k)=\theta(k)+g^{(0)}+\ord(k^{-1}),\quad k\to\infty,
\end{equation}
uniformly with respect to $\arg k\in\croch{0,2\pi}$, where $g^{(0)}\equiv g^{(0)}(\xi)\in\D{R}$ is a finite real number independent of $k$. Since the $\ord(k^{-1})$ term vanishes in the expansion of $\dd g/\dd k$, the meromorphic form $\widehat{\dd g}$ has no residue at $\infty$. Hence we can deform contours freely through infinity. By deforming the contour through $\infty$, it follows from \eqref{dg-first-conditions} that the loop integral of $\widehat{\dd g}$ around the cut $\croch{\bar E_2,E_2}$ also vanishes. Thus the conditions in \eqref{dg-conditions} imply that the value of $g$ is independent of the choice of contour in $\D{C}\setminus\Sigma^{\model}$.

Let $g_+$ and $g_-$ denote the boundary values of $g$ from the left and right sides of the different pieces of $\Sigma^{\model}$.

%-------------------%
%:lem 3.2
%-------------------%
\begin{lemma}   \label{g-properties}
The function $g(k)$ defined in \eqref{g-def} has the following properties:
\begin{enumerate}[\rm(a)]
\item
$g(k)-2k^2-\xi k$ is an analytic and bounded function of $k\in\hat{\D{C}}\setminus\Sigma^{\model}$ with continuous boundary values on $\Sigma^{\model}$.
\item
$g(k)$ obeys the symmetry
\begin{equation}  \label{g-symmetry}
g=g^*,\quad k\in\D{C}\setminus\Sigma^{\model}.
\end{equation}
\item
$g(k)$ satisfies the following jump conditions across $\Sigma^{\model}$:
\begin{subequations}  \label{g-jump}
\begin{equation}\label{g-jump-a}
g_+(k)+g_-(k)=
\begin{cases}
2\Omega_1,&k\in\Sigma_1,\\
2\Omega_2,&k\in\gamma_{(\bar\alpha,\bar\beta)}\cup\gamma_{(\beta,\alpha)},\\
0,&k\in\Sigma_2,
\end{cases}
\end{equation}
and
\begin{equation}
g_+(k)-g_-(k)=2\Omega_3,\quad k\in\gamma_{(\bar\beta,\beta)},
\end{equation}
\end{subequations}
where the three real constants $\Omega_j\equiv\Omega_j(\xi)$, $j= 1,2,3$, are given by
\begin{equation}   \label{Omegaj-def}
\begin{split}
\Omega_1&=g(E_1)=g(\bar E_1),\quad\Omega_2=g(\alpha)=g(\bar\alpha),\\
\Omega_3&=\frac{g_+(\beta)-g_-(\beta)}{2}=\frac{g_+(\bar\beta)-g_-(\bar\beta)}{2}.
\end{split}
\end{equation}
\end{enumerate}
\end{lemma}
%-------------------%

%-------------------%
\begin{proof}
It is clear that $g(k)$ is an analytic function of $k\in\D{C}\setminus\Sigma^{\model}$ with bounded and continuous boundary values on $\Sigma^{\model}$. Since equation \eqref{g-asymptotics} implies that $g(k)-2k^2-\xi k$ is analytic also at infinity, part (a) follows.

Using that
\[
\left(\frac{\dd g}{\dd k}\right)_+=\left(\frac{\dd g}{\dd k}\right)_-\!\times
\begin{cases}
-1,&k\in\C{C},\\
1,&k\in\gamma_{(\bar\beta,\beta)},
\end{cases}
\]
we find that $g_++g_-$ is constant on each of the curves in \eqref{g-jump-a} and that $g_+-g_-$ is constant on $\gamma_{(\bar\beta,\beta)}$. These constant values are fixed by evaluation at the endpoints. This proves \eqref{g-jump}. In particular, $g(E_2)=g(\bar E_2)=0$, and, using \eqref{dg-symmetry}, we conclude that (b) follows.

The three conditions in \eqref{dg-first-conditions} together with \eqref{dg-symmetry} ensure that $\Im g=0$ on $\D{R}$ and that $\Im g(E_1)=\Im g(\alpha)=\Im g_+(\beta)=\Im g_-(\beta)=0$. Hence the constants $\Omega_j$ defined in \eqref{Omegaj-def} are real. This completes the proof of (c).
\end{proof} 
%-------------------%
        
%---------------------------------------------------------%
%:s.3.4
%---------------------------------------------------------%
\subsection{Contour deformations}

On several occasions we will need to transform RH problems by deforming the contour. In this subsection, we give assumptions under which these deformations are allowed.

%-------------------%
%:fig 6 = 3.6
%-------------------%
\begin{figure}[ht]
\centering\includegraphics[scale=.95]{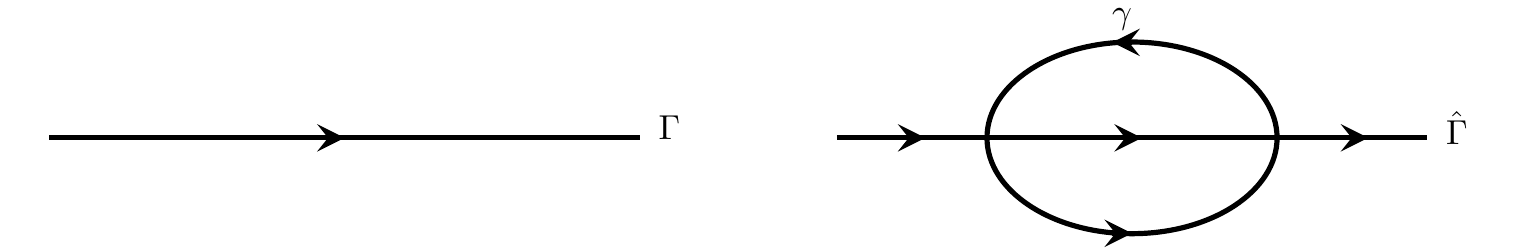}
\caption{The contour $\Gamma$ (left) and the contour $\hat\Gamma=\Gamma\cup\gamma$ (right).} 
\label{fig:Gamma-hatGamma}
\end{figure}
%-------------------%

Let $\Gamma$ be a piecewise smooth contour and let $\hat\Gamma= \Gamma\cup\gamma$ denote $\Gamma$ with a lens added as in Figure~\ref{fig:Gamma-hatGamma}. Consider the $2\times 2$-matrix RH problem
\begin{equation}   \label{rhp-lens}
\begin{cases}
m\in I+\dot E^2(\D{C}\setminus\Gamma),&\\
m_+(k)=m_-(k)v(k)&\text{for a.e. }k\in\Gamma,
\end{cases}
\end{equation}
where $v\colon\Gamma\to\GL(2,\D{C})$ is a jump matrix defined on $\Gamma$.

%-------------------%
%:lem 3.3
%-------------------%
\begin{lemma}[contour deformation]  \label{rhp-deformation}
Let $u$ be a $2\times 2$-matrix valued function such that $u,u^{-1}\in E^{\infty}(\D{C}\setminus\hat\Gamma)$. Then $m$ satisfies the RH problem \eqref{rhp-lens} if and only if the function $\hat m$ defined by
\[
\hat m(k)\coloneqq m(k)u(k),\quad k\in\D{C}\setminus\hat\Gamma,
\]
satisfies the RH problem
\begin{equation}  \label{rhp-deform}
\begin{cases}
\hat m\in I+\dot E^2(\D{C}\setminus\hat\Gamma),&\\
\hat m_+(k)=\hat m_-(k)\hat v(k)&\text{for a.e. }k\in\hat\Gamma,
\end{cases}
\end{equation}
with
\begin{equation}  \label{jump-deformed}
\hat v=
\begin{cases}
u_-^{-1}vu_+,&k\in\Gamma,\\
u_-^{-1}u_+,&k\in\gamma,
\end{cases}
\end{equation}
If $\gamma$ passes through infinity as in Figure~\ref{fig:hatGamma-infty}, then the same conclusion holds provided that $u,u^{-1}\in I+(\dot E^2\cap E^\infty)(\D{C}\setminus\hat\Gamma)$.
\end{lemma}
%-------------------%

%-------------------%
%:fig 7 = 3.7
%-------------------%
\begin{figure}[ht]
\centering\includegraphics[scale=1]{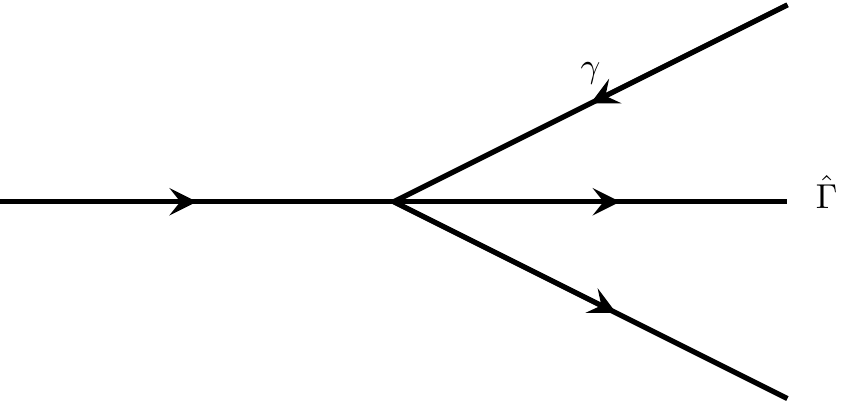}
\caption{The contour $\hat\Gamma=\Gamma\cup\gamma$ in the case when $\gamma$ passes through infinity.} 
\label{fig:hatGamma-infty}
\end{figure}
%-------------------%

We will also come across situations where $\gamma$ passes through infinity, but $u$ does not approach the identity matrix as $k\to\infty$, see Figure~\ref{fig:hatGamma-infty}. In this case, the following result is useful.

%-------------------%
%:lem 3.4
%-------------------%
\begin{lemma}[contour deformation when $u(\infty)\neq I$] \label{rhp-deformation-bis}
Let $u$ be a $2\times 2$-matrix valued function such that $u,u^{-1}\in E^{\infty}(\D{C}\setminus\hat\Gamma)$ and $u,u^{-1}$ are analytic at $k=\infty$. Then $m$ satisfies the RH problem \eqref{rhp-lens} if and only if the function $\hat m$ defined by
\[
\hat m(k)\coloneqq u(\infty)^{-1}m(k)u(k),\quad k\in\D{C}\setminus\hat\Gamma,
\]
satisfies the RH problem~\eqref{rhp-deform} with $\hat v$ given by \eqref{jump-deformed}.
\end{lemma}
%-------------------%

%---------------------------------------------------------%
%:s.4
%---------------------------------------------------------%
\section{Transformations of the RH problem}  \label{sec:transfos}

In order to determine the long-time asymptotics of the solution $\hat m$ of the RH problem \eqref{rhp-0}, we will perform a series of transformations of the RH problem. More precisely, starting with $\hat m$, we will define functions $\hat m^{(j)}(x,t,k)$, $j=1,\dots,5$, such that each $\hat m^{(j)}$ satisfies an RH problem which is equivalent to the original RH problem \eqref{rhp-0}. The RH problem for $\hat m^{(j)}$ has the form
\begin{equation}  \label{rhp-j}
\begin{cases}
\hat m^{(j)}(x,t,\,\cdot\,)\in I+\dot E^2(\D{C}\setminus\Sigma^{(j)}),&\\
\hat m_+^{(j)}(x,t,k)=\hat m_-^{(j)}(x,t,k)\hat v^{(j)}(x,t,k)&\text{for a.e. }k\in\Sigma^{(j)},
\end{cases}
\end{equation}
where the contours $\Sigma^{(j)}$ and jump matrices $\hat v^{(j)}$ are specified below.

The jump matrix $\hat v^{(5)}$ obtained after the fifth transformation has the property that it is a constant off-diagonal matrix on each of the four branch cuts and a constant diagonal matrix on $\gamma_{(\bar\beta,\beta)}$. Moreover, on all other parts of the contour, it approaches the identity matrix as $t\to\infty$. The leading order asymptotics as $t\to\infty$ can therefore be obtained by solving the model RH problem with jump across $\Sigma^{\model}$
obtained by replacing $\hat v^{(5)}$ with its large $t$ limit.

The symmetries \eqref{J-symmetry} and \eqref{m-symmetry} will be preserved at each stage of the transformations, so that
\begin{equation}  \label{vj-sym}
\hat v^{(j)}(x,t,k)=\begin{cases}
\sigma_3\sigma_1\overline{\hat v^{(j)}(x,t,\bar k)}\sigma_1\sigma_3,&k\in\Sigma^{(j)}\setminus\D{R},\\
\sigma_3\sigma_1\overline{\hat v^{(j)}(x,t,\bar k)}^{-1}\sigma_1\sigma_3,&k\in\Sigma^{(j)}\cap\D{R},
\end{cases}
\quad j=1,\dots,5,
\end{equation}
and
\begin{equation}   \label{mj-sym}
\hat m^{(j)}(x,t,k)=\sigma_3\sigma_1\overline{\hat m^{(j)}(x,t,\bar k)}\sigma_1\sigma_3,\quad k\in\D{C}\setminus\Sigma^{(j)},\quad j=1,\dots,5.
\end{equation}

%---------------------------------------------------------%
%:s.4.1
%---------------------------------------------------------%
\subsection{First transformation}

Define $\hat m^{(1)}$ by
\[
\hat m^{(1)}(x,t,k)\coloneqq\eul^{-\ii tg^{(0)}\sigma_3}\hat m(x,t,k)\eul^{\ii t(g(k)-\theta(k))\sigma_3},
\]
where $g^{(0)}\equiv g^{(0)}(\xi)=(g-\theta)(\xi,\infty)$ is defined in \eqref{g-asymptotics}. Part (a) of Lemma~\ref{g-properties} implies that
\[
\eul^{\ii t(g-\theta)\sigma_3}\in E^{\infty}(\hat{\D{C}}\setminus\Sigma^{\model}).
\]
Hence, by Lemma~\ref{rhp-deformation-bis}, $\hat m$ satisfies the RH problem \eqref{rhp-0} iff $\hat m^{(1)}$ satisfies the RH problem \eqref{rhp-j} with $j=1$, where the contour $\Sigma^{(1)}$ is defined by (see Figure~\ref{fig:jump-contour-1})
\[
\Sigma^{(1)}\coloneqq\D{R}\cup\Sigma^{\model}=\D{R}\cup\Sigma_1\cup\Sigma_2\cup\gamma_{(\beta,\alpha)}\cup\gamma_{(\bar\alpha,\bar\beta)}\cup\gamma_{(\bar\beta,\beta)},
\]
and the jump matrix $\hat v^{(1)}$ is given by
\[
\hat v^{(1)}(x,t,k)\coloneqq\eul^{-\ii tg_-(k)\sigma_3}\hat J_0(k)\eul^{\ii tg_+(k)\sigma_3},
\]
with $\hat J_0(k)$ as in Section~\ref{sec:rhp-0} for $k\in\D{R}\cup\Sigma_1\cup\Sigma_2$ and $\hat J_0(k)=I$ elsewhere, that is,
\[
\hat v^{(1)}=
\begin{cases}
\begin{pmatrix}
1&\hat r^*\eul^{-2\ii tg}\\
0&1\end{pmatrix}
\begin{pmatrix}
1&0\\
\hat r\eul^{2\ii tg}&1\end{pmatrix},&k\in\D{R},\\
\begin{pmatrix} 
-\ii\eul^{\ii t(g_+-g_-)}&0\\
\frac{\ii\eul^{-\ii\phi}}{\hat a_+\hat a_-}\eul^{2\ii t\Omega_1}&\ii\eul^{-\ii t(g_+-g_-)}\end{pmatrix},&k\in\Sigma_1\cap\D{C}^+,\\[1mm]
\begin{pmatrix} 
\frac{\hat a_-}{\hat a_+}\eul^{\ii t(g_+-g_-)}&\ii\nu_1^2\\
0&\frac{\hat a_+}{\hat a_-}\eul^{-\ii t(g_+-g_-)}\end{pmatrix},&k\in\Sigma_2\cap\D{C}^+,\\[1mm]
\begin{pmatrix} 
-\ii\eul^{\ii t(g_+-g_-)}&\frac{\ii\eul^{\ii\phi}}{\hat a_+^*\hat a_-^*}\eul^{-2\ii t\Omega_1}\\
0&\ii\eul^{-\ii t(g_+-g_-)}\end{pmatrix},&k\in\Sigma_1\cap\D{C}^-,\\[1mm]
\begin{pmatrix} 
\frac{\hat a_+^*}{\hat a_-^*}\eul^{\ii t(g_+-g_-)}&0\\
\ii\nu_1^{-2}&\frac{\hat a_-^*}{\hat a_+^*}\eul^{-\ii t(g_+-g_-)}\end{pmatrix},&k\in\Sigma_2\cap\D{C}^-,\\
\begin{pmatrix}
\eul^{\ii t(g_+-g_-)}&0\\
0&\eul^{-\ii t(g_+-g_-)}\end{pmatrix},&k\in\gamma_{(\beta,\alpha)}\cup\gamma_{(\bar\alpha,\bar\beta)}\cup\gamma_{(\bar\beta,\beta)}.
\end{cases}\]
%-------------------%
%:fig 8 = 4.1
%-------------------%
\begin{figure}[t]
\centering\includegraphics[scale=.7]{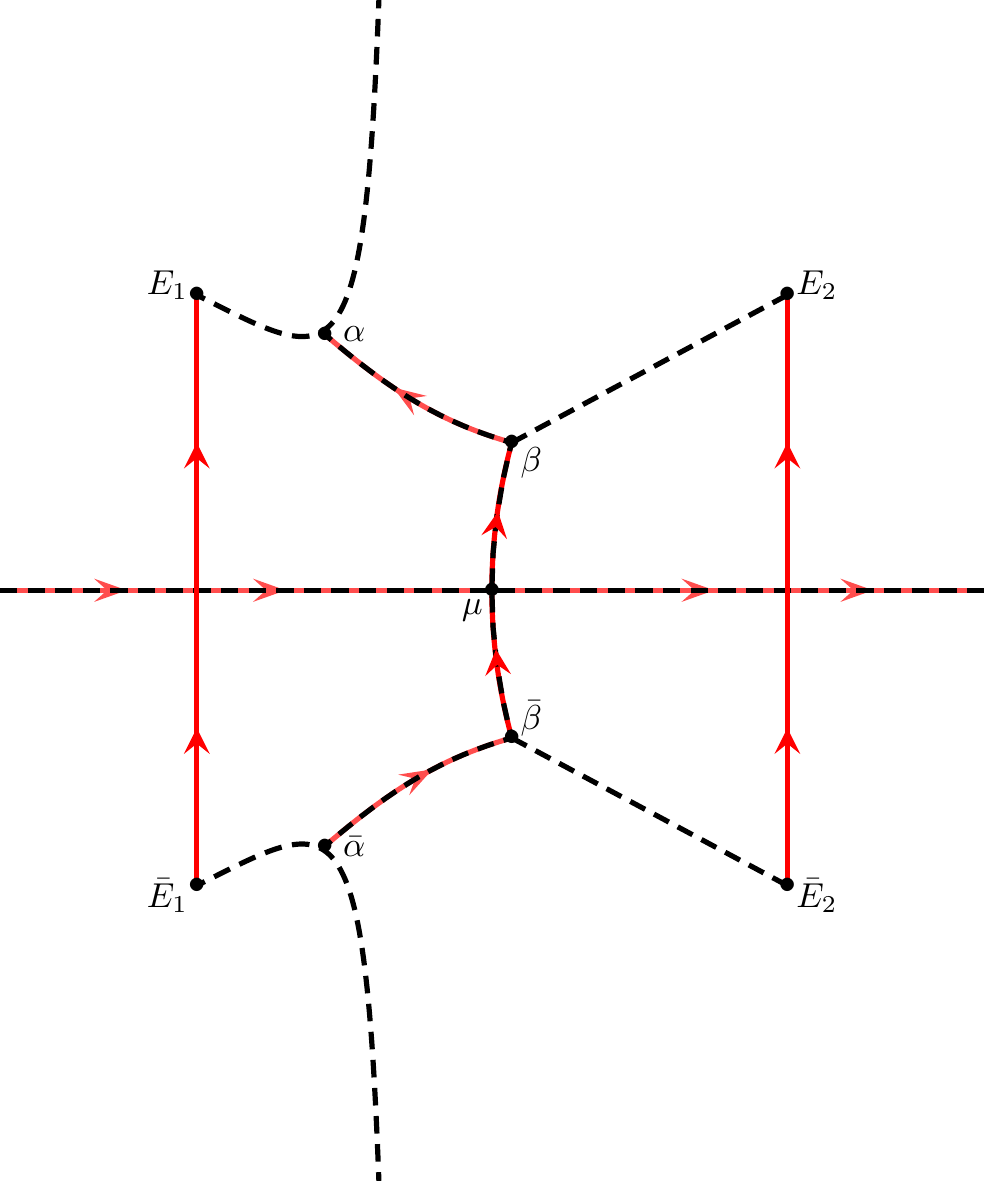}
\caption{The jump contour $\Sigma^{(1)}$ (red with arrows indicating orientation) together with the set where $\Im g=0$ (dashed black).} 
\label{fig:jump-contour-1}
\end{figure}
%-------------------%

%---------------------------------------------------------%
%:s.4.2
%---------------------------------------------------------%
\subsection{Second transformation}

The jump matrix $\hat v^{(1)}$ has the wrong factorization for $k\in(-\infty,\mu)$. Hence we introduce $\hat m^{(2)}$ by
\[
\hat m^{(2)}(x,t,k)\coloneqq\hat m^{(1)}(x,t,k)\delta(k)^{-\sigma_3},
\]
where the complex-valued function $\delta(k)\equiv\delta(\xi,k)$ is defined by
\[
\delta(k)\coloneqq\eul^{\frac{1}{2\pi\ii}\int_{-\infty}^{\mu}\frac{\ln(1+\abs{r(s)}^2)}{s-k}\,\dd s},\quad k\in\D{C}\setminus(-\infty,\mu\rbrack.
\]
Note that $\abs{r(s)}^2=\abs{\hat r(s)}^2$.

%-------------------%
%:lem 4.1
%-------------------%
\begin{lemma}   \label{delta-properties}
The function $\delta(k)$ has the following properties:
\begin{enumerate}[\rm(a)]
\item
$\delta(k)$ and $\delta(k)^{-1}$ are bounded and analytic functions of $k\in\D{C}\setminus(-\infty,\mu\rbrack$ with continuous boundary values on $(-\infty,-1)\cup(-1,\mu)$.
\item
$\delta$ obeys the symmetry
\[
\delta=(\delta^*)^{-1},\quad k\in\D{C}\setminus(-\infty,\mu\rbrack.
\]
\item
$\delta$ satisfies the jump condition
\[
\delta_+=\delta_-(1+\abs{r}^2),\quad k\in(-\infty,-1)\cup(-1,\mu).
\]
\item
As $k$ goes to infinity,
\[
\delta(k)=1+\ord(k^{-1}),\quad k\to\infty, 
\]
uniformly with respect to $\arg k\in\croch{0,2\pi}$.
\end{enumerate}
\end{lemma}
%-------------------%

%-------------------%
\begin{proof}
The function $\ln(1+\abs{r(s)}^2)$ is a smooth function of $s\in(-\infty,\mu\rbrack$ except for a possible jump discontinuity at $s=-1$ where the real axis intersects the branch cut $\Sigma_1$. Since we also have $\ln(1+\abs{r(s)}^2)=\ord(\abs{s}^{-N})$ for every $N\geq 1$ as $s\to -\infty$, parts (b), (c), and (d) follow. Since the constants $\ln(1+\abs{r_-(-1)}^2)$, $\ln(1+\abs{r_+(-1)}^2)$, and $\ln(1+\abs{q}^2)$ are real, Lemma~\ref{lem-C} implies that $\delta$ is bounded near the points $-1$ and $\mu$. This shows that (a) also holds.
\end{proof}
%-------------------%

%-------------------%
\begin{comment*}
Lemma~\ref{delta-properties} assumes implicitly $\mu>-1$, as in Figures~\ref{fig:Img-contour-plot}, \ref{fig:canonical-basis}, and \ref{fig:curves-aj-bj} where $\mu\in(-1,1)$. This is indeed the case in the framework of the existence theorem (Section~\ref{sec:thm}) since $\mu(\xi)$ is close to $0$ for $\xi>0$ small enough. In cases where $\mu(\xi)$ leaves the range $(-1,1)$ the cuts joining $E_j$ and $\bar E_j$ must be deformed into new cuts so that $\mu(\xi)\in(\hat B_1,\hat B_2)$ where $\hat B_j$ are the intersections of these cuts with the real axis.
\end{comment*}
%-------------------%

Lemma~\ref{delta-properties} implies that
\[
\delta(\xi,\,\cdot\,)^{\sigma_3}\in I+(\dot E^2\cap E^{\infty})(\D{C}\setminus(-\infty,\mu\rbrack).
\]
Hence it follows from Lemma~\ref{rhp-deformation} that $\hat m$ satisfies the RH problem \eqref{rhp-0} iff $\hat m^{(2)}$ satisfies the RH problem \eqref{rhp-j} with $j=2$, where $\Sigma^{(2)}=\Sigma^{(1)}$ and the jump matrix $\hat v^{(2)}$ is given by
\[
\hat v^{(2)}=\begin{pmatrix}
\delta_-&0\\0&\delta_-^{-1}\end{pmatrix}
\hat v^{(1)}\begin{pmatrix}
\delta_+^{-1}&0\\0&\delta_+\end{pmatrix}.
\]
We can write
\[
\hat v^{(2)}=
\begin{cases}
\begin{pmatrix}
1&0\\\frac{\hat r}{1+\hat r\hat r^*}\delta_-^{-2}\eul^{2\ii tg}&1\end{pmatrix}
\begin{pmatrix}
1&\frac{\hat r^*}{1+\hat r\hat r^*}\delta_+^2\eul^{-2\ii tg}\\0&1\end{pmatrix},
&k\in(-\infty,\mu),\\
\begin{pmatrix}
1&\hat r^*\delta^2\eul^{-2\ii tg}\\0&1\end{pmatrix}
\begin{pmatrix}
1&0\\\hat r\delta^{-2}\eul^{2\ii tg}&1\end{pmatrix},
&k\in(\mu,+\infty),\\
\begin{pmatrix}
-\ii\eul^{\ii t(g_+-g_-)}&0\\\frac{\ii\eul^{-\ii\phi}}{\hat a_+\hat a_-}\delta^{-2}\eul^{\ii t(g_++g_-)}&\ii\eul^{-\ii t(g_+-g_-)}\end{pmatrix},
&k\in\Sigma_1\cap\D{C}^+,\\
\begin{pmatrix}
\frac{\hat a_-}{\hat a_+}\eul^{\ii t(g_+-g_-)}&\ii\nu_1^2\delta^2\eul^{-\ii t(g_++g_-)}\\0&\frac{\hat a_+}{\hat a_-}\eul^{-\ii t(g_+-g_-)}\end{pmatrix},
&k\in\Sigma_2\cap\D{C}^+,\\
\begin{pmatrix}
\ii\eul^{\ii t(g_+-g_-)}&\frac{\ii\eul^{\ii\phi}}{\hat a_+^*\hat a_-^*}\delta^2\eul^{-\ii t(g_++g_-)}\\0&\ii\eul^{-\ii t(g_+-g_-)}\end{pmatrix},
&k\in\Sigma_1\cap\D{C}^-,\\
\begin{pmatrix}
\frac{\hat a_+^*}{\hat a_-^*}\eul^{\ii t(g_+-g_-)}&0\\\ii\nu_1^{-2}\delta^{-2}\eul^{\ii t(g_++g_-)}&\frac{\hat a_-^*}{\hat a_+^*}\eul^{-\ii t(g_+-g_-)}\end{pmatrix},
&k\in\Sigma_2\cap\D{C}^-,\\
\begin{pmatrix}
\eul^{\ii t(g_+-g_-)}&0\\0&\eul^{-\ii t(g_+-g_-)}\end{pmatrix},
&k\in\gamma_{(\bar\alpha,\bar\beta)}\cup\gamma_{(\bar\beta,\beta)}\cup\gamma_{(\beta,\alpha)}.
\end{cases}
\]
%---------------------------------------------------------%
%:s.4.3
%---------------------------------------------------------%
\subsection{Third transformation}

%-------------------%
%:fig 9 = 4.2
%-------------------%
\begin{figure}[ht]
\centering\includegraphics[scale=.75]{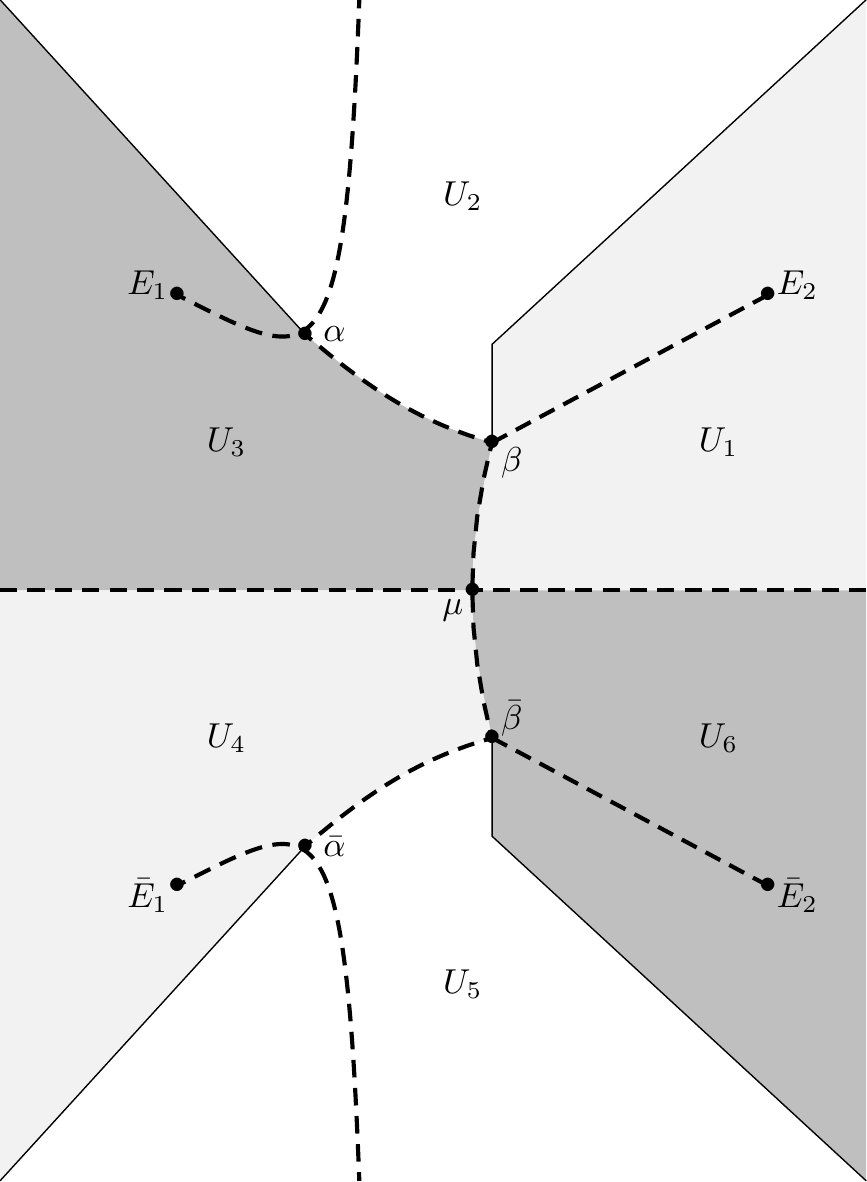}
\caption{The sets $\accol{U_j}_1^6$ together with the set where $\Im g=0$ (dashed black).} 
\label{fig:Uj}
\end{figure}
%-------------------%

Let $\accol{U_j}_1^6$ denote open sets of the form displayed in Figure~\ref{fig:Uj} and define $\hat m^{(3)}(x,t,k)$ by
\[
\hat m^{(3)}\coloneqq\hat m^{(2)}\times
\begin{cases}
\begin{pmatrix}
1&0\\
-\hat r\delta^{-2}\eul^{2\ii tg}&1\end{pmatrix},&k\in U_1,\\
\begin{pmatrix}
1&-\frac{\hat r^*}{1+\hat r\hat r^*}\delta^2\eul^{-2\ii tg}\\
0&1\end{pmatrix},&k\in U_3,\\
\begin{pmatrix}
1&0\\
\frac{\hat r}{1+\hat r\hat r^*}\delta^{-2}\eul^{2\ii tg}&1\end{pmatrix},&k\in U_4,\\
\begin{pmatrix}
1&\hat r^*\delta^2\eul^{-2\ii tg}\\
0&1\end{pmatrix},&k\in U_6,\\
\,I,&k\in U_2\cup U_5.
\end{cases}
\]
Then, there is no more jump across the real axis. Note also that $\hat r$ is always bounded at $E_2\in U_1$ and that, in the generic case, $\frac{\hat r^*}{1+\hat r\hat r^*}=\hat a\hat b$ is bounded at $E_1\in U_3$.

Across the curve $\Sigma_1\cap\D{C}^+$ which lies in $U_3$, we find the jump
\begin{align*}
\hat v^{(3)}&=\begin{pmatrix}
1&-\frac{\hat r_-^*}{1+\hat r_-\hat r_-^*}\delta^2\eul^{-2\ii tg_-}\\0&1\end{pmatrix}^{-1}\hat v^{(2)}\begin{pmatrix}
1&-\frac{\hat r_+^*}{1+\hat r_+\hat r_+^*}\delta^2\eul^{-2\ii tg_+}\\0&1\end{pmatrix}\\
&=\begin{pmatrix}1&\frac{\hat r_-^*}{1+\hat r_-\hat r_-^*}\delta^2\eul^{-2\ii tg_-}\\0&1\end{pmatrix}\begin{pmatrix}-\ii\eul^{\ii t(g_+-g_-)}&0\\-\frac{\eul^{-\ii\phi}\nu_{1+}^{-2}}{a_+a_-}\delta^{-2}\eul^{\ii t(g_++g_-)}&\ii\eul^{-\ii t(g_+-g_-)}\end{pmatrix}\begin{pmatrix}1&-\frac{\hat r_+^*}{1+\hat r_+\hat r_+^*}\delta^2\eul^{-2\ii tg_+}\\0&1\end{pmatrix}\\
&=\begin{pmatrix}0&\ii\hat a_+\hat a_-\eul^{\ii\phi}\delta^2\eul^{-\ii t(g_++g_-)}\\\frac{\ii\eul^{-\ii\phi}}{\hat a_+\hat a_-}\delta^{-2}\eul^{\ii t(g_++g_-)}&0\end{pmatrix},\quad k\in\Sigma_1\cap\D{C}^+,
\end{align*}  
where we have used \eqref{ab1} in the last step, together with $aa^*+bb^*=1$. Note that $\hat a_+\hat a_-\eul^{\ii\phi}=\hat a_+\hat b_+$. 

Similarly, across the curve $\Sigma_2\cap\D{C}^+$ which lies in $U_1$, we find the jump
\begin{align*}
\hat v^{(3)}&=\begin{pmatrix}
1&0\\-\hat r_-\delta^{-2}\eul^{2\ii tg_-}&1\end{pmatrix}^{-1}\hat v^{(2)}\begin{pmatrix}
1&0\\-\hat r_+\delta^{-2}\eul^{2\ii tg_+}&1\end{pmatrix}\\
&=\begin{pmatrix}
1&0\\\hat r_-\delta^{-2}\eul^{2\ii tg_-}&1\end{pmatrix}\begin{pmatrix}
\frac{a_-}{a_+}\eul^{\ii t(g_+-g_-)}&\ii\nu_1^2\delta^2\eul^{-\ii t(g_++g_-)}\\0&\frac{a_+}{a_-}\eul^{-\ii t(g_+-g_-)}\end{pmatrix}\begin{pmatrix}1&0\\-\hat r_+\delta^{-2}\eul^{2\ii tg_+}&1\end{pmatrix}\\
&=\begin{pmatrix}0&\ii\nu_1^2\delta^2\eul^{-\ii t(g_++g_-)}\\\ii\nu_1^{-2}\delta^{-2}\eul^{\ii t(g_++g_-)}&0\end{pmatrix},\quad k\in\Sigma_2\cap\D{C}^+,
\end{align*}  
where we have used \eqref{ab2} in the last step. Note that $\hat r_-\frac{a_-}{a_+}=\ii\nu_1^{-2}$.

Across the cut $\gamma_{(\beta,\alpha)}$ which lies on the boundary between $U_2$ and $U_3$ with $U_3$ to the left, we find the jump
\begin{align*}
\hat v^{(3)}&=\hat v^{(2)}\begin{pmatrix}
1&-\frac{\hat r^*}{1+\hat r\hat r^*}\delta^2\eul^{-2\ii tg_+}\\0&1\end{pmatrix}\\
&=\begin{pmatrix}\eul^{\ii t(g_+-g_-)}&0\\0&\eul^{-\ii t(g_+-g_-)}\end{pmatrix}\begin{pmatrix}
1&-\frac{\hat r^*}{1+\hat r\hat r^*}\delta^2\eul^{-2\ii tg_+}\\0&1\end{pmatrix}\\
&=\eul^{-\ii tg_-\sigma_3}\begin{pmatrix}1&-\hat a\hat b\delta^2\\0&1\end{pmatrix}\eul^{\ii tg_+\sigma_3},\quad k\in\gamma_{(\beta,\alpha)}.
\end{align*}
Across the curve $\gamma_{(\mu,\beta)}$ which lies on the boundary between $U_1$ and $U_3$ with $U_3$ to the left, we find the jump
\begin{align*}
\hat v^{(3)}&=\begin{pmatrix}
1&0\\-\hat r\delta^{-2}\eul^{2\ii tg_-}&1\end{pmatrix}^{-1}\hat v^{(2)}\begin{pmatrix}
1&-\frac{\hat r^*}{1+\hat r\hat r^*}\delta^2\eul^{-2\ii tg_+}\\0&1\end{pmatrix}\\
&=\begin{pmatrix}1&0\\-\hat r\delta^{-2}\eul^{2\ii tg_-}&1\end{pmatrix}^{-1}\begin{pmatrix}\eul^{\ii t(g_+-g_-)}&0\\0&\eul^{-\ii t(g_+-g_-)}\end{pmatrix}\begin{pmatrix}1&-\frac{\hat r^*}{1+\hat r\hat r^*}\delta^2\eul^{-2\ii tg_+}\\0&1\end{pmatrix}\\
&=\eul^{-\ii tg_-\sigma_3}\begin{pmatrix}1&-\hat a\hat b\delta^2\\\hat r\delta^{-2}&\hat a\hat a^*\end{pmatrix}\eul^{\ii tg_+\sigma_3},\quad k\in\gamma_{(\mu,\beta)}.
\end{align*}

Estimates \eqref{reflec-infty} together with Lemma~\ref{g-properties}, Lemma~\ref{delta-properties}, and the genericity assumption (on the behavior of $a(k)$ at $E_1$) imply that
\begin{alignat*}{2}
&\hat r\delta^{-2}\eul^{2\ii tg}\in(\dot E^2\cap E^{\infty})(U_1\setminus\Sigma_2),&\qquad&\frac{\hat r^*}{1+\hat r\hat r^*}\delta^2\eul^{-2\ii tg}\in(\dot E^2\cap E^{\infty})(U_3\setminus\Sigma_1),\\
&\frac{\hat r}{1+\hat r\hat r^*}\delta^{-2}\eul^{2\ii tg}\in(\dot E^2\cap E^{\infty})(U_4\setminus\Sigma_1),&&\hat r^*\delta^2\eul^{-2\ii tg}\in(\dot E^2\cap E^{\infty})(U_6\setminus\Sigma_2).
\end{alignat*}
Notice that $\hat r(k)$ is always bounded at $E_2$ whereas $\frac{\hat r^*(k)}{1+\hat r(k)\hat r^*(k)}=\hat a(k)\hat b(k)$ is bounded at $E_1$ in the generic case.

Since $g$ is quadratic in $k$ the decay of $\eul^{2\ii tg}$ in regions where $\Im g>0$ indeed prevails over the growth of $r(k)$ provided by \eqref{reflec-infty}.
%-------------------%
%:fig 10 = 4.3
%-------------------%
\begin{figure}[ht]
\centering\includegraphics[scale=.75]{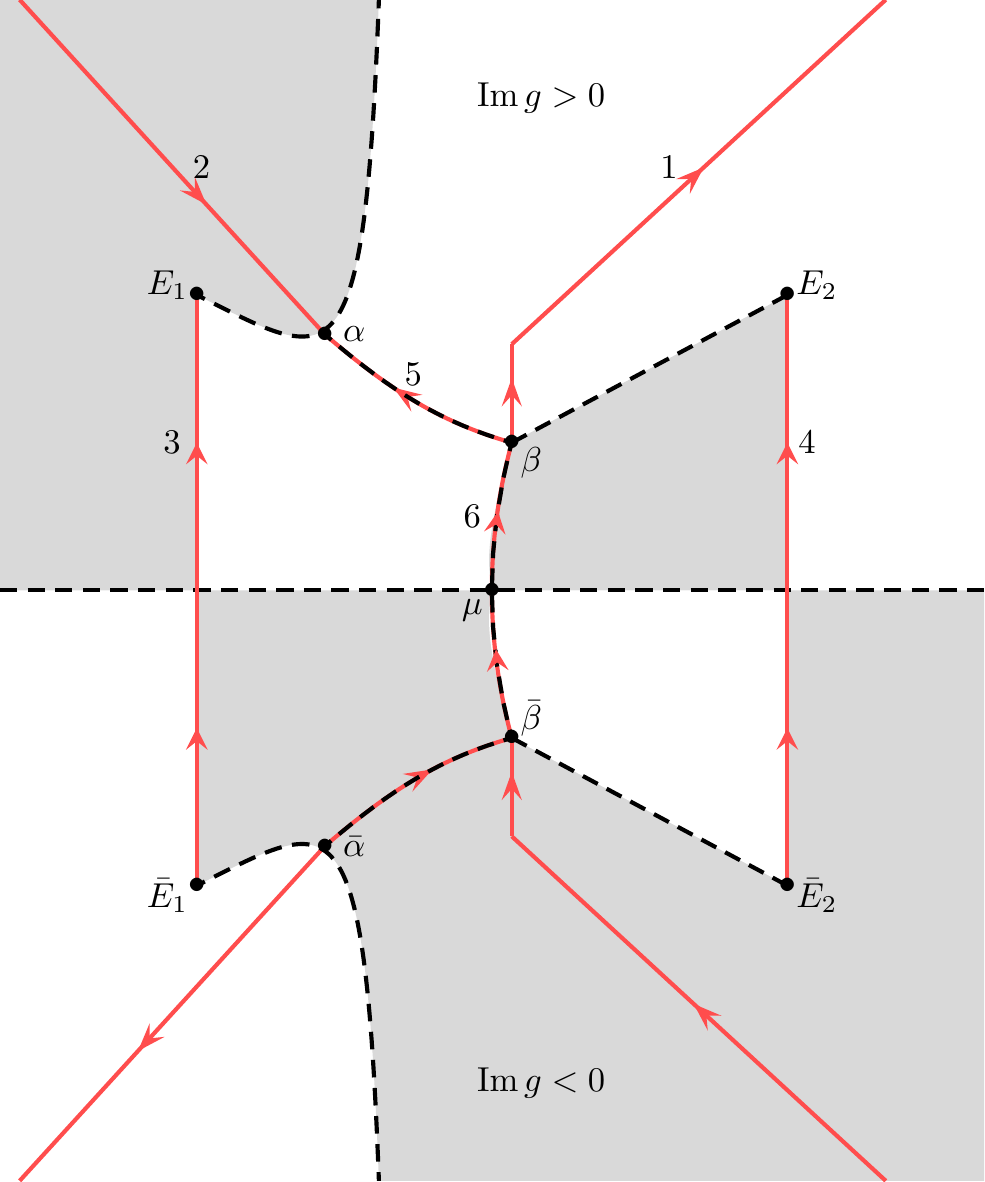}
\caption{The jump contour $\Sigma^{(3)}$ (red with arrows indicating orientation) together with the set where $\Im g=0$ (dashed black). The region where $\Im g<0$ is shaded and the region where $\Im g>0$ is white.} 
\label{fig:jump-contour-3}
\end{figure}
%-------------------%
Hence it follows from Lemma~\ref{rhp-deformation} that $\hat m$ satisfies the RH problem \eqref{rhp-0} iff $\hat m^{(3)}$ satisfies the RH problem \eqref{rhp-j} with $j=3$, where $\Sigma^{(3)}$ is the jump contour displayed in Figure~\ref{fig:jump-contour-3} and the jump matrix $\hat v^{(3)}$ is given in the upper half-plane by ($\hat v_l^{(3)}$ denotes the restriction of $\hat v^{(3)}$ to the contour labeled by $l$ in Figure~\ref{fig:jump-contour-3}) 
\begin{alignat*}{2}
&\hat v_1^{(3)}=
\begin{pmatrix}
1&0\\
\hat r\delta^{-2}\eul^{2\ii tg}&1\end{pmatrix},&&
\hat v_2^{(3)}=
\begin{pmatrix}
1&\hat a\hat b\delta^2\eul^{-2\ii tg}\\
0&1\end{pmatrix},\\
&\hat v_3^{(3)}=
\begin{pmatrix}0&\ii\hat a_+\hat a_-\eul^{\ii\phi}\delta^2\eul^{-\ii t(g_++g_-)}\\\frac{\ii\eul^{-\ii\phi}}{\hat a_+\hat a_-}\delta^{-2}\eul^{\ii t(g_++g_-)}&0\end{pmatrix},&&\\
&\hat v_4^{(3)}=
\begin{pmatrix}
0&\ii\nu_1^2\delta^2\eul^{-\ii t(g_++g_-)}\\
\ii\nu_1^{-2}\delta^{-2}\eul^{\ii t(g_++g_-)}&0\end{pmatrix},&&\\
&\hat v_5^{(3)}=
\eul^{-\ii tg_-\sigma_3}\begin{pmatrix}
1&-\hat a\hat b\delta^2\\
0&1\end{pmatrix}\eul^{\ii tg_+\sigma_3},&&
\hat v_6^{(3)}=
\eul^{-\ii tg_-\sigma_3}\begin{pmatrix}
1&-\hat a\hat b\delta^2\\
\hat r\delta^{-2}&\hat a\hat a^*\end{pmatrix}\eul^{\ii tg_+\sigma_3},
\end{alignat*}
and extended to the lower half-plane by means of the symmetry \eqref{vj-sym}.
%---------------------------------------------------------%
%:s.4.4
%---------------------------------------------------------%
\subsection{Fourth transformation}

The purpose of the fourth transformation is to make the jumps across $\gamma_{(\beta,\alpha)}\cup\gamma_{(\bar\alpha,\bar\beta)}$ and $\gamma_{(\bar\beta,\beta)}$ off-diagonal and diagonal, respectively.

Using the factorization
\[
\begin{pmatrix}1&P^{-1}\\0&1\end{pmatrix}=\begin{pmatrix}1&0\\P&1\end{pmatrix}\begin{pmatrix}0&P^{-1}\\-P&0\end{pmatrix}\begin{pmatrix}1&0\\P&1\end{pmatrix}
\]
we can write the jump matrix $\hat v^{(3)}$ for $k\in\gamma_{(\beta,\alpha)}$ as
\[
\hat v_5^{(3)}=\eul^{-\ii tg_-\sigma_3}\begin{pmatrix}1&0\\-\frac{1}{\hat a\hat b\delta^2}&1\end{pmatrix}\begin{pmatrix}0&-\hat a\hat b\delta^2\\\frac{1}{\hat a\hat b\delta^2}&0\end{pmatrix}\begin{pmatrix}1&0\\-\frac{1}{\hat a\hat b\delta^2}&1\end{pmatrix}\eul^{\ii tg_+\sigma_3},\quad k\in\gamma_{(\beta,\alpha)}.
\]
Similarly, using the factorization
\[
\begin{pmatrix}1&P\\Q&1+PQ\end{pmatrix}=\begin{pmatrix}1&\frac{P}{1+PQ}\\0&1\end{pmatrix}\begin{pmatrix}\frac{1}{1+PQ}&0\\0&1+PQ\end{pmatrix}\begin{pmatrix}1&0\\\frac{Q}{1+PQ}&1\end{pmatrix}
\]
the jump matrix $\hat v^{(3)}$ for $k\in\gamma_{(\mu,\beta)}$ can be written as
\[
\hat v_6^{(3)}=\eul^{-\ii tg_-\sigma_3}\begin{pmatrix}1&-\frac{\hat b}{\hat a^*}\delta^2\\0&1\end{pmatrix}\begin{pmatrix}\frac{1}{\hat a\hat a^*}&0\\0&\hat a\hat a^*\end{pmatrix}\begin{pmatrix}1&0\\\frac{\hat b^*}{\hat a^2\hat a^*}\delta^{-2}&1\end{pmatrix}\eul^{\ii tg_+\sigma_3},\quad k\in\gamma_{(\mu,\beta)}.
\]
%-------------------%
%:fig 11 = 4.4
%-------------------%
\begin{figure}[ht]
\centering\includegraphics[scale=.8]{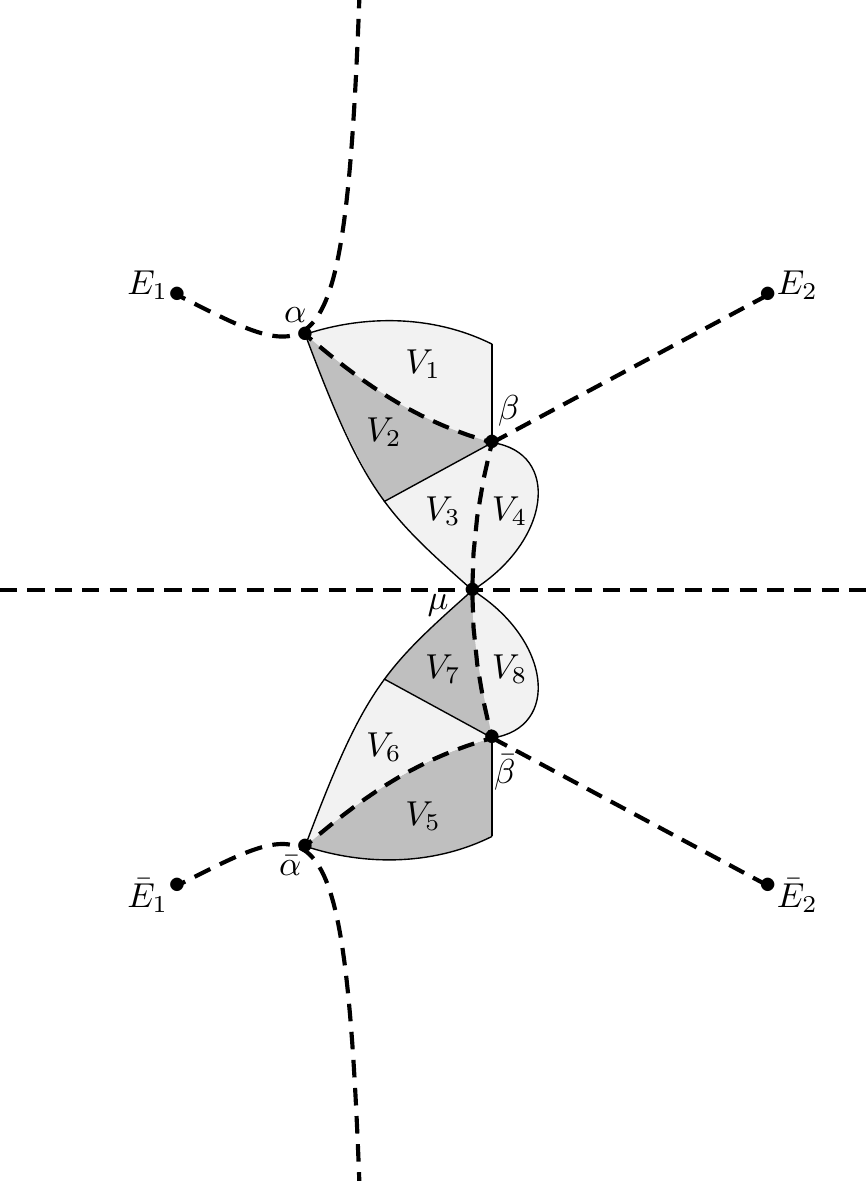}
\caption{The sets $\accol{V_j}_1^8$ together with the level set $\Im g=0$ (dashed).} 
\label{fig:Vj}
\end{figure}
%-------------------%

Let $\accol{V_j}_1^8$ denote the open sets displayed in Figure~\ref{fig:Vj}. Define $\hat m^{(4)}$ for $k$ in the upper half-plane by
\[
\hat m^{(4)}\coloneqq\hat m^{(3)}\times
\begin{cases}
\begin{pmatrix}
1&0\\-\frac{\eul^{2\ii tg}}{\hat a\hat b\delta^2}&1\end{pmatrix},&k\in V_1,\\
\begin{pmatrix}
1&0\\\frac{\eul^{2\ii tg}}{\hat a\hat b\delta^2}&1\end{pmatrix},&k\in V_2,\\
\begin{pmatrix}
1&0\\-\frac{\hat b^*}{\hat a^2\hat a^*}\delta^{-2}\eul^{2\ii tg}&1\end{pmatrix},&k\in V_3,\\
\begin{pmatrix}
1&-\frac{\hat b}{\hat a^*}\delta^2\eul^{-2\ii tg}\\0&1\end{pmatrix},&k\in V_4,\\
I,&\text{elsewhere in }\D{C}^+,
\end{cases}
\]
and extend the definition to the lower half-plane by means of the symmetry \eqref{mj-sym}.

Lemma~\ref{g-properties} and Lemma~\ref{delta-properties} imply
\begin{alignat*}{2}
&\frac{\eul^{2\ii tg}}{\hat a\hat b\delta^2}\in E^{\infty}(V_1),&\qquad&\frac{\eul^{2\ii tg}}{\hat a\hat b\delta^2}\in E^{\infty}(V_2),\\
&\frac{\hat b^*}{\hat a^2\hat a^*}\delta^{-2}\eul^{2\ii tg}\in E^{\infty}(V_3),&&\frac{\hat b}{\hat a^*}\delta^2\eul^{-2\ii tg}\in E^{\infty}(V_4).
\end{alignat*}

%-------------------%
%:fig 12 = 4.5
%-------------------%
\begin{figure}[ht]
\centering\includegraphics[scale=.8]{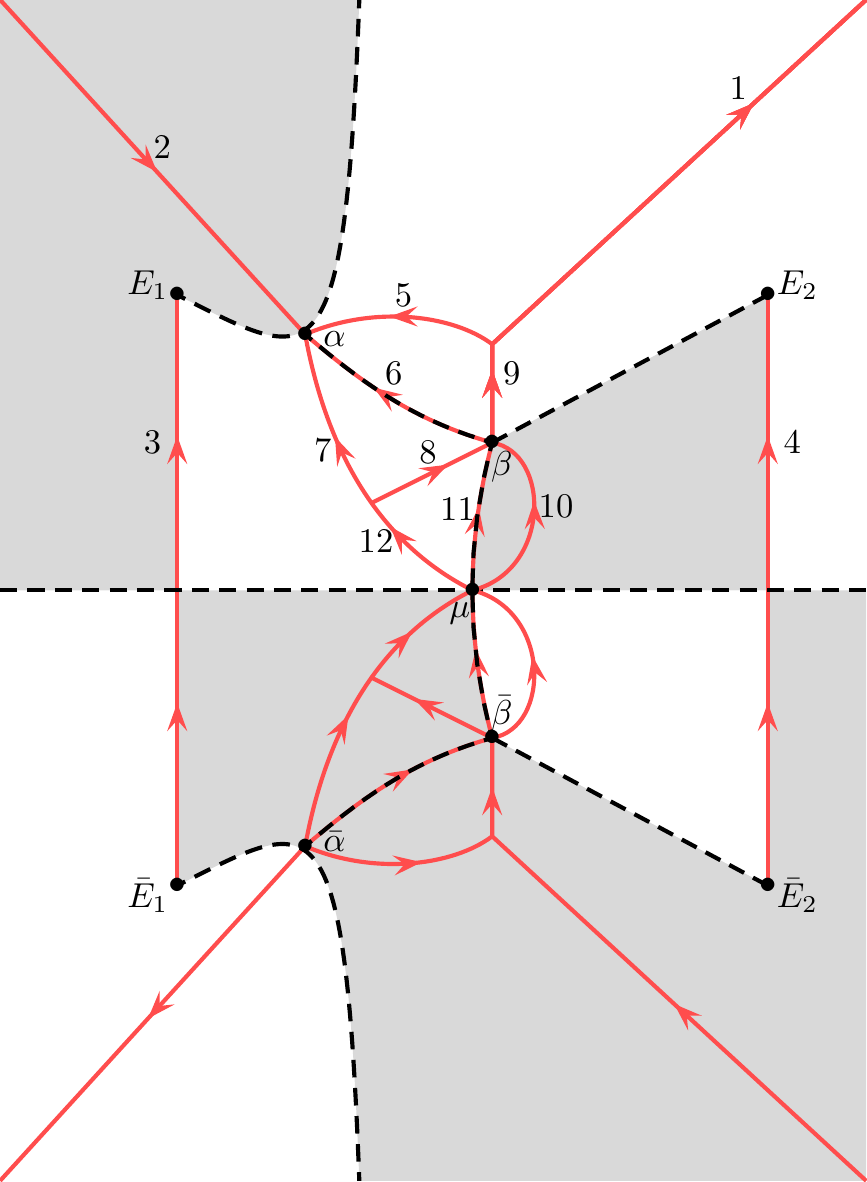}
\caption{The jump contour $\Sigma^{(4)}=\Sigma^{(5)}$ (red with arrows indicating orientation) together with the set where $\Im g=0$ (dashed black). The region where $\Im g<0$ is shaded and the region where $\Im g>0$ is white.} 
\label{fig:jump-contour-4}
\end{figure}
%-------------------%

Hence it follows from Lemma~\ref{rhp-deformation} that $\hat m$ satisfies the RH problem \eqref{rhp-0} iff $\hat m^{(4)}$ satisfies the RH problem \eqref{rhp-j} with $j=4$, where $\Sigma^{(4)}$ is the jump contour displayed in Figure~\ref{fig:jump-contour-4} and the jump matrix $\hat v^{(4)}$ is given on the part of $\Sigma^{(4)}$ that lies in the upper half-plane by ($\hat v_l^{(4)}$ denotes the restriction of $\hat v^{(4)}$ to the contour labeled by $l$ in Figure~\ref{fig:jump-contour-4}):
\begin{alignat*}{2}
&\hat v_1^{(4)}=\begin{pmatrix}1&0\\\frac{\hat b^*}{\hat a}\delta^{-2}\eul^{2\ii tg}&1\end{pmatrix},&&\hat v_2^{(4)}=\begin{pmatrix}1&\hat a\hat b\delta^2\eul^{-2\ii tg}\\0&1\end{pmatrix},\\
&\hat v_3^{(4)}=\begin{pmatrix}0&\ii\hat a_+\hat a_-\eul^{\ii\phi}\delta^2\eul^{-\ii t(g_++g_-)}\\\frac{\ii\eul^{-\ii\phi}}{\hat a_+\hat a_-}\delta^{-2}\eul^{\ii t(g_++g_-)}&0\end{pmatrix},&&\\
&\hat v_4^{(4)}=\begin{pmatrix}0&\ii\nu_1^2\delta^2\eul^{-\ii t(g_++g_-)}\\\ii\nu_1^{-2}\delta^{-2}\eul^{\ii t(g_++g_-)}&0\end{pmatrix},&&\hat v_5^{(4)}=\hat v_7^{(4)}=\begin{pmatrix}1&0\\-\frac{\eul^{2\ii tg}}{\hat a\hat b\delta^2}&1\end{pmatrix},\\
&\hat v_6^{(4)}=\begin{pmatrix}0&-\hat a\hat b\delta^2\eul^{-\ii t(g_++g_-)}\\\frac{1}{\hat a\hat b\delta^2}\eul^{\ii t(g_++g_-)}&0\end{pmatrix},&&\hat v_8^{(4)}=\begin{pmatrix}1&0\\\frac{1}{\hat a^2\hat a^*\hat b}\delta^{-2}\eul^{2\ii tg}&1\end{pmatrix},\\
&\hat v_9^{(4)}=\begin{pmatrix}1&0\\-\frac{\hat a^*}{\hat b}\delta^{-2}\eul^{2\ii tg}&1\end{pmatrix},&&\hat v_{10}^{(4)}=\begin{pmatrix}1&-\frac{\hat b}{\hat a^*}\delta^2\eul^{-2\ii tg}\\0&1\end{pmatrix},\\
&\hat v_{11}^{(4)}=\begin{pmatrix}\frac{\eul^{\ii t(g_+-g_-)}}{\hat a\hat a^*}&0\\0&\hat a\hat a^*\eul^{-\ii t(g_+-g_-)}\end{pmatrix},&&\hat v_{12}^{(4)}=\begin{pmatrix}1&0\\\frac{\hat b^*}{\hat a^2\hat a^*}\delta^{-2}\eul^{2\ii tg}&1\end{pmatrix},
\end{alignat*}
and is extended to the part of $\Sigma^{(4)}$ that lies in the lower half-plane by means of the symmetry \eqref{vj-sym}.

%---------------------------------------------------------%
%:s.4.5
%---------------------------------------------------------%
\subsection{Fifth transformation}

The purpose of the fifth transformation is to make the jumps across the four branch cuts and across the curve $\gamma_{(\bar\beta,\beta)}$ constant in $k$.

We define $\hat m^{(5)}$ by
\[
\hat m^{(5)}(x,t,k)\coloneqq\eul^{-\ii h(\infty)\sigma_3}\hat m^{(4)}(x,t,k)\eul^{\ii h(k)\sigma_3},
\]
where $h(k)\equiv h(\xi,k)$ is a function yet to be determined. Suppose $h(k)$ is an analytic function of $k\in\hat{\D{C}}\setminus\Sigma^{\model}$ such that $h=h^*$ and
\begin{equation}
\eul^{\ii h\sigma_3}\in E^{\infty}(\hat{\D{C}}\setminus\Sigma^{\model}).
\end{equation}
Then Lemma~\ref{rhp-deformation-bis} implies that $\hat m$ satisfies the RH problem \eqref{rhp-0} iff $\hat m^{(5)}$ satisfies the RH problem \eqref{rhp-j} with $j=5$, where $\Sigma^{(5)}=\Sigma^{(4)}$ and
\[
\hat v^{(5)}=\eul^{-\ii h_-\sigma_3}\hat v^{(4)}\eul^{\ii h_+\sigma_3}.
\]
Using the relations \eqref{g-jump} we find that the jump matrix $\hat v^{(5)}$ is given in the upper half-plane by ($\hat v_l^{(5)}$ denotes the restriction of $\hat v^{(5)}$ to the contour labeled by $l$ in Figure~\ref{fig:jump-contour-4})
\begin{equation}   \label{jump-v5}
\begin{split}
&\hat v_1^{(5)}=\begin{pmatrix}1&0\\\frac{\hat b^*}{\hat a}\delta^{-2}\eul^{2\ii tg}\eul^{2\ii h}&1\end{pmatrix},\qquad\qquad\qquad\qquad\qquad\hat v_2^{(5)}=\begin{pmatrix}1&\hat a\hat b\delta^2\eul^{-2\ii tg}\eul^{-2\ii h}\\0&1\end{pmatrix},\\
&\hat v_3^{(5)}=\begin{pmatrix}0&\ii\hat a_+\hat a_-\eul^{\ii\phi}\delta^2\eul^{-2\ii t\Omega_1}\eul^{-\ii(h_++h_-)}\\\frac{\ii\eul^{-\ii\phi}}{\hat a_+\hat a_-}\delta^{-2}\eul^{2\ii t\Omega_1}\eul^{\ii(h_++h_-)}&0\end{pmatrix},\\
&\hat v_4^{(5)}=\begin{pmatrix}0&\ii\nu_1^2\delta^2\eul^{-\ii(h_++h_-)}\\\ii\nu_1^{-2}\delta^{-2}\eul^{\ii(h_++h_-)}&0\end{pmatrix},\qquad\qquad\quad\hat v_5^{(5)}=\hat v_7^{(5)}=\begin{pmatrix}1&0\\-\frac{\eul^{2\ii tg}}{\hat a\hat b\delta^2}\eul^{2\ii h}&1\end{pmatrix},\\
&\hat v_6^{(5)}=\begin{pmatrix}0&-\hat a\hat b\delta^2\eul^{-2\ii t\Omega_2}\eul^{-\ii(h_++h_-)}\\\frac{1}{\hat a\hat b\delta^2}\eul^{2\ii t\Omega_2}\eul^{\ii(h_++h_-)}&0\end{pmatrix},\\
&\hat v_8^{(5)}=\begin{pmatrix}1&0\\\frac{1}{\hat a^2\hat a^*\hat b}\delta^{-2}\eul^{2\ii tg}\eul^{2\ii h}&1\end{pmatrix},\qquad\qquad\qquad\qquad\quad\hat v_9^{(5)}=\begin{pmatrix}1&0\\-\frac{\hat a^*}{\hat b}\delta^{-2}\eul^{2\ii tg}\eul^{2\ii h}&1\end{pmatrix},\\
&\hat v_{10}^{(5)}=\begin{pmatrix}1&-\frac{\hat b}{\hat a^*}\delta^2\eul^{-2\ii tg}\eul^{-2\ii h}\\0&1\end{pmatrix},\\
&\hat v_{11}^{(5)}=\begin{pmatrix}\frac{\eul^{2\ii t\Omega_3}\eul^{\ii(h_+-h_-)}}{\hat a\hat a^*}&0\\0&\hat a\hat a^*\eul^{-2\ii t\Omega_3}\eul^{-\ii(h_+-h_-)}\end{pmatrix},\quad\ \ \hat v_{12}^{(5)}=\begin{pmatrix}1&0\\\frac{\hat b^*}{\hat a^2\hat a^*}\delta^{-2}\eul^{2\ii tg}\eul^{2\ii h}&1\end{pmatrix},
\end{split}
\end{equation}
and is extended to the lower half-plane by means of the symmetry \eqref{vj-sym}. We can make the jumps across the four cuts and across $\gamma_{(\bar\beta,\beta)}$ constant in $k$ by choosing the function $h$ so that it obeys the symmetry $h=h^*$ and satisfies the jump conditions
\begin{subequations}   \label{h-jump}
\begin{equation}  \label{h-jump-a}
h_++h_-=
\begin{cases}
2\omega_1-\ii\ln(\hat a_+\hat a_-\delta^2\eul^{\ii\phi}),&k\in\Sigma_1\cap\D{C}^+,\\
2\omega_1+\ii\ln(\hat a_+^*\hat a_-^*\delta^{-2}\eul^{-\ii\phi}),&k\in\Sigma_1\cap\D{C}^-,\\
2\omega_2-\ii\ln(\ii\hat a\hat b\delta^2),&k\in\gamma_{(\beta,\alpha)},\\
2\omega_2+\ii\ln(-\ii\hat a^*\hat b^*\delta^{-2}),&k\in\gamma_{(\bar\alpha,\bar\beta)},\\
-\ii\ln(\nu_1^2\delta^2),&k\in\Sigma_2,
\end{cases}
\end{equation}
and
\begin{equation}  \label{h-jump-b}
h_+-h_-=
\begin{cases}
2\omega_3-\ii\ln(aa^*),&k\in\gamma_{(\mu,\beta)},\\
2\omega_3+\ii\ln(aa^*),&k\in\gamma_{(\bar\beta,\mu)},
\end{cases}
\end{equation}
\end{subequations}
where $\accol{\omega_j}_1^3$ are real constants yet to be determined. The branches of the logarithms in \eqref{h-jump-a} can be chosen arbitrarily as long as the function $h_s\coloneqq h_++h_-$ satisfies the following two conditions:
\begin{enumerate}[(i)]
\item
$h_s$ is continuous on each of the five segments $\Sigma_1\cap\D{C}^+$, $\Sigma_1\cap\D{C}^-$, $\gamma_{(\beta,\alpha)}$, $\gamma_{(\bar\alpha,\bar\beta)}$, and $\Sigma_2$.
\item
$h_s$ obeys the symmetry $h_s=h_s^*$ for all $k\in\C{C}=\Sigma_1\cup\Sigma_2\cup\gamma_{(\beta,\alpha)}\cup\gamma_{(\bar\alpha,\bar\beta)}$.
\end{enumerate}
Since $\hat a(k)$ and $\hat b(k)$ are assumed to be nonzero, these conditions can easily be fulfilled. The branches of the logarithms in \eqref{h-jump-b} are fixed by requiring that $\ln(aa^*)=\ln(\abs{a(\mu)}^2)\in\D{R}$ for $k=\mu$ and extended in such a way that $h_+-h_-$ is continuous on each of the contours $\gamma_{(\mu,\beta)}$ and $\gamma_{(\bar\beta,\mu)}$.
 
If the conditions in \eqref{h-jump} are satisfied, then the jump matrix takes the form
\begin{alignat*}{2}
&\hat v_1^{(5)}=\begin{pmatrix}1&0\\\frac{\hat b^*}{\hat a}\delta^{-2}\eul^{2\ii tg}\eul^{2\ii h}&1\end{pmatrix},&\qquad&\hat v_2^{(5)}=\begin{pmatrix}1&\hat a\hat b\delta^2\eul^{-2\ii tg}\eul^{-2\ii h}\\0&1\end{pmatrix},\\
&\hat v_3^{(5)}=\begin{pmatrix}0&\ii\eul^{-2\ii(t\Omega_1+\omega_1)}\\\ii\eul^{2\ii(t\Omega_1+\omega_1)}&0\end{pmatrix},&&\hat v_4^{(5)}=\begin{pmatrix}0&\ii\\\ii&0\end{pmatrix},\\
&\hat v_5^{(5)}=\hat v_7^{(5)}=\begin{pmatrix}1&0\\-\frac{\eul^{2\ii tg}}{\hat a\hat b\delta^2}\eul^{2\ii h}&1\end{pmatrix},&&\hat v_6^{(5)}=\begin{pmatrix}0&\ii\eul^{-2\ii(t\Omega_2+\omega_2)}\\\ii\eul^{2\ii(t\Omega_2+\omega_2)}&0\end{pmatrix},\\
&\hat v_8^{(5)}=\begin{pmatrix}1&0\\\frac{1}{\hat a^2\hat a^*\hat b}\delta^{-2}\eul^{2\ii tg}\eul^{2\ii h}&1\end{pmatrix},&&\hat v_9^{(5)}=\begin{pmatrix}1&0\\-\frac{\hat a^*}{\hat b}\delta^{-2}\eul^{2\ii tg}\eul^{2\ii h}&1\end{pmatrix},\\
&\hat v_{10}^{(5)}=\begin{pmatrix}1&-\frac{\hat b}{\hat a^*}\delta^2\eul^{-2\ii tg}\eul^{-2\ii h}\\0&1\end{pmatrix},&&\hat v_{11}^{(5)}=\begin{pmatrix}\eul^{2\ii(t\Omega_3+\omega_3)}&0\\0&\eul^{-2\ii(t\Omega_3+\omega_3)}\end{pmatrix},\\
&\hat v_{12}^{(5)}=\begin{pmatrix}1&0\\\frac{\hat b^*}{\hat a^2\hat a^*}\delta^{-2}\eul^{2\ii tg}\eul^{2\ii h}&1\end{pmatrix}.&&
\end{alignat*}

%---------------------------------------------------------%
%:s.4.6
%---------------------------------------------------------%
\subsection{Construction of $\BS{h}$}

In order to construct a function $h(k)$ satisfying \eqref{h-jump}, we define the function $\C{H}(k)\equiv\C{H}(\xi,k)$ for $k\in\Sigma^{\model}\cap\D{C}^+$ by
\begin{subequations}     \label{calH-def}
\begin{equation}   \label{calH}
\C{H}(k)\coloneqq\begin{cases}
\frac{2\omega_1+h_1}{w_+},&k\in\Sigma_1,\\
\frac{2\omega_2+h_2}{w_+},&k\in\gamma_{(\beta,\alpha)}\cup\gamma_{(\bar\alpha,\bar\beta)},\\
\frac{2\omega_3+h_3}{w},&k\in\gamma_{(\bar\beta,\beta)},\\
\frac{h_4}{w_+},&k\in\Sigma_2,
\end{cases}
\end{equation}
with
\begin{equation}    \label{hj-defa}
\begin{cases}
h_1=-\ii\ln(\hat a_+\hat a_-\delta^2\eul^{\ii\phi}),&k\in\Sigma_1\cap\D{C}^+,\\
h_2=-\ii\ln(\ii\hat a\hat b\delta^2)&k\in\gamma_{(\beta,\alpha)},\\
h_3=-\ii\ln(aa^*),&k\in\gamma_{(\mu,\beta)},\\
h_4=-\ii\ln(\nu_1^2\delta^2),&k\in\Sigma_2\cap\D{C}^+,
\end{cases}
\end{equation}
and extend it to $\Sigma^{\model}\cap\D{C}^-$ by means of the symmetries \begin{equation}    \label{hj-defb}
h_j=h_j^*,\ \ j=1,\dots,4.
\end{equation}
\end{subequations}
If the $\omega_j$ are real, these symmetries are equivalent to the symmetry
\begin{equation}\label{calH-symmetry}
\C{H}=\C{H}^*.
\end{equation}
Then we define $h$ by
\begin{equation}  \label{h-def}
h(k)\coloneqq\frac{w(k)}{2\pi\ii}\int_{\Sigma^{\model}}\frac{\C{H}(s)}{s-k}\,\dd s,\quad k\in\D{C}\setminus\Sigma^{\model}.
\end{equation}  
In general, the function $h$ has a pole at $\infty$ and is unbounded as $k$ approaches $\mu$ and $-1$. The following lemma shows that by choosing the constants $\omega_j$ appropriately, the pole at $\infty$ can be removed. Even for this choice of $\omega_j$, the function $h(k)$ is, in general, unbounded as $k$ approaches $\mu$ and $-1$; however, as the lemma shows, $\eul^{\ii h\sigma_3}$ is bounded and analytic for $k\in\hat{\D{C}}\setminus\Sigma^{\model}$. 

For $r>0$ and $z\in\D{C}$, we let denote $D_r(z)$ the open disk of radius $r$ centered at $z$.

%-------------------%
%:lem 4.2
%-------------------%
\begin{lemma}   \label{h-properties}
There is a unique choice of the real constants $\omega_j\equiv\omega_j(\xi)$, $j=1,2,3$, such that the function $h(k)$ defined in \eqref{h-def} has the following properties:
\begin{enumerate}[\rm(a)]
\item
$h$ obeys the symmetry
\begin{equation}   \label{h-symmetry}
h=h^*,\quad k\in\hat{\D{C}}\setminus\Sigma^{\model}.
\end{equation}
\item
As $k$ goes to infinity,
\begin{equation}  \label{h-at-infty}
h(k)=h(\infty)+\ord(k^{-1}),\quad k\to\infty,
\end{equation}
where $h(\infty)\equiv h(\xi,\infty)$ is a finite real number given by
\begin{equation}  \label{h(infty)}
h(\infty)=-\frac{1}{2\pi\ii}\int_{\Sigma^{\model}}s^3\C{H}(s)\dd s.
\end{equation}
\item
$\eul^{\ii h\sigma_3}$ is a bounded and analytic function of $k\in\hat{\D{C}}\setminus\Sigma^{\model}$.
\item
$h$ satisfies the jump conditions in \eqref{h-jump} across $\Sigma^{\model}$.
\end{enumerate}
\end{lemma}
%-------------------%

%-------------------%
\begin{proof}
For any choice of the $\omega_j$, $h(k)$ is an analytic function of $k\in\D{C}\setminus\Sigma^{\model}$ which satisfies the jump conditions in \eqref{h-jump}. Moreover, assuming the $\omega_j$ are real (which is proven below), the symmetry \eqref{calH-symmetry} and the general identity
\begin{equation}   \label{int-sym}
\overline{\int_{\gamma}f(s)\dd s}=\int_{\bar\gamma}\overline{f(\bar s)}\dd s
\end{equation}
valid for functions $f\colon\gamma\to\D{C}$, imply that $h$ obeys the symmetry \eqref{h-symmetry}.

We next show that there is a unique choice of the $\omega_j$ such that $h$ is bounded at infinity. We have $w(k)=k^4+\ord(k^3)$ as $k\to\infty$ and
\begin{align*}
\int_{\Sigma^{\model}}\frac{\C{H}(s)}{s-k}\,\dd s
&=-\frac{1}{k}\int_{\Sigma^{\model}}\C{H}(s)\dd s-\frac{1}{k^2}\int_{\Sigma^{\model}}s\C{H}(s)\dd s\\
&\quad-\frac{1}{k^3}\int_{\Sigma^{\model}}s^2\C{H}(s)\dd s+\ord(k^{-4}),\quad k\to\infty.
\end{align*} 
Hence $h$ is bounded as $k\to\infty$ provided that
\begin{equation}  \label{h-bounded}
\int_{\Sigma^{\model}}s^n\C{H}(s)\dd s=0,\quad n=0,1,2.
\end{equation}
Define the holomorphic differentials $\accol{\tilde\zeta_n}_0^2$ on $M$ by
\[
\tilde\zeta_n\coloneqq\frac{2s^n}{w}\,\dd s,\quad n=0,1,2.
\]
Isolating the dependence on the $\omega_j$, we can write the conditions in \eqref{h-bounded} as
\begin{equation}\label{h-bounded-bis}
\omega_1\int_{\Sigma_1}(\tilde\zeta_n)_++\omega_2\int_{\gamma_{(\beta,\alpha)}\cup\gamma_{(\bar\alpha,\bar\beta)}}(\tilde\zeta_n)_++\omega_3\int_{\gamma_{(\bar\beta,\beta)}}\tilde\zeta_n=F_n,\quad n=0,1,2,
\end{equation}
where $F_n\equiv F_n(\xi)$, $n=0,1,2$, are some complex numbers. The relations
\[
\int_{\Sigma_1}(\tilde\zeta_n)_+=-\frac{1}{2}\int_{a_1}\tilde\zeta_n,\qquad\int_{\gamma_{(\beta,\alpha)}\cup\gamma_{(\bar\alpha,\bar\beta)}}(\tilde\zeta_n)_+=-\frac{1}{2}\int_{a_2+a_3}\tilde\zeta_n,\qquad\int_{\gamma_{(\bar\beta,\beta)}}\tilde\zeta_n=\frac{1}{2}\int_{a_3}\tilde\zeta_n,
\]
show that equation \eqref{h-bounded-bis} imposes linear conditions on the periods of $\tilde\zeta_n$ along $a_1$, $a_2$, and $a_3$:
\[
\omega_1\int_{a_1}\tilde\zeta_n+\omega_2\int_{a_2}\tilde\zeta_n+(\omega_2-\omega_3)\int_{a_3}\tilde\zeta_n=-2F_n,\quad n=0,1,2.
\]
Since the $\tilde\zeta_n$ form a basis for the space $\C{H}^1(M)$ of holomorphic differentials on $M$ the system $\omega_1\int_{a_1}\tilde\zeta_n+\omega_2\int_{a_2}\tilde\zeta_n+(\omega_2-\omega_3)\int_{a_3}\tilde\zeta_n=0$ means that $\omega_1\int_{a_1}\eta+\omega_2\int_{a_2}\eta+(\omega_2-\omega_3)\int_{a_3}\eta=0$ for $\eta\in\C{H}^1(M)$. But this linear condition is necessarily trivial, since a holomorphic differential form on $M$ whose all the $a$-periods vanish is trivial, see e.g.~\cite{FK92}*{Proposition III.3.3}. Hence there is a unique choice of $\omega_j$ for which the conditions in \eqref{h-bounded} are satisfied. For this choice, $h(k)$ is analytic at $k=\infty$, so equation \eqref{h-at-infty} follows. 

Moreover, the $\omega_j$ are real, because \eqref{int-sym} with $f=f^*$ and $\bar\gamma=-\gamma$ shows that the values of the three integrals on the left-hand side of \eqref{h-bounded-bis} are all pure imaginary, and similarly for the constants $F_n$ which are sums of integrals of the same type:
\[
2F_n=\int_{\Sigma_1}h_1(\tilde\zeta_n)_++\int_{\gamma_{(\beta,\alpha)}\cup\gamma_{(\bar\alpha,\bar\beta)}}h_2(\tilde\zeta_n)_++\int_{\gamma_{(\bar\beta,\beta)}}h_3\tilde\zeta_n+\int_{\Sigma_2}h_4(\tilde\zeta_n)_+.
\]

We henceforth assume that the $\omega_j$ are chosen so that \eqref{h-at-infty} holds. The proof of the lemma will be complete if we can show that $\eul^{\pm\ii h}$ is bounded as $k$ approaches the points $E_1$, $E_2$, $\alpha$, $\beta$, $\mu$, and $-1$.

We first show that $h$ is bounded near $\alpha$. By the definition \eqref{h-def} of $h$, we have
\[
h(k)=w(k)\left(\frac{1}{2\pi\ii}\int_{\beta}^{\alpha}\frac{f_1(s)}{(\sqrt{s-\alpha})_+}\,\frac{\dd s}{s-k}+f_2(k)\right),
\]
where $\sqrt{s-\alpha}$ has the branch cut along the curve $\gamma_{(\beta,\alpha)}$ and the functions $f_1$ and $f_2$ are analytic at $\alpha$. Applying formula (29.5) from \cite{Mu92}*{Chapter 4, \S29} for $\gamma=\frac{1}{2}$ we get that the function
\[
f(k)\coloneqq\frac{1}{2\pi\ii}\int_{\beta}^{\alpha}\frac{f_1(s)}{(\sqrt{s-\alpha})_+}\,\frac{\dd s}{s-k}
\]
has the following expansion at $k=\alpha$:
\[
f(k)=\frac{1}{2}\frac{f_1(k)}{\sqrt{k-\alpha}}+\osmall\left(\frac{1}{\sqrt{k-\alpha}}\right).
\]
It follows that $\abs{f}\leq C\abs{k-\alpha}^{-1/2}$ for $k$ near $\alpha$. Hence $h$ is bounded near $\alpha$. The boundedness near $E_1$, $E_2$, and $\beta$ follows in a similar way.

We next show that the functions $\eul^{\pm\ii h(k)}$ are bounded near $\mu$. In this case, we write \eqref{h-def} as
\begin{equation}   \label{h-def-bis}
h(k)=\frac{w(k)}{2\pi\ii}\int_{\gamma_{(\bar\beta,\beta)}}\frac{2\omega_3}{w(s)(s-k)}\dd s+F_{\beta}(s)+F_{\bar\beta}(s)+f_3(k),\quad k\in\D{C}\setminus\Sigma^{\model},
\end{equation}
where $f_3$ is analytic at $\mu$,
\[
F_{\beta}(k)\coloneqq\frac{w(k)}{2\pi\ii}\int_{\gamma_{(\mu,\beta)}}\frac{f_1(s)}{s-k}\dd s,\qquad F_{\bar\beta}(k)\coloneqq\frac{w(k)}{2\pi\ii}\int_{\gamma_{(\bar\beta,\mu)}}\frac{f_2(s)}{s-k}\dd s,
\]
and
\[
f_1(s)\coloneqq-\frac{\ii\ln(a(s)a^*(s))}{w(s)},\qquad f_2(s)\coloneqq\frac{\ii\ln(a(s)a^*(s))}{w(s)}.
\]
Let $\epsilon>0$ be small. Since the function $w(s)$, $s\in\gamma_{(\bar\beta,\beta)}$, is smooth and nonzero at $s=\mu$, the first term on the right-hand side of \eqref{h-def-bis} is bounded for $k\in D_{\epsilon}(\mu)\setminus\gamma_{(\bar\beta,\beta)}$. Moreover, since $w(\mu)f_j(\mu)\in\ii\,\D{R}$ for $j=1,2$, Lemma~\ref{lem-C} implies that $\eul^{\pm\ii F_{\beta}}$ and $\eul^{\pm\ii F_{\bar\beta}}$ are bounded as $k\to\mu$. It follows that $\eul^{\pm\ii h}$ is bounded near $\mu$.

It only remains to show that the functions $\eul^{\pm\ii h(k)}$ are bounded near $-1$. In this case, we write \eqref{h-def} as 
\begin{equation}  \label{h-def-ter}
h(k)=\frac{w(k)}{2\pi\ii}\int_{\Sigma_1}\frac{2\omega_1}{w_+(s)(s-k)}\dd s-w(k)\left(F_{E_1}(k)-F_{\bar E_1}(k)\right)+f_1(k),\end{equation}
where
\begin{subequations}   \label{FE1s-def}
\begin{align}     \label{FE1-def}
F_{E_1}(k)&\coloneqq\frac{1}{2\pi}\int_{\Sigma_1\cap\D{C}^+}\frac{\ln\bigl(\hat a_+(s)\hat a_-(s)\delta(s)^2\eul^{\ii\phi}\bigr)}{w_+(s)(s-k)}\dd s,\\  \label{FbarE1-def}
F_{\bar E_1}(k)&\coloneqq\frac{1}{2\pi}\int_{\Sigma_1\cap\D{C}^-}\frac{\ln\bigl(\hat a_+^*(s)\hat a_-^*(s)\delta(s)^{-2}\eul^{-\ii\phi}\bigr)}{w_+(s)(s-k)}\dd s,
\end{align}
\end{subequations}
and $f_1(k)$ is analytic at $k=-1$. Since the function $w_+(s)$, $s\in\Sigma_1$, is smooth and nonzero at $s=-1$, the first term on the right-hand side of \eqref{h-def-ter} is bounded for $k\in D_{\epsilon}(-1)\setminus\Sigma_1$. In general, the function $r(k)$ is discontinuous at $k=-1$, so that the function $\delta(s)$ oscillates rapidly as $s\in\Sigma_1$ approaches $-1$. Therefore, in order to analyze $F_{E_1}(k)-F_{\bar E_1}(k)$ for $k$ near $-1$, we first deform the contours in \eqref{FE1s-def}. Assume first that $k\in D_{\epsilon}(-1)$ with $\Re k>-1$. We know from \eqref{ab1} that $a_-=\ii\eul^{-\ii\phi}b_+$ on $\Sigma_1$ which, together with $\nu_{1-}=-\ii\nu_{1+}$, gives $\hat a_-=\eul^{-\ii\phi}\hat b_+$. We use this relation to eliminate $\hat a_-$ from \eqref{FE1s-def}. Then, deforming the contours so that they pass through $-1-\epsilon$, we find
\begin{align}
F_{E_1}(k)&=-\frac{1}{2\pi}\int_{-1-\epsilon}^{-1}\frac{\ln\bigl(\hat a(s)\hat b(s)\delta_+(s)^2\bigr)}{w(s)(s-k)}\dd s+f_2(k),\\
F_{\bar E_1}(k)&=\frac{1}{2\pi}\int_{-1-\epsilon}^{-1}\frac{\ln\bigl(\hat a^*(s)\hat b^*(s)\delta_-(s)^{-2}\bigr)}{w(s)(s-k)}\dd s+f_3(k),
\end{align}
where the functions $f_2(k)$ and $f_3(k)$ are analytic near $k=-1$. By assumption, $b$ is everywhere nonzero; in particular, $b_{\pm}(-1)\neq 0$. Hence, shrinking $\epsilon$ if necessary, we may assume that $r(s)\neq 0$ for $s\in(-1-\epsilon,-1)$. We find
\begin{align*}
F_{E_1}(k)-F_{\bar E_1}(k)
&=-\frac{1}{2\pi}\int_{-1-\epsilon}^{-1}\frac{\ln\bigl(\hat a(s)\hat b(s)\delta_+(s)^2\bigr)+\ln\bigl(\hat a^*(s)\hat b^*(s)\delta_-(s)^{-2}\bigr)}{w(s)(s-k)}\dd s+f_4(k)\\
&=-\frac{1}{2\pi}\int_{-1-\epsilon}^{-1}\frac{\ln\bigl(\abs{a(s)b(s)}^2(1+\abs{r(s)}^2)^2\bigr)}{w(s)(s-k)}\dd s+f_4(k)\\
&=-\frac{1}{2\pi}\int_{-1-\epsilon}^{-1}\frac{\ln(\abs{r(s)}^2)}{w(s)(s-k)}\dd s+f_4(k),
\end{align*}
where $f_4(k)$ is analytic at $k=-1$. Since $\ln(\abs{r(-1)}^2)$ is real, Lemma~\ref{lem-C} implies that the functions $\eul^{\pm\ii w(k)\left(F_{E_1}(k)-F_{\bar E_1}(k)\right)}$ are bounded as $k\to -1$. This shows that $\eul^{\pm\ii h}$ is bounded for $k\in D_{\epsilon}(-1)$ with $\Re k>-1$; the boundedness for $k\in D_{\epsilon}(-1)$ with $\Re k<-1$ can be proved in a similar way by deforming the contours so that they pass through $-1+\epsilon$. This completes the proof of the lemma.
\end{proof}
%-------------------%

%---------------------------------------------------------%
%:s.5
%---------------------------------------------------------%
\section{Solution of the model problem}  \label{sec:model}

%-------------------%
%:fig 13 = 5.1
%-------------------%
\begin{figure}[ht]
\centering\includegraphics[scale=.75]{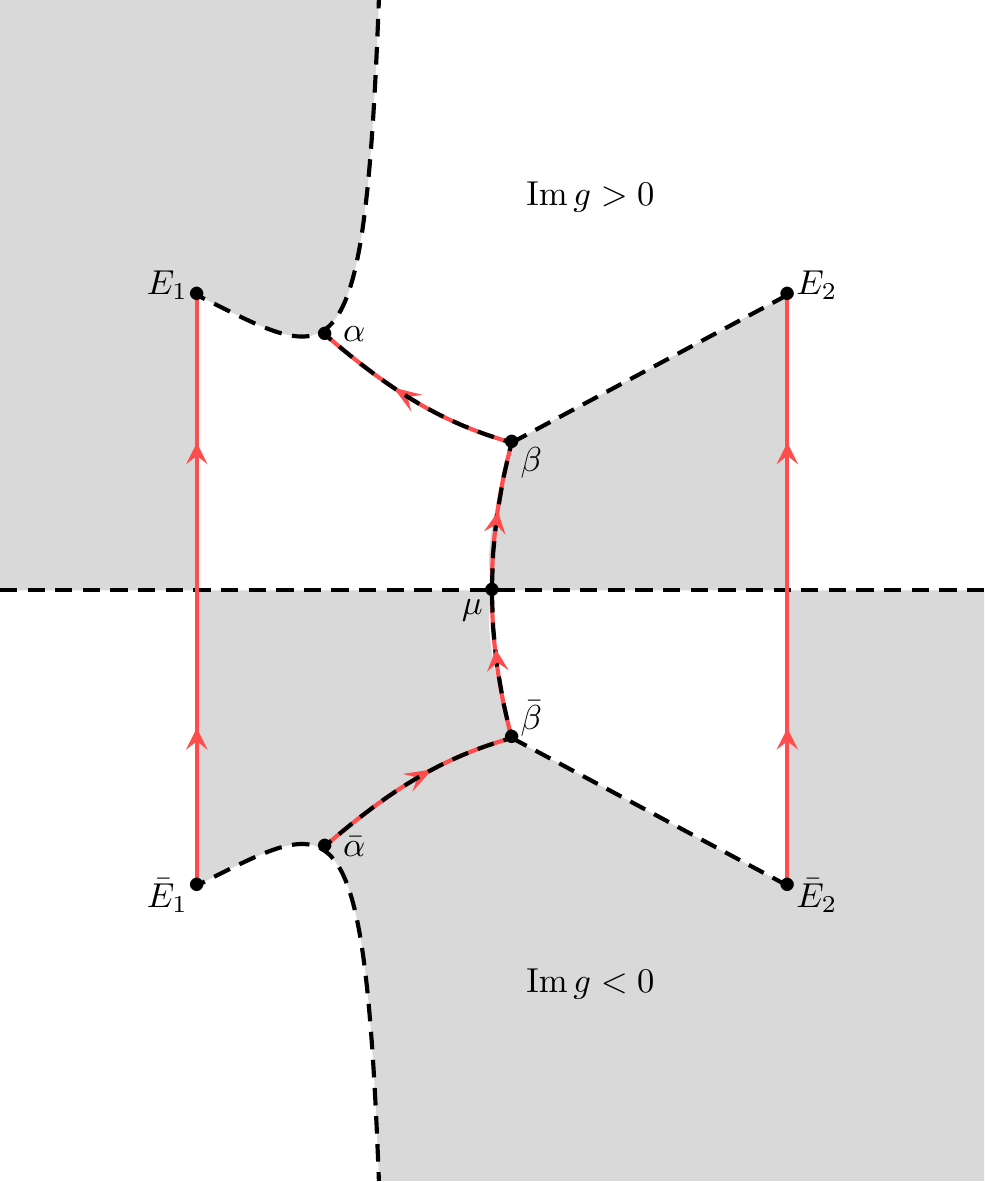}
\caption{The jump contour $\Sigma^{\model}$ (red with arrows indicating orientation) together with the set where $\Im g=0$ (dashed black).} 
\label{fig:jump-contour-model}
\end{figure}
%-------------------%

Recall that $\Sigma^{\model}\coloneqq\C{C}\cup\gamma_{(\bar\beta,\beta)}$ denotes the union of the four branch cuts $\Sigma_1$, $\Sigma_2$, $\gamma_{(\beta,\alpha)}$, $\gamma_{(\bar\alpha,\bar\beta)}$ and the curve $\gamma_{(\bar\beta,\beta)}$, oriented as in Figure~\ref{fig:jump-contour-model}. Away from $\Sigma^{\model}$, i.e., on $\Sigma^{(5)}\setminus\Sigma^{\model}$, the jump matrix $\hat v^{(5)}$ approaches the identity matrix as $t\to\infty$. This leads us to expect that in the limit $t\to\infty$, the solution $\hat m^{(5)}$ approaches the solution $m^{\model}$ of the RH problem
\begin{equation}   \label{rhp-model}
\begin{cases}
m^{\model}(x,t,\,\cdot\,)\in I+\dot E^2(\hat{\D{C}}\setminus\Sigma^{\model}),&\\
m_+^{\model}(x,t,k)=m_-^{\model}(x,t,k)v^{\model}(x,t,k)&\text{for a.e. }k\in\Sigma^{\model},
\end{cases}
\end{equation}
where $v^{\model}$ denotes the restriction of $\hat v^{(5)}$ to $\Sigma^{\model}$, i.e.,
\begin{equation}   \label{vmodel-def}
v^{\model}\coloneqq\begin{cases}
\begin{pmatrix}
0&\ii\eul^{-2\ii(t\Omega_1+\omega_1)}\\
\ii\eul^{2\ii(t\Omega_1+\omega_1)}&0\end{pmatrix},&k\in\Sigma_1,\\
\begin{pmatrix}
0&\ii\eul^{-2\ii(t\Omega_2+\omega_2)}\\
\ii\eul^{2\ii(t\Omega_2+\omega_2)}&0\end{pmatrix},&k\in\gamma_{(\beta,\alpha)}\cup\gamma_{(\bar\alpha,\bar\beta)},\\
\begin{pmatrix}
\eul^{2\ii(t\Omega_3+\omega_3)}&0\\0&\eul^{-2\ii(t\Omega_3+\omega_3)}\end{pmatrix},&k\in\gamma_{(\bar\beta,\beta)},\\
\begin{pmatrix}
0&\ii\\
\ii&0\end{pmatrix},&k\in\Sigma_2.
\end{cases}
\end{equation}
The jump matrix $v^{\model}$ is off-diagonal and independent of $k$ on each of the four branch cuts and it is diagonal and independent of $k$ on $\gamma_{(\bar\beta,\beta)}$. This implies that we can write down an explicit solution of the RH problem \eqref{rhp-model}-\eqref{vmodel-def} in terms of theta functions.

Define the function $\nu(k)\equiv\nu(\xi,k)$ for $k\in\D{C}\setminus\Sigma^{\model}$ by
\[
\nu(k)\coloneqq\left(\frac{(k-E_1)(k-\alpha)(k-\beta)(k-E_2)}{(k-\bar E_1)(k-\bar\alpha)(k-\bar\beta)(k-\bar E_2)}\right)^{\frac{1}{4}},\quad k\in\D{C}\setminus\Sigma^{\model},
\]
where the branch is fixed by requiring that $\nu=1+\ord(k^{-1})$ as $k\to\infty$. Let $\hat\nu\equiv\hat\nu(\xi,\,\cdot\,)$ denote the function $M\to\hat{\D{C}}$ which is given by $\nu^2$ on the upper sheet and by $-\nu^2$ on the lower sheet of $M$, that is, $\hat\nu(k^\pm)=\pm\nu(k)^2$ for $k\in\D{C}\setminus\Sigma^{\model}$. Then $\hat\nu$ is a meromorphic function on $M$. Noting, for example, that $\hat\nu$ has four simple zeros at $E_1$, $\alpha$, $\beta$, $E_2$, we see that $\hat\nu$ has degree four. Hence the function $\hat\nu-1$ has four zeros on $M$ counting multiplicity; we denote these zeros by $\infty^+,\,P_1,\,P_2,\,P_3\in M$. Let $D$ denote the divisor
\begin{equation}  \label{D-def}
D\coloneqq P_1P_2P_3
\end{equation}
on $M$. We define the vector $\C{K}\equiv\C{K}(\xi)\in\D{C}^3$ by
\begin{equation}
\C{K}\coloneqq\frac{1}{2}\left(\vece^{(1)}+\vece^{(3)}+\tau^{(1)}+\tau^{(2)}+\tau^{(3)}\right).
\end{equation}
We will show below that the projection of $\C{K}$ in the Jacobian variety is the vector of Riemann constants.

We define the complex vector $d\equiv d(\xi)\in\D{C}^3$ by
\begin{equation}  \label{d-def}
d\coloneqq\varphi(D)+\C{K},
\end{equation}
where $\varphi(D)\coloneqq\sum_1^3\varphi(P_j)$; if a zero $P_j$ has projection in $\Sigma^{\model}$, then we fix the value of $d$ by letting $\varphi(P_j)$ denote the boundary value $\varphi_+(P_j)$ from the left say. We also define the vector-valued function $v(t)\equiv v(\xi,t)$ by
\begin{equation}    \label{v-def}
v(t)\coloneqq -\frac{1}{\pi}\left(t\Omega_1+\omega_1,t\Omega_2+\omega_2,t(\Omega_2-\Omega_3)+\omega_2-\omega_3+\frac{\pi}{2}\right).
\end{equation}

%-------------------%
%:thm 5.1
%-------------------%
\begin{theorem}[Solution of the model RH problem] \label{rhp-thm}
For each choice of the six real constants $\accol{\Omega_j,\omega_j}_1^3$ and for each $t\geq 0$, the RH problem \eqref{rhp-model} has a unique solution $m^{\model}(x,t,k)$. Moreover, this solution satisfies
\begin{align}  \label{mmodel-at-infty}
\lim_{k\to\infty}k(m^{\model}(x,t,k))_{12}
&=-\frac{\ii}{2}\Im(E_1+E_2+\alpha+\beta)\notag\\
&\quad\times\frac{\Theta\left(\varphi(\infty^+)+d\right)\Theta\left(\varphi(\infty^+)-v(t)-d\right)}{\Theta\left(\varphi(\infty^+)-d\right)\Theta\left(\varphi(\infty^+)+v(t)+d\right)},
\end{align}
where the limit is uniform with respect to $\arg k\in\croch{0,2\pi}$.
\end{theorem}
%-------------------%

%-------------------%
\begin{proof}
Define the vector-valued function $\C{M}(k,u,e)\equiv\C{M}(x,t,k,u,e)$ by
\[
\C{M}(k,u,e)\coloneqq\left(\frac{\Theta\left(\varphi(k^+)+u+e\right)}{\Theta\left(\varphi(k^+)+e\right)},\frac{\Theta\left(\varphi(k^+)-u-e\right)}{\Theta\left(\varphi(k^+)-e\right)}\right),\quad k\in\D{C}\setminus\Sigma^{\model},\ \ u,e\in\D{C}^3.
\]
The function $P\mapsto\Theta(\varphi(P)-e)$ either vanishes identically on $M$ or has exactly three zeros counting multiplicity \cite{FK92}*{Theorem VI.2.4, p.~309}. In the following, we assume that $u\in\D{C}^3$ is arbitrary and that $e\in\D{C}^3$ is such that the functions $\Theta(\varphi(\,\cdot\,)\pm e)$ are not identically zero. Since the function $P\mapsto\Theta(\varphi(P)-c)$, $c\in\D{C}^3$, is invariant under continuation along the curve $a_j$ as a consequence of \eqref{Theta-relat}, we see that $\C{M}$ is a well-defined and analytic function of $k\in\D{C}\setminus\Sigma^{\model}$ except for possible singularities at the zeros of $\Theta(\varphi(k^+)\pm e)$.

%-------------------%
%:claim 1
%-------------------%
\begin{claim}   \label{claim-one}
The function $\C{M}$ satisfies the following jump conditions across $\Sigma^{\model}$:
\begin{equation}  \label{calM-jump}
\C{M}_+(k,u,e)=\C{M}_-(k,u,e)\times
\begin{cases}
\begin{pmatrix}
0&\eul^{2\pi\ii u_1}\\
\eul^{-2\pi\ii u_1}&0\end{pmatrix},&k\in\Sigma_1,\\
\begin{pmatrix}
0&\eul^{2\pi\ii u_2}\\
\eul^{-2\pi\ii u_2}&0\end{pmatrix},&k\in\gamma_{(\beta,\alpha)}\cup\gamma_{(\bar\alpha,\bar\beta)},\\
\begin{pmatrix}
0&1\\
1&0\end{pmatrix},&k\in\Sigma_2,\\
\begin{pmatrix}
\eul^{-2\pi\ii(u_2-u_3)}&0\\
0&\eul^{2\pi\ii(u_2-u_3)}\end{pmatrix},&k\in\gamma_{(\bar\beta,\beta)}.
\end{cases}
\end{equation}
\end{claim}
%-------------------%

%-------------------%
\begin{proof}[Proof of Claim~\ref{claim-one}]
It follows from \eqref{varphi-jump-1} and \eqref{Theta-relat} that
\begin{align*}
\C{M}_{1+}(k,u,e)
&\coloneqq\frac{\Theta\left(\varphi_+(k^+)+u+e\right)}{\Theta\left(\varphi_+(k^+)+e\right)}=\frac{\Theta\left(-\varphi_-(k^+)+\tau^{(1)}+u+e\right)}{\Theta\left(-\varphi_-(k^+)+\tau^{(1)}+e\right)}\\
&=\frac{\eul^{2\pi\ii(-(-\varphi_-(k^+)+u+e)_1-\frac{\tau_{11}}{2})}\Theta\left(-\varphi_-(k^+)+u+e\right)}{\eul^{2\pi\ii(-(-\varphi_-(k^+)+e)_1-\frac{\tau_{11}}{2})}\Theta\left(-\varphi_-(k^+)+e\right)}\\
&=\frac{\eul^{-2\pi\ii u_1}\Theta\left(\varphi_-(k^+)-u-e\right)}{\Theta\left(\varphi_-(k^+)-e\right)}\\
&=\eul^{-2\pi\ii u_1}\C{M}_{2-}(k,u,e),\qquad k\in\Sigma_1.
\end{align*}
Since $\C{M}_2(k,u,e)=\C{M}_1(k,-u,-e)$ it follows that
\[
\C{M}_{2+}(k,u,e)=\eul^{2\pi\ii u_1}\C{M}_{1-}(k,u,e),\qquad k\in\Sigma_1.
\]
This establishes \eqref{calM-jump} in the case of $k\in\Sigma_1$. The other cases are established in a similar way, the only difference being that the case of $k\in\gamma_{(\bar\beta,\beta)}$ utilizes \eqref{varphi-jump-2} instead of \eqref{varphi-jump-1}.
\end{proof}
%-------------------%

Define the $2\times 2$-matrix valued function $\C{N}(k,u,e)\equiv\C{N}(x,t,k,u,e)$ by 
\begin{equation}  \label{calN-def}
\C{N}(k,u,e)\coloneqq\frac{1}{2}\begin{pmatrix}
(\nu+\nu^{-1})\C{M}_1(k,u,e)&(\nu-\nu^{-1})\C{M}_2(k,u,e)\\
(\nu-\nu^{-1})\C{M}_1(k,u,-e)&(\nu+\nu^{-1})\C{M}_2(k,u,-e)\end{pmatrix},\quad k\in\D{C}\setminus\Sigma^{\model},\ \ u,e\in\D{C}^3.
\end{equation}

%-------------------%
%:claim 2
%-------------------%
\begin{claim}   \label{claim-two}
The function $\C{N}$ satisfies the following jump conditions across $\Sigma^{\model}$:
\begin{equation}   \label{calN-jump}
\C{N}_+(k,u,e)=\C{N}_-(k,u,e)\times
\begin{cases}
\begin{pmatrix}
0&\ii\eul^{2\pi\ii u_1}\\
\ii\eul^{-2\pi\ii u_1}&0\end{pmatrix},&k\in\Sigma_1,\\
\begin{pmatrix}
0&\ii\eul^{2\pi\ii u_2}\\
\ii\eul^{-2\pi\ii u_2}&0\end{pmatrix},&k\in\gamma_{(\beta,\alpha)}\cup\gamma_{(\bar\alpha,\bar\beta)},\\
\begin{pmatrix}
-\eul^{-2\pi\ii(u_2-u_3)}&0\\
0&-\eul^{2\pi\ii(u_2-u_3)}\end{pmatrix},&k\in\gamma_{(\bar\beta,\beta)},\\
\begin{pmatrix}
0&\ii\\
\ii&0\end{pmatrix},&k\in\Sigma_2.
\end{cases}
\end{equation}
\end{claim}
%-------------------%

%-------------------%
\begin{proof}[Proof of Claim 2]
The function $\nu(k)$ is analytic for $k\in\D{C}\setminus\Sigma^{\model}$ and satisfies the following jump condition across $\Sigma^{\model}$:
\[
\nu_+(k)=\nu_-(k)\times
\begin{cases}
\ii,&k\in\C{C},\\
-1,&k\in\gamma_{(\bar\beta,\beta)}.
\end{cases}
\]
The claim now follows from \eqref{calM-jump} and straightforward computation. For example,
\begin{align*}
\C{N}_+(k,u,e)&=\frac{1}{2}\begin{pmatrix}
\ii(\nu_--\nu_-^{-1})\eul^{-2\pi\ii u_1}\C{M}_{2-}(k,u,e)&\ii(\nu_-+\nu_-^{-1})\eul^{2\pi\ii u_1}\C{M}_{1-}(k,u,e)\\
\ii(\nu_-+\nu_-^{-1})\eul^{-2\pi\ii u_1}\C{M}_{2-}(k,u,-e)&\ii(\nu_--\nu_-^{-1})\eul^{2\pi\ii u_1}\C{M}_{1-}(k,u,-e)\end{pmatrix}\\
&=\C{N}_-(k,u,e)\begin{pmatrix}
0&\ii\eul^{2\pi\ii u_1}\\
\ii\eul^{-2\pi\ii u_1}&0\end{pmatrix},\quad k\in\Sigma_1,
\end{align*}
which establishes the jump across $\Sigma_1$.
\end{proof}
%-------------------%

Taking $u=v(t)$ in \eqref{calN-jump}, where $v(t)$ is defined in \eqref{v-def}, we see that the function $\C{N}(k,v(t),e)$ satisfies the correct jump condition $\C{N}_+=\C{N}_-v^{\model}$ across $\Sigma^{\model}$.

Let $L(M)\subset\D{C}^3$ be the lattice generated by the six vectors $\accol{\vece^{(j)},\tau^{(j)}}_1^3$, so that $\D{C}^3/L(M)$ is the Jacobian variety of $M$.

%-------------------%
%:claim 3
%-------------------%
\begin{claim}   \label{claim-three}
The projection of $\C{K}\in\D{C}^3$ in $\D{C}^3/L(M)$ is the vector of Riemann constants.
\end{claim}
%-------------------%

%-------------------%
\begin{proof}[Proof of Claim~\ref{claim-three}]
We infer from Figures~\ref{fig:canonical-basis} and \ref{fig:curves-aj-bj} that the following relations are valid in $\D{C}^3/L(M)$:
\begin{alignat*}{2}
&\varphi(\bar E_2)=0,&\quad&\varphi(E_2)=\frac{1}{2}\left(\vece^{(1)}+\vece^{(2)}+\vece^{(3)}\right),\\
&\varphi_\pm(\bar\beta)=\frac{1}{2}\left(\vece^{(1)}+\vece^{(2)}+\vece^{(3)}+\tau^{(3)}\right),&&\varphi_\pm(\beta)=\frac{1}{2}\left(\vece^{(1)}+\vece^{(2)}+\tau^{(3)}\right),\\
&\varphi(\bar\alpha)=\frac{1}{2}\left(\vece^{(1)}+\vece^{(2)}+\tau^{(2)}\right),&&\varphi(\alpha)=\frac{1}{2}\left(\vece^{(1)}+\tau^{(2)}\right),\\
&\varphi(\bar E_1)=\frac{1}{2}\left(\vece^{(1)}+\tau^{(1)}\right),&&\varphi(E_1)=\frac{1}{2}\tau^{(1)}.
\end{alignat*}
We see that $\varphi(\bar E_1)$, $\varphi(\bar\alpha)$, and $\varphi_\pm(\bar\beta)$ are odd half-periods. Hence $\Theta\circ\varphi$ has the zero divisor $\bar E_1\bar\alpha\bar\beta$ and the vector of Riemann constants is given by (see \cite{FK92}*{VII.1.2, p.~325})
\begin{equation}  \label{Riemann-constants}
\varphi(\bar E_1\bar\alpha\bar\beta)=\frac{1}{2}\left(\vece^{(1)}+\vece^{(3)}+\tau^{(1)}+\tau^{(2)}+\tau^{(3)}\right)\eqqcolon\C{K}\ \modulo\ L(M).
\end{equation}
This proves the claim.
\end{proof}
%-------------------%

According to the next claim, the assumption that $\Theta(\varphi(\,\cdot\,)\pm e)$ is not identically zero automatically holds for $e=d$,  where $d$ is defined in \eqref{d-def}.

%-------------------%
%:claim 4
%-------------------%
\begin{claim}   \label{claim-four}
The divisor $D$ defined in \eqref{D-def} is the zero divisor of the multivalued holomorphic function $\psi$ on $M$ defined by
\[
\psi\colon P\mapsto\Theta(\varphi(P)-d). 
\]
In particular, $\Theta(\varphi(\infty^+)\pm d)\neq 0$.
\end{claim}
%-------------------%

%-------------------%
\begin{proof}[Proof of Claim~\ref{claim-four}]
Let $\infty,k_1,k_2,k_3$ denote the natural projections of the zeros $\infty^+,P_1,P_2,P_3$ of $\hat\nu-1$ in $\hat{\D{C}}$. We first show that the points $\infty,k_1,k_2,k_3$ are all distinct. Since
\begin{equation}  \label{nu-at-infty}
\nu(k)=1-\frac{\ii}{2k}\Im(E_1+E_2+\alpha+\beta)+\ord(k^{-2}),\quad k\to\infty,
\end{equation}
the zero $\infty^+$ of $\hat\nu-1$ is simple. Since $\hat\nu(\infty^-)=-1$, this shows that the $k_j$ are all finite numbers. To show that the $k_j$ are distinct, we note that $\hat\nu\neq 1$ at the branch points; hence no $k_j$ is a branch point. Furthermore, the function $\hat\nu$ obeys the symmetries $\hat\nu(k^+)=\overline{\hat\nu(\bar k^+)^{-1}}$ and $\hat\nu(k^+)=-\hat\nu(k^-)$ for $k\in\D{C}\setminus\Sigma^{\model}$. Hence the divisor $D$ is invariant under complex conjugation. Since $D$ has order three, it follows that if $P_0$ is a double or triple zero of $\hat\nu-1$, then the natural projection $k_0=k(P_0)$ must be real. However, writing $E_3\coloneqq\alpha$ and $E_4\coloneqq\beta$, we have
\[
\hat\nu(k^\pm)=\pm\sqrt{\prod_{j=1}^4\frac{k-E_j}{k-\bar E_j}}\,,
\]
and hence
\[
\frac{\dd}{\dd k}\hat\nu(k^\pm)=\frac{\hat\nu(k^\pm)}{2}\sum_{j=1}^4\left(\frac{1}{k-E_j}-\frac{1}{k-\bar E_j}\right).
\]
For $k$ real, this gives
\[
\frac{\dd}{\dd k}\hat\nu(k^\pm)=\frac{\hat\nu(k^\pm)}{2}\sum_{j=1}^4\frac{2\ii\Im E_j}{\abs{k-E_j}^2},\quad k\in\D{R},
\]
where $\abs{\hat\nu(k^\pm)}=1$. Since all $E_j$ have strictly positive imaginary part, we conclude that $\frac{\dd}{\dd k}\hat\nu(k^\pm)\neq 0$ for all $k\in\D{R}$. We conclude that $\hat\nu-1$ has no double or triple zero on $M$. If $\hat\nu-1$ has a zero at $P_j$, then the symmetry $\hat\nu(k^+)=-\hat\nu(k^-)$ shows that $\varsigma(P_j)$ is not a zero, where $\varsigma(k^\pm)\coloneqq k^\mp$ denotes the sheet-changing involution. This shows that the points $\infty,k_1,k_2,k_3\in\hat{\D{C}}$ are all distinct.

Let $i(D)$ denote the index of specialty of the divisor $D$. The function $\psi(P)=\Theta(\varphi(P)-\varphi(D)-\C{K})$ is not identically zero on $M$ iff $i(D)=0$, and if this is the case then $D$ is the divisor of zeros of $\psi$ (cf.~\cite{FK92}*{Theorem VI.3.3 b, p.~313}). Proceeding as in \cite{FK92}*{VII.1.1, p.~324}, we see that $i(D)=0$ provided that the matrix
\[
\begin{pmatrix}
1&k_1&k_1^2\\
1&k_2&k_2^2\\
1&k_3&k_3^2
\end{pmatrix}
\]
has full rank. But the determinant of this matrix equals $-(k_1-k_2)(k_1-k_3)(k_2-k_3)$, which is nonzero because no two of the $k_j$ are equal. This shows that $i(D)=0$; hence $D$ is the zero divisor of $\psi$.

We have shown that $\Theta(\varphi(P)-d)$ has zero divisor $D$. Since $\Theta$ is even, the symmetry \eqref{zeta-symmetry} implies that $\Theta(\varphi(P)+d)=\Theta(\varphi(\varsigma(P))-d)$ has zero divisor $\varsigma(D)$. Since $D$ contains neither of the points $\infty^\pm$ as a factor, we conclude that $\Theta(\varphi(\infty^+)\pm d)\neq 0$.
\end{proof}
%-------------------%

%-------------------%
%:claim 5
%-------------------%
\begin{claim}   \label{claim-five}
We have $\varphi(\infty^+)+d=0$ in $\D{C}^3/L(M)$. Hence,
\[
\Theta(\varphi(\infty^+)\pm v(t)+d)\neq 0,\quad t\geq 0.
\]
\end{claim}
%-------------------%

%-------------------%
\begin{proof}[Proof of Claim~\ref{claim-five}]
By \eqref{d-def}, we have $\varphi(\infty^+)+d=\varphi(\infty^+D)+\C{K}$, where $\infty^+D$ is the zero divisor of the meromorphic function $\hat\nu-1$ on $M$. Now $\hat\nu$ has pole divisor $\bar E_1\bar\alpha\bar\beta\bar E_2$. Hence Abel's theorem and equation \eqref{Riemann-constants} together with $\varphi(\bar E_2)=0$ give
\[
\varphi(\infty^+D)=\varphi(\bar E_1\bar\alpha\bar\beta\bar E_2)=\C{K}=-\C{K}\ \modulo\ L(M).
\]
Thus $\varphi(\infty^+)+d=0$ in $\D{C}^3/L(M)$, and hence $\Theta$ vanishes at $\varphi(\infty^+)\pm v(t)+d$ iff $\Theta(\pm v(t))=\Theta(v(t))$ vanishes. But $\Theta(v(t))$ does not vanish for any $t\geq 0$ because $v(t)\in\D{R}^3$ and $\Theta(u)\neq 0$ for any $u\in\D{R}^3$ (see \cite{F73}*{Corollary 6.13, p.~122}).
\end{proof}
%-------------------%

%-------------------%
%:claim 6
%-------------------%
\begin{claim}   \label{claim-six}
For each choice of the six real constants $\accol{\Omega_j,\omega_j}_1^3$ and for each $t\geq 0$, the function $m^{\model}(x,t,k)$ defined by
\begin{equation}   \label{mmodel-def}
m^{\model}(x,t,k)\coloneqq\C{N}(\infty,v(t),d)^{-1}\C{N}(k,v(t),d),\quad k\in\hat{\D{C}}\setminus\Sigma^{\model},
\end{equation}
is the unique solution of the RH problem \eqref{rhp-model}.
\end{claim}
%-------------------%

%-------------------%
\begin{proof}[Proof of Claim~\ref{claim-six}]
The solution is unique because $\det v^{\model}=1$. We will show that the function $m^{\model}(x,t,\,\cdot\,)$ is well-defined and belongs to $I+\dot E^2(\hat{\D{C}}\setminus\Sigma^{\model})$. Since we have already seen that $\C{N}(k,v(t),d)$ has the correct jump across $\Sigma^{\model}$, this will complete the proof of the claim.

Define the multivalued meromorphic functions $\C{P}_j(P)\equiv\C{P}_j(x,t,P)$, $j=1,2$, on $M$ by
\begin{align*}
\C{P}_1(P)&\coloneqq(\hat\nu(P)-1)\frac{\Theta(\varphi(P)-v(t)-d)}{\Theta(\varphi(P)-d)},\\
\C{P}_2(P)&\coloneqq(\hat\nu(P)-1)\frac{\Theta(\varphi(P)+v(t)-d)}{\Theta(\varphi(P)-d)}\,.
\end{align*}
Then, using the symmetry $\varphi(k^+)=-\varphi(k^-)$,
\[
\C{N}(k,v(t),d)=\frac{1}{2\nu(k)}\begin{pmatrix}-\C{P}_1(k^-)&\C{P}_1(k^+)\\\C{P}_2(k^+)&-\C{P}_2(k^-)\end{pmatrix},\quad k\in\hat{\D{C}}\setminus\Sigma^{\model}.
\]
The holomorphic functions $k\mapsto\Theta(\varphi(k^\pm)\pm v(t)-d)$ are bounded on $\hat{\D{C}}\setminus\Sigma^{\model}$. Also, by Claim~\ref{claim-four}, the denominator $\Theta(\varphi(P)-d)$ has zero divisor $D$ on $M$, which is a factor of that of $\hat\nu-1$. It follows that
\begin{equation}  \label{calN-estimate}
\abs{\C{N}(k,v(t),d)}\leq C(\abs{\nu(k)}+\abs{\nu(k)^{-1}}),\quad k\in\hat{\D{C}}\setminus\Sigma^{\model},
\end{equation}
where $C\equiv C(x,t)$ is independent of $k$. In particular, $\C{N}(k,v(t),d)$ is an analytic function of $k\in\hat{\D{C}}\setminus\Sigma^{\model}$ which is bounded away from the eight branch points. If $k_0$ denotes one of the eight branch points, then \eqref{calN-estimate} shows that
\begin{equation}
\abs{\C{N}(k,v(t),d)}\leq C\abs{k-k_0}^{-1/4}
\end{equation}
as $k\in\hat{\D{C}}\setminus\Sigma^{\model}$ approaches $k_0$. Letting $k\to\infty$ in the definition~\eqref{calN-def} of $\C{N}$, we obtain
\begin{align*}
\C{N}(\infty,v(t),d)&=\begin{pmatrix}
\C{M}_1(\infty,v(t),d)&0\\
0&\C{M}_2(\infty,v(t),-d)\end{pmatrix}\notag\\
&=\begin{pmatrix}
\frac{\Theta(\varphi(\infty^+)+v(t)+d)}{\Theta(\varphi(\infty^+)+d)}&0\\
0&\frac{\Theta(\varphi(\infty^+)-v(t)+d)}{\Theta(\varphi(\infty^+)+d)}\end{pmatrix}.
\end{align*}
By Claims~\ref{claim-four} and \ref{claim-five}, the values of $\Theta(\varphi(\infty^+)+d)$ and $\Theta(\varphi(\infty^+)\pm v(t)+d)$ are finite and nonzero. It follows that $\C{N}(\infty,v(t),d)$ is invertible. This shows that $m^{\model}(x,t,\,\cdot\,)$ is well-defined by \eqref{mmodel-def} and belongs to $I+\dot E^2(\hat{\D{C}}\setminus\Sigma^{\model})$.
\end{proof}
%-------------------%

%-------------------%
%:claim 7
%-------------------%
\begin{claim}   \label{claim-seven}
The solution $m^{\model}(x,t,k)$ defined in \eqref{mmodel-def} satisfies \eqref{mmodel-at-infty}.
\end{claim}
%-------------------%

%-------------------%
\begin{proof}[Proof of Claim~\ref{claim-seven}]
We have
\begin{align*}
&\lim_{k\to\infty}km_{12}^{\model}(x,t,k)=\frac{1}{\C{M}_1(\infty,v(t),d)}\lim_{k\to\infty}k\C{N}_{12}(k,v(t),d)\\
&=\frac{\Theta(\varphi(\infty^+)+d)}{\Theta(\varphi(\infty^+)+v(t)+d)}\lim_{k\to\infty}\frac{k(\nu(k)-\nu(k)^{-1})}{2}\,\frac{\Theta(\varphi(k^+)-v(t)-d)}{\Theta(\varphi(k^+)-d)}.
\end{align*}
Since, by \eqref{nu-at-infty},
\[
\nu(k)-\frac{1}{\nu(k)}=-\frac{\ii}{k}\Im(E_1+E_2+\alpha+\beta)+\ord(k^{-2}),\quad\text{uniformly as }k\to\infty,
\]
equation \eqref{mmodel-at-infty} follows.
\end{proof}
%-------------------%
\renewcommand{\qed}{}
\end{proof}
%-------------------%

%-------------------%
%:rem 5.2
%-------------------%
\begin{remark}
Although the function $\C{N}$ depends on the value of $d$ in $\D{C}^3$, the function $m^{\model}$ defined in \eqref{mmodel-def} only depends on the projection of $d$ in $\D{C}^3/L(M)$. Indeed, $\C{N}$ satisfies
\[
\C{N}(k,u,d+\vece^{(j)})=\C{N}(k,u,d),\quad\C{N}(k,u,d+\tau^{(j)})=\eul^{-2\ii\pi u_j\sigma_3}\C{N}(k,u,d),\quad j=1,2,3.
\]
\end{remark}
%-------------------%

%---------------------------------------------------------%
%:s.6
%---------------------------------------------------------%
\section{Local models}  \label{sec:local}

The many transformations presented in Section~\ref{sec:main} resulted in a RH problem for $\hat m^{(5)}$ with the property that the matrix $\hat v^{(5)}-I$ decays to zero for $k\in\Sigma^{(5)}\setminus\Sigma^{\model}$ as $t\to\infty$. However, this decay is not uniform with respect to $k$ as $k$ approaches $\Sigma^{\model}$. Thus, on the parts of the contour $\Sigma^{(5)}$ that lie near $\Sigma^{\model}$, we need to introduce local solutions that are better approximations of $\hat m^{(5)}$ than $m^{\model}$ is. Using these local approximations, we can derive appropriate error estimates as well as higher order asymptotics beyond the $\ord(1)$ term.

Recall that $\alpha$, $\beta$, $\bar\alpha$, $\bar\beta$, and $\mu$ are distinct. Let $\epsilon\equiv\epsilon(\xi)$ be so small that the five disks
\begin{equation}   \label{disks}
D_{\epsilon}(\alpha),\ \ D_{\epsilon}(\beta),\ \ D_{\epsilon}(\mu),\ \ D_{\epsilon}(\bar\alpha),\text{ and }\ D_{\epsilon}(\bar\beta)
\end{equation}
are disjoint from each other and from the cuts $\Sigma_1$ and $\Sigma_2$. In the following three subsections, we define three local solutions, denoted by $m^{\alpha}(x,t,k)$, $m^{\beta}(x,t,k)$, and $m^{\mu}(x,t,k)$, which are good approximations of $\hat m^{(5)}$ for $k$ in the three disks $D_{\epsilon}(\alpha)$, $D_{\epsilon}(\beta)$, and $D_{\epsilon}(\mu)$, respectively.

%---------------------------------------------------------%
%:s.6.1
%---------------------------------------------------------%
\subsection{Local model near $\BS{\alpha}$}

We define the function $m^{(\alpha 0)}$ for $k$ near $\alpha$ by 
\begin{equation}   \label{malpha-0}
m^{(\alpha 0)}(x,t,k)\coloneqq\hat m^{(5)}(x,t,k)\eul^{-\ii(\frac{\ii}{2}\ln(-\delta(k)^2\hat a(k)\hat b(k))+tg(\alpha)+h(k))\sigma_3},\quad k\in D_{\epsilon}(\alpha)\setminus\Sigma^{(5)}.
\end{equation}
Lemmas~\ref{delta-properties} and \ref{h-properties} imply that the exponential in \eqref{malpha-0} is bounded and analytic for $k\in D_{\epsilon}(\alpha)\setminus\Sigma^{(5)}$. We also define the function $g_{\alpha}(k)\equiv g_{\alpha}(\xi,k)$ for $k$ near $\alpha$ by
\[
g_{\alpha}(k)\coloneqq g(k)-g(\alpha),\quad k\in D_{\epsilon}(\alpha)\setminus\gamma_{(\beta,\alpha)}.
\]
Let $\accol{\C{S}_j}_1^4$ denote the open subsets of $D_{\epsilon}(\alpha)$ shown in Figure~\ref{fig:contour-near-alpha}. Let $\C{Y}_j=\bar{\C{S}}_{j-1}\cap\bar{\C{S}}_j$, $j=1,\dots,4$, $\bar{\C{S}}_0\equiv\bar{\C{S}}_4$, denote the curves separating the $\C{S}_j$ oriented toward $\alpha$ as in Figure~\ref{fig:contour-near-alpha}.

It follows from the expression \eqref{jump-v5} for the jump matrix $\hat v^{(5)}$ that $m^{(\alpha 0)}$ satisfies the following jump condition on the part of $\Sigma^{(5)}$ that lies in $D_{\epsilon}(\alpha)$:
\[
m_+^{(\alpha 0)}(x,t,k)=m_-^{(\alpha 0)}(x,t,k)v^{(\alpha 0)}(x,t,k),\quad k\in\Sigma^{(5)}\cap D_{\epsilon}(\alpha),
\]
where 
\[
v^{(\alpha 0)}=
\begin{cases}
\begin{pmatrix}
1&-\eul^{-2\ii tg_{\alpha}}\\0&1\end{pmatrix},&k\in\C{Y}_1,\\
\begin{pmatrix}
1&0\\\eul^{2\ii tg_{\alpha}}&1\end{pmatrix},&k\in\C{Y}_2\cup\C{Y}_4,\\
\begin{pmatrix}
0&1\\-1&0\end{pmatrix},&k\in\C{Y}_3,
\end{cases}
\]
with
\[
g_{\alpha}(k)\coloneqq g(k)-g(\alpha),\quad k\in D_{\epsilon}(\alpha)\setminus\gamma_{(\beta,\alpha)}.
\]
%-------------------%
%:fig 14 = 6.1
%-------------------%
\begin{figure}[ht]
\centering\includegraphics[scale=.9]{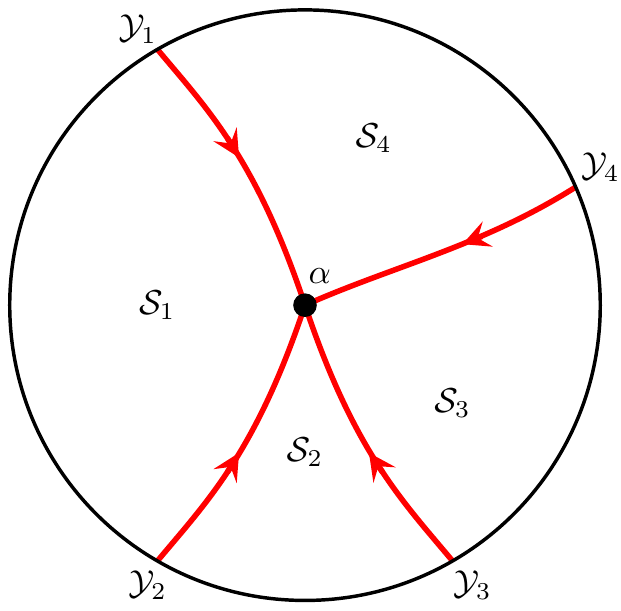}
\caption{The contour in the $\epsilon$-neighborhood $D_{\epsilon}(\alpha)$ of $\alpha$ and the sets $\accol{\C{S}_j}_1^4$.} 
\label{fig:contour-near-alpha}
\end{figure}
%-------------------%

In order to relate $m^{(\alpha 0)}$ to the Airy solution $m^{\Airy}$ of Appendix~\ref{sec:B}, we make a local change of variables for $k$ near $\alpha$ and introduce the new variable $\zeta\equiv\zeta(k)$ by
\[
\zeta\coloneqq\left(\frac{3\ii t}{2}g_{\alpha}(k)\right)^{2/3}.
\]
Since $g_{\alpha}(k)=\int_{\alpha}^k\sqrt{s-\alpha}\,f(s)\dd s$ where $f(s)$ is analytic and nonzero at $s=\alpha$, we can write (shrinking $\epsilon$ if necessary)
\begin{equation}  \label{zeta-k-near-alpha}
\zeta\coloneqq\left(\frac{3t}{2}\right)^{2/3}(k-\alpha)\psi_{\alpha}(k),
\end{equation}
where the function $\psi_{\alpha}(k)\equiv\psi_{\alpha}(\xi,k)$ is analytic for $k\in D_{\epsilon}(\alpha)$ and $\psi_{\alpha}(\alpha)\neq 0$. Equation \eqref{zeta-k-near-alpha} shows that the map $k\mapsto\zeta$ is a biholomorphism from the disk $D_{\epsilon}(\alpha)$ onto a neighborhood of the origin. Since $\Im g_{\alpha}=0$ on $\gamma_{(\beta,\alpha)}$, by choosing a different branch in the definition of $\zeta$ if necessary, we may assume that the map $k\mapsto\zeta$ sends $\C{Y}_3$ into $\D{R}_-$. Actually, by deforming the contours $\C{Y}_1$, $\C{Y}_2$, and $\C{Y}_4$ if necessary, we may assume that it maps $\C{Y}_j$ and $\C{S}_j$ into $Y_j$ and $S_j$, respectively, for all $j$, where $Y_j$ and $S_j$ are as in Figure~\ref{fig:rays-Y}.

Since $\zeta$ satisfies
\[
\frac{4}{3}\zeta^{3/2}=2\ii tg_{\alpha}(k),
\]
we see that $m^{(\alpha 0)}(k)$ has the same jumps as $m^{\Airy}(\zeta(k))$ near $\alpha$. Hence we seek a parametrix $m^{\alpha}$ for $\hat m^{(5)}$ near $\alpha$ of the form
\begin{equation}  \label{malpha-def}
m^{\alpha}(x,t,k)=Y_{\alpha}(x,t,k)m^{\Airy}(\zeta(k))\eul^{\ii(\frac{\ii}{2}\ln(-\delta^2\hat a\hat b)+tg(\alpha)+h(k))\sigma_3},\quad k\in D_{\epsilon}(\alpha)\setminus\Sigma^{(5)},
\end{equation}
where $Y_{\alpha}(x,t,k)$ is analytic in $D_{\epsilon}(\alpha)$. To ensure that $m^{\alpha}$ is a good approximation of $\hat m^{(5)}$, we want to choose $Y_{\alpha}(x,t,k)$ so that $m^{\alpha}(m^{\model})^{-1}\to I$ on $\partial D_{\epsilon}(\alpha)$ as $t\to\infty$. Hence we choose
\begin{equation}  \label{Yalpha-def}
Y_{\alpha}(x,t,k)\coloneqq m^{\model}\eul^{-\ii(\frac{\ii}{2}\ln(-\delta^2\hat a\hat b)+tg(\alpha)+h(k))\sigma_3}(m_{\asympt,N}^{\Airy,\inv}(\zeta(k))), 
\end{equation}
where $N\geq 0$ is some integer and $m_{\asympt,N}^{\Airy,\inv}$ is the approximation of $(m^\Airy)^{-1}$ defined in \eqref{mAiry-approx-defb}. Equation \eqref{mAiry-approx-jumpb} implies that $m_{\asympt,N}^{\Airy,\inv}(\zeta(k))$ is analytic near $\alpha$ except for a jump across $\C{Y}_3$ given by
\begin{equation}  \label{mAiry-approx-jump-bis}
(m_{\asympt,N}^{\Airy,\inv}(\zeta(k)))_+=\begin{pmatrix}0&-1\\1&0\end{pmatrix}(m_{\asympt,N}^{\Airy,\inv}(\zeta(k)))_-,\quad k\in\C{Y}_3.
\end{equation}
Equation \eqref{vmodel-def} shows that the jump of $m^{\model}(k)\eul^{-\ii(\frac{\ii}{2}\ln(-\delta^2\hat a\hat b)+tg(\alpha)+h(k))\sigma_3}$ across $\C{Y}_3$ cancels the jump in \eqref{mAiry-approx-jump-bis}, so that the function $Y_{\alpha}$ is analytic in $D_{\epsilon}(\alpha)$.

The functions $m^{\model}$, $\eul^{-\ii(\frac{\ii}{2}\ln(-\delta^2\hat a\hat b)+tg(\alpha)+h)\sigma_3}$, and their inverses are bounded on the circle $\partial D_{\epsilon}(\alpha)$ as a consequence of Lemmas~\ref{delta-properties} and \ref{h-properties}. Moreover, shrinking $\epsilon$ if necessary, we may assume that
\[
\inf_{k\in\partial D_{\epsilon}(\alpha)}\abs{g_{\alpha}(k)}>0. 
\]
The asymptotic formula \eqref{mAiry-approx-mAiryb} then implies
\begin{equation}   \label{malpha-mmodel-t-infty}
m^{\alpha}(k)m^{\model}(k)^{-1}=I+\ord(t^{-N-1}),\quad t\to\infty,\quad k\in\partial D_{\epsilon}(\alpha),
\end{equation}
uniformly with respect to $k$ in the given range. We have proved the following lemma.

%-------------------%
%:lem 6.1
%-------------------%
\begin{lemma}
The function $m^{\alpha}(x,t,k)$ defined in \eqref{malpha-def} is an analytic and bounded function of $k\in D_{\epsilon}(\alpha)\setminus\Sigma^{(5)}$, which satisfies the same jump conditions as $\hat m^{(5)}$ across $\Sigma^{(5)}\cap D_{\epsilon}(\alpha)$. The function $Y_{\alpha}(x,t,k)$ defined in \eqref{Yalpha-def} is an analytic and bounded function of $k\in D_{\epsilon}(\alpha)$. As $t\to\infty$, $m^{\alpha}$ satisfies \eqref{malpha-mmodel-t-infty}.
\end{lemma}
%-------------------%

%---------------------------------------------------------%
%:s.6.2
%---------------------------------------------------------%
\subsection{Local model near $\BS{\beta}$}

We define the function $m^{(\beta 0)}$ for $k$ near $\beta$ by
\begin{equation}  \label{mbeta-0-def}
m^{(\beta 0)}(x,t,k)\coloneqq\hat m^{(5)}(x,t,k)\eul^{-\ii(\frac{\ii}{2}\ln(\delta(k)^2\hat a(k)\hat b(k))+h(k))\sigma_3},\quad k\in D_{\epsilon}(\beta)\setminus\Sigma^{(5)}.
\end{equation}
Lemmas~\ref{delta-properties} and \ref{h-properties} imply that the exponential in \eqref{mbeta-0-def} is bounded and analytic for $k\in D_{\epsilon}(\beta)\setminus\Sigma^{(5)}$.

%-------------------%
%:fig 15 = 6.2
%-------------------%
\begin{figure}[ht]
\centering\includegraphics[scale=.9]{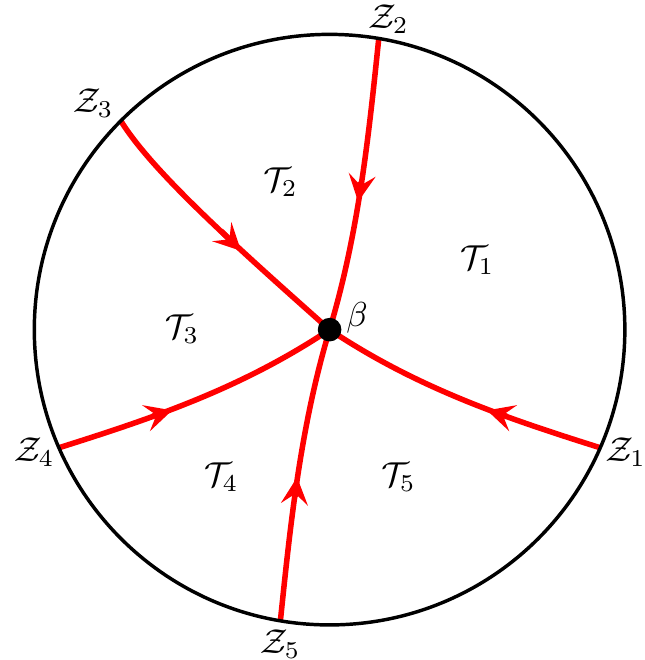}
\caption{The contour in the $\epsilon$-neighborhood $D_{\epsilon}(\beta)$ of $\beta$ and the sets $\accol{\C{T}_j}_1^5$.} 
\label{fig:contour-near-beta}
\end{figure}
%-------------------%

Let $\accol{\C{T}_j}_1^5$ denote the open subsets of $D_{\epsilon}(\beta)$ shown in Figure~\ref{fig:contour-near-beta}. Let $\C{Z}_j\coloneqq\bar{\C{T}}_{j-1}\cap\bar{\C{T}}_j$, $j=1,\dots,5$, $\bar{\C{T}}_0\equiv\bar{\C{T}}_5$, denote the curves separating the $\C{T}_j$ oriented toward $\beta$ as in Figure~\ref{fig:contour-near-beta}.

It follows from the expression \eqref{jump-v5} for the jump matrix $\hat v^{(5)}$ that $m^{(\beta 0)}$ satisfies the following jump condition on the part of $\Sigma^{(5)}$ that lies in $D_{\epsilon}(\beta)$:
\[
m_+^{(\beta 0)}(x,t,k)=m_-^{(\beta 0)}(x,t,k)v^{(\beta 0)}(x,t,k),\quad k\in\Sigma^{(5)}\cap D_{\epsilon}(\beta),
\]
where
\[
v^{(\beta 0)}\coloneqq\begin{cases}
\begin{pmatrix}1&-\frac{1}{aa^*}\eul^{-2\ii tg}\\0&1\end{pmatrix},&k\in\C{Z}_1,\\
\begin{pmatrix}1&0\\aa^*\eul^{2\ii tg}&1\end{pmatrix},&k\in\C{Z}_2,\\
\begin{pmatrix}0&\eul^{-\ii t(g_++g_-)}\\-\eul^{\ii t(g_++g_-)}&0\end{pmatrix},&k\in\C{Z}_3,\\
\begin{pmatrix}1&0\\\frac{1}{aa^*}\eul^{2\ii tg}&1\end{pmatrix},&k\in\C{Z}_4,\\
\begin{pmatrix}\frac{\eul^{\ii t(g_+-g_-)}}{aa^*}&0\\0&aa^*\eul^{-\ii t(g_+-g_-)}\end{pmatrix},&k\in\C{Z}_5.
\end{cases}
\]

We define the function $g_{\beta}(k)\equiv g_{\beta}(\xi,k)$ for $k$ near $\beta$ by
\[
g_{\beta}(k)\coloneqq\int_{\beta}^k\dd g=
\begin{cases}
g(k)-g_-(\beta),&k\in\C{T}_1\cup\C{T}_2\cup\C{T}_5,\\
g(k)-g_+(\beta),&k\in\C{T}_3\cup\C{T}_4,
\end{cases}
\]
where the integration contour lies in $D_{\epsilon}(\beta)\setminus(\gamma_{(\bar\beta,\beta)}\cup\gamma_{(\beta,\alpha)})$. Let
\[
m^{(\beta 1)}(x,t,k)\coloneqq m^{(\beta 0)}(x,t,k)A(k),\quad k\in D_{\epsilon}(\beta)\setminus\Sigma^{(5)},
\]
where $A(k)\equiv A(\xi,k)$ denotes the sectionally holomorphic function
\[
A(k)\coloneqq
\begin{cases}
(aa^*)^{-\sigma_3/2}\eul^{-\ii tg_-(\beta)\sigma_3},&k\in\C{T}_1\cup\C{T}_2\cup\C{T}_5,\\
(aa^*)^{\sigma_3/2}\eul^{-\ii tg_+(\beta)\sigma_3},&k\in\C{T}_3\cup\C{T}_4.
\end{cases}
\]
Using that $g_++g_-=g_+(\beta)+g_-(\beta)$ on $\gamma_{(\beta,\alpha)}$ and $g_+-g_-=g_+(\beta)-g_-(\beta)$ on $\gamma_{(\mu,\beta)}$, we find that $m^{(\beta 1)}$ satisfies the jump condition $m_+^{(\beta 1)}=m_-^{(\beta 1)}v^{(\beta 1)}$ with
\[
v^{(\beta 1)}\coloneqq\begin{cases}
\begin{pmatrix}1&-\eul^{-2\ii tg_{\beta}}\\0&1\end{pmatrix},&k\in\C{Z}_1,\\
\begin{pmatrix}1&0\\\eul^{2\ii tg_{\beta}}&1\end{pmatrix},&k\in\C{Z}_2\cup\C{Z}_4,\\
\begin{pmatrix}0&1\\-1&0\end{pmatrix},&k\in\C{Z}_3,\\
\,I,&k\in\C{Z}_5.
\end{cases}
\]
In order to relate $m^{(\beta 0)}$ to the solution $m^{\Airy}$ of Appendix~\ref{sec:B}, we make a local change of variables for $k$ near $\beta$ and introduce the new variable $\zeta\equiv\zeta(k)$ by
\[
\zeta\coloneqq\left(\frac{3\ii t}{2}g_{\beta}(k)\right)^{2/3}.
\]
As in \eqref{zeta-k-near-alpha}, we can write
\begin{equation}   \label{zeta-k-near-beta}
\zeta=\left(\frac{3t}{2}\right)^{2/3}(k-\beta)\psi_{\beta}(k),
\end{equation}
where the function $\psi_{\beta}(k)\equiv\psi_{\beta}(\xi,k)$ is analytic for $k\in D_{\epsilon}(\beta)$ and $\psi_{\beta}(\beta)\neq 0$. Replacing $\psi_{\beta}$ by $\eul^{\frac{2\pi\ii}{3}}\psi_{\beta}$ or $\eul^{\frac{4\pi\ii}{3}}\psi_{\beta}$ if necessary, we may assume that the map $k\mapsto\zeta$ maps $\C{Z}_3$ into $\D{R}_-$. Furthermore, deforming the contours $\C{Z}_1$, $\C{Z}_2$, and $\C{Z}_4$ if necessary, we may assume that it maps $\C{Z}_j$ into $Y_j$ for $j=1,\dots,4$, where $Y_j$ are the rays in Figure~\ref{fig:rays-Y}.

Since $\zeta$ satisfies
\[
\frac{4}{3}\zeta^{3/2}=2\ii tg_{\beta}(k),
\]
we see that $m^{(\beta 1)}(k)$ has the same jumps as $m^{\Airy}(\zeta(k))$ near $\beta$. Hence we seek a parametrix $m^{\beta}$ for $\hat m^{(5)}$ near $\beta$ of the form
\begin{equation}  \label{mbeta-def}
m^{\beta}(x,t,k)=Y_{\beta}(x,t,k)m^{\Airy}(\zeta(k))A(k)^{-1}\eul^{\ii(\frac{\ii}{2}\ln(\hat a\hat b\delta^2)+tg(\beta)+h(k))\sigma_3},\quad k\in D_\epsilon(\beta)\setminus\Sigma^{(5)},
\end{equation}
where $Y_{\beta}(x,t,k)$ is analytic in $D_{\epsilon}(\beta)$. To ensure that $m^{\beta}$ is a good approximation of $\hat m^{(5)}$, we want to choose $Y_{\beta}(x,t,k)$ so that $m^{\beta}(m^{\model})^{-1}\to I$ on $\partial D_{\epsilon}(\beta)$ as $t\to\infty$. Hence we choose
\begin{equation}  \label{Ybeta-def}
Y_{\beta}(x,t,k)\coloneqq m^{\model}\eul^{-\ii(\frac{\ii}{2}\ln(\hat a\hat b\delta^2)+tg(\beta)+h)\sigma_3}A(k)m_{\asympt,N}^{\Airy,\inv}(\zeta(k)),
\end{equation}
where $N\geq 0$ is some integer and $m_{\asympt,N}^{\Airy,\inv}$ is defined in \eqref{mAiry-approx-defb}. Equation \eqref{mAiry-approx-jumpb} implies that $m_{\asympt,N}^{\Airy,\inv}(\zeta(k))$ is analytic near $\beta$ except for a jump across $\C{Z}_3$ given by
\[
(m_{\asympt,N}^{\Airy,\inv}(\zeta(k)))_+=\begin{pmatrix}0&-1\\ 1&0\end{pmatrix}(m_{\asympt,N}^{\Airy,\inv}(\zeta(k)))_-,\quad k\in\C{Z}_3.
\]
Equation \eqref{vmodel-def} then shows that the jumps of $m^{\model}\eul^{-\ii(\frac{\ii}{2}\ln(\hat a\hat b\delta^2)+tg(\beta)+h)\sigma_3}A(k)$ near $\beta$ are such that the function $Y_{\beta}$ is analytic in $D_{\epsilon}(\beta)$.

The matrices $m^{\model}$, $\eul^{-\ii(\frac{\ii}{2}\ln(\hat a\hat b\delta^2)+tg(\beta)+h)\sigma_3}$, $A(k)$, and their inverses are uniformly bounded for $k\in\partial D_{\epsilon}(\beta)$. Shrinking $\epsilon$ if necessary, we may assume that 
\[
\inf_{k\in\partial D_{\epsilon}(\beta)}\abs{g_{\beta}(k)}>0.
\]
The asymptotic formula \eqref{mAiry-approx-mAiryb} and the boundedness of $m^{\model}$ then imply
\begin{equation}   \label{mbeta-mmodel-t-infty}
m^{\beta}(k)m^{\model}(k)^{-1}=I+\ord(t^{-N-1}),\quad t\to\infty,\quad k\in\partial D_{\epsilon}(\beta),
\end{equation}
uniformly with respect to $k$ in the given range. We have proved the following lemma.

%-------------------%
%:lem 6.2
%-------------------%
\begin{lemma}
The function $m^{\beta}(x,t,k)$ defined in \eqref{mbeta-def} is an analytic and bounded function of $k\in D_{\epsilon}(\beta)\setminus\Sigma^{(5)}$, which satisfies the same jump conditions as $\hat m^{(5)}$ across $\Sigma^{(5)}\cap D_{\epsilon}(\beta)$. The function $Y_{\beta}(x,t,k)$ defined in \eqref{Ybeta-def} is an analytic and bounded function of $k\in D_{\epsilon}(\beta)$. As $t\to+\infty$, $m^{\beta}$ satisfies \eqref{mbeta-mmodel-t-infty}.
\end{lemma}
%-------------------%
 
%---------------------------------------------------------%
%:s.6.3
%---------------------------------------------------------%
\subsection{Local model near $\BS{\mu}$}

Let $\accol{\C{R}_j}_1^6$ denote the open subsets of $D_{\epsilon}(\mu)$ displayed in Figure~\ref{fig:contour-near-mu}. Let $\C{X}_j=\C{R}_{j-1}\cap\C{R}_j$, $j=1,\dots,6$, $\C{R}_0\equiv\C{R}_6$, denote the rays separating the $\C{R}_j$. Define $g_{\mu}(k)$ for $k$ near $\mu$ by
\[
g_{\mu}(k)\coloneqq\int_{\mu}^k\dd g=
\begin{cases}
g(k)-g_-(\mu),&k\in\C{R}_1\cup\C{R}_2\cup\C{R}_6,\\
g(k)-g_+(\mu),&k\in\C{R}_3\cup\C{R}_4\cup\C{R}_5,
\end{cases},\quad
k\in D_{\epsilon}(\mu),
\]
where the integration contour lies on $D_{\epsilon}(\mu)\setminus\gamma_{(\bar\beta,\beta)}$.
%-------------------%
%:fig 16 = 6.3
%-------------------%
\begin{figure}[ht]
\centering\includegraphics[scale=.9]{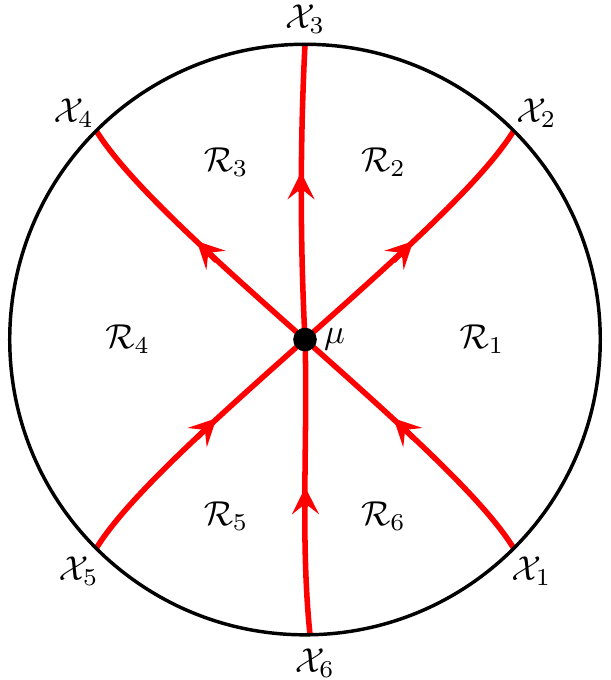}
\caption{The contour in the disk of radius $\epsilon$ centered at $\mu$ and the sets $\accol{\C{R}_j}_1^6$.} 
\label{fig:contour-near-mu}
\end{figure}
%-------------------%

Define the piecewise constant function $B(k)\equiv B(\xi,k)$ by
\begin{equation}  \label{B-def}
B(k)\coloneqq\begin{cases}
\eul^{-\ii tg_-(\mu)\sigma_3},&k\in\C{R}_1\cup\C{R}_2\cup\C{R}_6,\\  
\eul^{-\ii tg_+(\mu)\sigma_3},&k\in\C{R}_3\cup\C{R}_4\cup\C{R}_5,  
\end{cases}
\end{equation}
and define the function $m^{(\mu 0)}$ for $k$ near $\mu$ by
\begin{equation}
m^{(\mu 0)}(x,t,k)\coloneqq\hat m^{(5)}(x,t,k)\eul^{-\ii h\sigma_3}B(k),\quad k\in D_{\epsilon}(\mu)\setminus\Sigma^{(5)}.
\end{equation}
Across the part of $\Sigma^{(5)}$ that lies in $D_{\epsilon}(\mu)$, $m^{(\mu 0)}$ satisfies the jump condition $m_+^{(\mu 0)}=m_-^{(\mu 0)}v^{(\mu 0)}$ where
\[
v^{(\mu 0)}\coloneqq\begin{cases}
\begin{pmatrix}1&0\\\frac{\hat b^*}{\hat a}\delta^{-2}\eul^{2\ii tg_{\mu}}&1\end{pmatrix},&k\in\C{X}_1,\\
\begin{pmatrix}1&-\frac{\hat b}{\hat a^*}\delta^2\eul^{-2\ii tg_{\mu}}\\0&1\end{pmatrix},&k\in\C{X}_2,\\
\begin{pmatrix}\frac{1}{aa^*}&0\\0&aa^*\end{pmatrix},&k\in\C{X}_3,\\
\begin{pmatrix}1&0\\\frac{\hat b^*}{\hat a^2\hat a^*}\delta^{-2}\eul^{2\ii tg_{\mu}}&1\end{pmatrix},&k\in\C{X}_4,\\
\begin{pmatrix}1&-\frac{\hat b}{\hat a(\hat a^*)^2}\delta^2\eul^{-2\ii tg_{\mu}}\\0&1\end{pmatrix},&k\in\C{X}_5,\\
\begin{pmatrix}aa^*&0\\0&\frac{1}{aa^*}\end{pmatrix},&k\in\C{X}_6,
\end{cases}
\]
and all contours are oriented upward as in Figure~\ref{fig:contour-near-mu}.

We want to eliminate the jumps across $\gamma_{(\bar\beta,\beta)}$. Hence we define the complex-valued function $\tilde\delta(k)\equiv\tilde\delta(\xi,k)$ by
\begin{equation}  \label{tilde-delta-def}
\tilde\delta(k)\coloneqq\eul^{-\frac{1}{2\pi\ii}\int_{\gamma_{(\mu,\beta)}}\frac{\ln(a(s)a^*(s))}{s-k}\dd s}\eul^{\frac{1}{2\pi\ii}\int_{\gamma_{(\bar\beta,\mu)}}\frac{\ln(a(s)a^*(s))}{s-k}\dd s},\quad k\in\D{C}\setminus\gamma_{(\bar\beta,\beta)},
\end{equation}
where the branch of the logarithm is chosen so that $\ln(a(s)a^*(s))$ is continuous on each contour and real for $s=\mu$. Since $a(k)$ has no zeros or poles, the function $aa^*$ is nonzero and finite everywhere on the contours. The function $\tilde\delta$ is in general singular at $\beta$ and $\bar\beta$, but near $\mu$ (the only region we are concerned about) it is bounded according to the following lemma.

%-------------------%
%:lem 6.3
%-------------------%
\begin{lemma}   \label{tilde-delta-properties}
The function $\tilde\delta(k)$ has the following properties:
\begin{enumerate}[\rm(a)]
\item
$\tilde\delta(k)$ and $\tilde\delta(k)^{-1}$ are bounded and analytic functions of $k\in D_{\epsilon}(\mu)\setminus\gamma_{(\bar\beta,\beta)}$.
\item
$\tilde\delta$ obeys the symmetry
\[
\tilde\delta=(\tilde\delta^*)^{-1},\quad k\in\D{C}\setminus\gamma_{(\bar\beta,\beta)}.
\]
\item
$\tilde\delta$ satisfies the jump condition
\begin{equation}  \label{tilde-delta-jump}
\tilde\delta_+=\tilde\delta_-\times
\begin{cases}
\frac{1}{aa^*},&k\in\gamma_{(\mu,\beta)},\\[2mm]
aa^*,&k\in\gamma_{(\bar\beta,\mu)}.
\end{cases}
\end{equation}
\end{enumerate}
\end{lemma}
%-------------------%

%-------------------%
\begin{proof}
Since $\ln(\abs{a(\mu)}^2)\in\D{R}$, Lemma~\ref{lem-C} shows that $\tilde\delta(k)$ is bounded for $k\in D_{\epsilon}(\mu)\setminus\gamma_{(\bar\beta,\beta)}$. The other properties follow easily from the definition~\eqref{tilde-delta-def}.
\end{proof}
%-------------------%

Let
\[
m^{(\mu 1)}(x,t,k)\coloneqq m^{(\mu 0)}(x,t,k)\tilde\delta(k)^{-\sigma_3},\quad k\in D_{\epsilon}(\mu).
\]
Then $m^{(\mu 1)}$ satisfies the jump condition $m_+^{(\mu 1)}=m_-^{(\mu 1)}v^{(\mu 1)}$ where
\[
v^{(\mu 1)}\coloneqq\tilde\delta_-^{\sigma_3}v^{(\mu 0)}\tilde\delta_+^{-\sigma_3}.
\]
Defining $\delta_2(k)\equiv\delta_2(\xi,k)$ by $\delta_2\coloneqq\delta\tilde\delta$, we find   
\[
v^{(\mu 1)}= 
\begin{cases}
\begin{pmatrix}1&0\\\frac{\hat b^*}{\hat a}\delta_2^{-2}\eul^{2\ii tg_{\mu}}&1\end{pmatrix},&k\in\C{X}_1,\\ 
\begin{pmatrix}1&-\frac{\hat b}{\hat a^*}\delta_2^2\eul^{-2\ii tg_{\mu}}\\0&1\end{pmatrix},&k\in\C{X}_2,\\ 
\begin{pmatrix}1&0\\\frac{\hat b^*}{\hat a^2\hat a^*}\delta_2^{-2}\eul^{2\ii tg_{\mu}}&1\end{pmatrix},&k\in\C{X}_4,\\ 
\begin{pmatrix}1&-\frac{\hat b}{\hat a(\hat a^*)^2}\delta_2^2\eul^{-2\ii tg_{\mu}}\\0&1\end{pmatrix},&k\in\C{X}_5,\\ 
\,I,&k\in\C{X}_3\cup\C{X}_6,
\end{cases}
\]
or, in terms of $\hat r$,
\[
v^{(\mu 1)}= 
\begin{cases}
\begin{pmatrix}1&0\\\hat r\delta_2^{-2}\eul^{2\ii tg_{\mu}}&1\end{pmatrix},&k\in\C{X}_1,\\ 
\begin{pmatrix}1&-\hat r^*\delta_2^2\eul^{-2\ii tg_{\mu}}\\0&1\end{pmatrix},&k\in\C{X}_2,\\ 
\begin{pmatrix}1&0\\\hat r(1+\hat r\hat r^*)\delta_2^{-2}\eul^{2\ii tg_{\mu}}&1\end{pmatrix},&k\in\C{X}_4,\\ 
\begin{pmatrix}1&-\hat r^*(1+\hat r\hat r^*)\delta_2^2\eul^{-2\ii tg_{\mu}}\\0&1\end{pmatrix},&k\in\C{X}_5,\\ 
\,I,&k\in\C{X}_3\cup\C{X}_6.
\end{cases}
\]
%-------------------%
%:fig 17 = 6.4
%-------------------%
\begin{figure}[ht]
\centering\includegraphics[scale=.95]{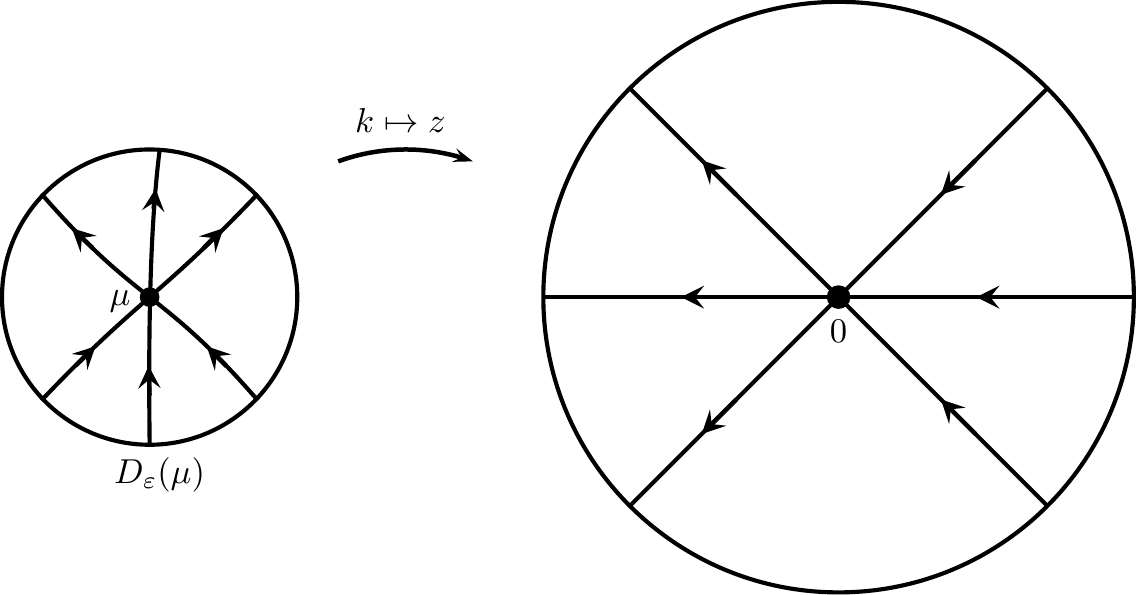}
\caption{The local change of variables $k\mapsto z$ is a biholomorphism of the disk $D_{\epsilon}(\mu)$ onto a neighborhood of the origin in the $z$-plane which straightens out the contour.} 
\label{fig:k-to-z}
\end{figure}
%-------------------%

In order to relate $m^{(\mu 1)}$ to the solution $m^X$ of Appendix~\ref{sec:A}, we make a local change of variables for $k$ near $\mu$ and introduce the new variable $z\equiv z(k)$ by (see Figure~\ref{fig:k-to-z})
\[
z\coloneqq\sqrt{tg_{\mu}(k)},
\]
where the branch of the square root is fixed as follows: Since $g_{\mu}(k)$ has a double zero at $k=\mu$, we can write
\begin{equation}  \label{psimu-def}
z=\ii\sqrt{t}(k-\mu)\psi_{\mu}(k),
\end{equation} 
where the function $\psi_{\mu}(k)\equiv\psi_{\mu}(\xi,k)$ is analytic for $k\in D_{\epsilon}(\mu)$. A computation using the definition \eqref{dg} of $\dd g$ shows that $\frac{\dd^2g}{\dd k^2}(\mu)<0$, hence we can fix the branch by requiring that $\psi_{\mu}(\mu)>0$. Shrinking $\epsilon>0$ if necessary, we may assume that $\Re\psi_{\mu}(k)>0$ for $k\in\overline{D_{\epsilon}(\mu)}$. We note that $z$ satisfies 
\[
2\ii tg_{\mu}(k)=2\ii z^2.
\]

The map $k\mapsto z$ is a biholomorphism from the disk $D_{\epsilon}(\mu)$ onto a neighborhood of the origin, which maps $\C{X}_3$ and $\C{X}_6$ into $\D{R}_-$ and $\D{R}_+$, respectively. By deforming the contours $\C{X}_1$, $\C{X}_2$, $\C{X}_4$, and $\C{X}_5$ if necessary, we may assume that they are mapped into the straight rays of the complex $z$-plane for which $\arg z$ equals $\frac{\pi}{4}$, $\frac{3\pi}{4}$, $-\frac{3\pi}{4}$, and $-\frac{\pi}{4}$, respectively.

We next consider the behavior of $\tilde\delta$ as $k$ approaches $\mu$. Let $\ln_{\beta}(k-\mu)$ and $\ln_{\bar\beta}(k-\mu)$ denote the function $\ln(k-\mu)$ with branch cut along $\gamma_{(\mu,\beta)}$ and $\gamma_{(\mu,\bar\beta)}$, respectively. More precisely,
\begin{alignat*}{2}
&\ln_{\beta}(k-\mu)=\ln(k-\mu),&\qquad&k\in D_{\epsilon}\setminus\gamma_{(\mu,\beta)},\\
&\ln_{\bar\beta}(k-\mu)=\ln(k-\mu),&&k\in D_{\epsilon}\setminus\gamma_{(\bar\beta,\mu)},
\end{alignat*}
where the branches are fixed by requiring that $\ln_{\beta}(k-\mu)$ and $\ln_{\bar\beta}(k-\mu)$ are strictly positive for $k>\mu$. We also define two functions $L_{\beta}$ and $L_{\bar\beta}$ by
\begin{alignat*}{3}
&L_{\beta}(s,k)\coloneqq\ln(k-s),&\quad&s\in\gamma_{(\mu,\beta)},&\quad&k\in D_{\epsilon}\setminus\gamma_{(\mu,\beta)},\\
&L_{\bar\beta}(s,k)\coloneqq\ln(k-s),&&s\in\gamma_{(\bar\beta,\mu)},&&k\in D_{\epsilon}\setminus\gamma_{(\bar\beta,\mu)}, 
\end{alignat*}
where the branches are fixed by requiring that:
\begin{enumerate}[(i)]
\item
$L_{\beta}(s,k)$ is a continuous function of $s\in\gamma_{(\mu,\beta)}$ for each $k\in D_{\epsilon}\setminus\gamma_{(\mu,\beta)}$.
\item
$L_{\bar\beta}(s,k)$ is a continuous function of $s\in\gamma_{(\bar\beta,\mu)}$ for each $k\in D_{\epsilon}\setminus\gamma_{(\bar\beta,\mu)}$.
\item
For $s=\mu$, we have $L_{\beta}(\mu,k)=\ln_{\beta}(k-\mu)$ and $L_{\bar\beta}(\mu,k)=\ln_{\bar\beta}(k-\mu)$.
\end{enumerate}
Let $q\equiv q(\xi)$ denote the constant $q\coloneqq\hat r(\mu)$. Integrating by parts, we find
\begin{align*}
\int_{\gamma_{(\mu,\beta)}}\frac{\ln(1+r(s)r^*(s))}{s-k}\dd s
&=L_{\beta}(\beta,k)\ln(1+r(\beta)r^*(\beta))-\ln_{\beta}(k-\mu)\ln(1+\abs{q}^2)\\
&\quad-\int_{\gamma_{(\mu,\beta)}}L_{\beta}(s,k)\dd\ln(1+r(s)r^*(s)) 
\end{align*}
and
\begin{align*}
\int_{\gamma_{(\bar\beta,\mu)}}\frac{\ln(1+r(s)r^*(s))}{s-k}\dd s
&=-L_{\bar\beta}(\bar\beta,k)\ln(1+r(\bar\beta)r^*(\bar\beta))+\ln_{\bar\beta}(k-\mu)\ln(1+\abs{q}^2)\\
&\quad-\int_{\gamma_{(\bar\beta,\mu)}}L_{\bar\beta}(s,k)\dd\ln(1+r(s)r^*(s)).
\end{align*}
Hence we can write
\begin{equation}  \label{tilde-delta-bis}
\tilde\delta(k)=\eul^{\ii\nu\croch{\ln_{\beta}(k-\mu)+\ln_{\bar\beta}(k-\mu)}+\tilde\chi(k)},\quad k\in\D{C}\setminus\gamma_{(\bar\beta,\beta)},
\end{equation}
where the constant $\nu\equiv\nu(\xi)>0$ is defined by
\begin{equation}  \label{nu-def}
\nu\coloneqq\frac{1}{2\pi}\ln(1+\abs{q}^2),
\end{equation}
and the function $\tilde\chi(k)\equiv\tilde\chi(\xi,k)$ is defined by
\begin{align}
\tilde\chi(k)&\coloneqq\frac{1}{2\pi\ii}L_{\beta}(\beta,k)\ln(1+r(\beta)r^*(\beta))+\frac{1}{2\pi\ii}L_{\bar\beta}(\bar\beta,k)\ln(1+r(\bar\beta)r^*(\bar\beta))\notag\\
&\quad-\frac{1}{2\pi\ii}\int_{\gamma_{(\mu,\beta)}}L_{\beta}(s,k)\dd\ln(1+r(s)r^*(s))\notag\\
&\quad+\frac{1}{2\pi\ii}\int_{\gamma_{(\bar\beta,\mu)}}L_{\bar\beta}(s,k)\dd\ln(1+r(s)r^*(s)).
\end{align}

We also consider the behavior of $\delta(k)$ as $k$ approaches $\mu$. Integration by parts gives
\begin{align*}
\int_{-\infty}^{\mu}\frac{\ln(1+\abs{r(s)}^2)}{s-k}\dd s&=\ln(k-s)\ln(1+\abs{r(s)}^2)\vert_{s=-\infty}^{-1}\\
&\quad +\ln(k-s)\ln(1+\abs{r(s)}^2)\vert_{s=-1}^{\mu}\\
&\quad -\left(\int_{-\infty}^{-1}+\int_{-1}^{\mu}\right)\ln(k-s)\dd\ln(1+\abs{r(s)}^2),\quad k\in\D{C}\setminus(-\infty,\mu\rbrack.
\end{align*}
Hence we can write $\delta$ as
\begin{equation}  \label{delta-bis}
\delta(k)=\eul^{-\ii\nu\ln(k-\mu)+\chi(k)},\quad k\in\D{C}\setminus(-\infty,\mu\rbrack,
\end{equation}
where the function $\chi(k)\equiv\chi(\xi,k)$ is defined by
\begin{align}
\chi(k)&\coloneqq\frac{1}{2\pi\ii}\ln(k+1)\ln\frac{1+\abs{r_+(-1)}^2}{1+\abs{r_-(-1)}^2}\notag\\
&\quad-\frac{1}{2\pi\ii}\left(\int_{-\infty}^{-1}+\int_{-1}^{\mu}\right)\ln(k-s)\dd\ln(1+\abs{r(s)}^2),\quad k\in\D{C}\setminus(-\infty,\mu\rbrack,
\end{align}
with $r_+(-1)$ and $r_-(-1)$ denoting the values of $r(k)$ on the left and right sides of $\Sigma_1$, respectively. 

For a real number $a$, let $\ln_a(k)$ denote the logarithm of $k$ with branch cut along $\arg k=a$, i.e., $\ln_ak=\ln\abs{k}+\ii\arg k$ with $\arg k\in(a,a+2\pi\rbrack$. Then $\ln_{-\pi}k=\ln k$. 

Define the function $p(z)\equiv p(\xi,z)$ by
\[
p(z)\coloneqq\eul^{-\ii\nu\croch{\ln_{-\pi/2}(z)-\ln z-\ln_0(z)}},\quad z\in\D{C}\setminus(\D{R}\cup\ii\,\D{R}_-).
\]
Moreover, define the functions $\delta_0(t)\equiv\delta_0(\xi,t)$ and $\delta_1(k)\equiv\delta_1(\xi,k)$ by
\begin{alignat*}{2}
\delta_0(t)&\coloneqq\eul^{\frac{\pi\nu}{2}}t^{-\frac{\ii\nu}{2}}\eul^{-\ii\nu\ln\psi_{\mu}(\mu)}\eul^{\chi(\mu)+\tilde\chi(\mu)},&\qquad&t>0,\\
\delta_1(k)&\coloneqq\eul^{-\ii\nu\ln\frac{\psi_{\mu}(k)}{\psi_{\mu}(\mu)}}\eul^{\chi(k)-\chi(\mu)+\tilde\chi(k)-\tilde\chi(\mu)},&&k\in D_{\epsilon}(\mu).
\end{alignat*}

%-------------------%
%:lem 6.4
%-------------------%
\begin{lemma}\label{delta-two-properties}
We have
\[
\delta_2(k)=p(z(k))\delta_0(t)\delta_1(k),\quad k\in D_{\epsilon}(\mu) \setminus\bigl((-\infty,\mu\rbrack\cup\gamma_{(\bar\beta,\beta)}\bigr).
\]
\end{lemma}
%-------------------%

%-------------------%
\begin{proof}
Using that the map $k\mapsto z$ takes $(-\infty,\mu)\cap D_{\epsilon}(\mu)$ into $\ii\,\D{R}_-$ and that $\Re\psi_{\mu}>0$ in $D_{\epsilon}(\mu)$, we find
\[
\ln(k-\mu)=\ln\frac{-\ii z}{\sqrt{t}\psi_{\mu}(k)}=-\frac{\pi\ii}{2}-\frac{\ln t}{2}+\ln_{-\pi/2}(z)-\ln\psi_{\mu}(k),\quad k\in D_{\epsilon}(\mu)\setminus(-\infty,\mu\rbrack,
\]      
where $z\equiv z(k)$. Thus equation \eqref{delta-bis} can be written as
\begin{equation}   \label{delta-ter}
\delta(k)=\eul^{-\frac{\pi\nu}{2}}t^{\frac{\ii\nu}{2}}\eul^{-\ii\nu\croch{\ln_{-\pi/2}(z)-\ln\psi_{\mu}(k)}+\chi(k)},\quad k\in\D{C}\setminus(-\infty,\mu\rbrack. 
\end{equation}
Similarly, since the map $k\mapsto z$ takes $\C{X}_3$ into $\D{R}_-$ and $\C{X}_6$ into $\D{R}_+$,
\begin{alignat*}{2}
&\ln_{\beta}(k-\mu)=\ln_{\beta}\frac{-\ii z}{\sqrt{t}\psi_{\mu}(k)}=-\frac{\ii\pi}{2}-\frac{\ln t}{2}+\ln z-\ln\psi_{\mu}(k),&\quad&k\in D_\epsilon\setminus\gamma_{(\mu,\beta)},\\
&\ln_{\bar\beta}(k-\mu)=\ln_{\bar\beta}\frac{-\ii z}{\sqrt{t}\psi_{\mu}(k)}=-\frac{\ii\pi}{2}-\frac{\ln t}{2}+\ln_0z-\ln\psi_{\mu}(k),&&k\in D_\epsilon\setminus\gamma_{(\bar\beta,\mu)}.
\end{alignat*}
Substituting these expressions into \eqref{tilde-delta-bis}, we obtain
\begin{equation}  \label{tilde-delta-ter}
\tilde\delta(k)=\eul^{\pi\nu}t^{-\ii\nu}\eul^{\ii\nu\croch{\ln z+\ln_0z-2\ln\psi_{\mu}(k)}+\tilde\chi(k)},\quad k\in\D{C}\setminus\gamma_{(\bar\beta,\beta)}.
\end{equation}
Equations~\eqref{delta-ter} and \eqref{tilde-delta-ter} yield
\[
\delta_2=\eul^{\frac{\pi\nu}{2}}t^{-\frac{\ii\nu}{2}}\eul^{\ii\nu\croch{-\ln_{-\pi/2}z+\ln z+\ln_0z-\ln\psi_{\mu}(k)}}\eul^{\chi(k)+\tilde\chi(k)},
\]
so the lemma follows.
\end{proof}
%-------------------%

Define $m^{(\mu 2)}(x,t,z)$ by
\[
m^{(\mu 2)}(x,t,z(k))\coloneqq m^{(\mu 1)}(x,t,k)\delta_0(t)^{\sigma_3},\quad k\in D_{\epsilon}(\mu)\setminus\Sigma^{(5)}.
\]
Then $m^{(\mu 2)}$ is an analytic function of $z$ except for a jump across the cross $X$ displayed in Figure~\ref{fig:contour-X}. Across $X$, $m^{(\mu 2)}$ satisfies the jump condition $m_+^{(\mu 2)}=m_-^{(\mu 2)}v^{(\mu 2)}$ where
\[
v^{(\mu 2)}\coloneqq
\begin{cases}
\begin{pmatrix}1&0\\\hat r\delta_1^{-2}p(z)^{-2}\eul^{2\ii z^2}&1\end{pmatrix},&\arg z=\frac{\pi}{4},\\ 
\begin{pmatrix}1&\hat r^*\delta_1^2p(z)^2\eul^{-2\ii z^2}\\0&1\end{pmatrix},&\arg z=\frac{3\pi}{4},\\ 
\begin{pmatrix}1&0\\-\hat r(1+\hat r\hat r^*)\delta_1^{-2}p(z)^{-2}\eul^{2\ii z^2}&1\end{pmatrix},&\arg z=\frac{5\pi}{4},\\ 
\begin{pmatrix}1&-\hat r^*(1+\hat r\hat r^*)\delta_1^2p(z)^2\eul^{-2\ii z^2}\\0&1\end{pmatrix},&\arg z=\frac{7\pi}{4}, 
\end{cases}
\] 
and all contours are oriented toward the origin as in Figure~\ref{fig:contour-X}. Define the function $\rho(z)\equiv\rho(\xi,z)$ by
\[
\rho(z)\coloneqq\eul^{\ii\nu\ln_{-\pi/2}(z)},
\]
that is, $\rho(z)$ equals the function $z^{\ii\nu}$ with branch cut along the negative imaginary axis. Since
\[
p(z)=\rho(z)\times
\begin{cases}
1,&\arg z\in(0,\pi),\\
\eul^{2\pi\nu},&\arg z\in(-\pi,-\pi/2),\\
\eul^{-2\pi\nu},&\arg z\in(-\pi/2,0),
\end{cases}
\]
and $\eul^{2\pi\nu}=1+\abs{q}^2$, we can write
\[
v^{(\mu2)}=
\begin{cases}
\begin{pmatrix}1&0\\\hat r\delta_1^{-2}\rho^{-2}\eul^{2\ii z^2}&1\end{pmatrix},&\arg z=\frac{\pi}{4},\\ 
\begin{pmatrix}1&\hat r^*\delta_1^2\rho^2\eul^{-2\ii z^2}\\0&1\end{pmatrix},&\arg z=\frac{3\pi}{4},\\ 
\begin{pmatrix}1&0\\-\frac{\hat r(1+\hat r\hat r^*)}{(1+\abs{q}^2)^2}\delta_1^{-2}\rho^{-2}\eul^{2\ii z^2}&1\end{pmatrix},&\arg z=\frac{5\pi}{4},\\[1mm]
\begin{pmatrix}1&-\frac{\hat r^*(1+\hat r\hat r^*)}{(1+\abs{q}^2)^2}\delta_1^2\rho^2\eul^{-2\ii z^2}\\0&1\end{pmatrix},&\arg z=\frac{7\pi}{4}. 
\end{cases}
\]

For a fixed $z$, $\hat r(k(z))\to q$ and $\delta_1(k(z))\to 1$ as $t\to\infty$. This suggests that $v^{(\mu 2)}$ tends to the jump matrix $v^X$ defined in \eqref{jump-A} for large $t$. In other words, it suggests that the jumps of $\hat m^{(5)}$ for $k$ near $\mu$ approach those of the function $m^X\delta_0^{-\sigma_3}\tilde\delta^{\sigma_3}B(k)^{-1}\eul^{\ii h\sigma_3}$ as $t\to\infty$, where $m^X$ is the solution of the RH problem \eqref{rhp-A}. This suggests that we approximate $\hat m^{(5)}$ near $\mu$ with a $2\times 2$-matrix valued function $m^{\mu}$ of the form
\begin{equation}    \label{mmu-def}
m^{\mu}(x,t,k)=Y_{\mu}(x,t,k)m^X(q,z(k))\delta_0(t)^{-\sigma_3}\tilde\delta(k)^{\sigma_3}B(k)^{-1}\eul^{\ii h(k)\sigma_3},
\end{equation}
where $Y_{\mu}(x,t,k)$ is a function which is analytic for $k\in D_{\epsilon}(\mu)$. To ensure that $m^{\mu}$ is a good approximation of $\hat m^{(5)}$ for large $t$, we want to choose $Y_{\mu}(k)$ so that $m^{\mu}(m^{\model})^{-1}\to I$ on $\partial D_{\epsilon}(\mu)$ as $t\to\infty$. Hence we choose
\begin{equation}    \label{Ymu-def}
Y_{\mu}(x,t,k)\coloneqq m^{\model}(x,t,k)\eul^{-\ii h(k)\sigma_3}B(k)\tilde\delta(k)^{-\sigma_3}\delta_0(t)^{\sigma_3}.
\end{equation}
The function $f(k)\coloneqq m^{\model}(x,t,k)\eul^{-\ii h(k)\sigma_3}$ is analytic in $D_{\epsilon}(\mu)$ except for a jump
across $\C{X}_3\cup\C{X}_6$ given by
\[
f_+=f_-\times
\begin{cases}
\begin{pmatrix}\frac{\eul^{\ii t(g_+-g_-)}}{aa^*}&0\\0&aa^*\eul^{-\ii t(g_+-g_-)}\end{pmatrix},&k\in\C{X}_3,\\
\begin{pmatrix}aa^*\eul^{\ii t(g_+-g_-)}&0\\0&\frac{\eul^{-\ii t(g_+-g_-)}}{aa^*}\end{pmatrix},&k\in\C{X}_6.
\end{cases}
\]
But equations \eqref{B-def} and \eqref{tilde-delta-jump} show that $\tilde\delta(k)^{\sigma_3}B(k)^{-1}$ has the same jump across $\C{X}_3\cup\C{X}_6$. Thus $Y_{\mu}$ is analytic in $D_{\epsilon}(\mu)$.

The part of $\Sigma^{(5)}$ that lies in $D_{\epsilon}(\mu)$ is the union of the cross
\begin{equation}   \label{cross}
\C{X}\coloneqq\C{X}_1\cup\C{X}_2\cup\C{X}_4\cup\C{X}_5
\end{equation}
and the segment $\gamma_{(\bar\beta,\beta)}\cap D_{\epsilon}(\mu)=\C{X}_3\cup\C{X}_6$.

%-------------------%
%:lem 6.5
%-------------------%
\begin{lemma}     \label{mmu-properties}
The function $m^{\mu}(x,t,k)$ defined in \eqref{mmu-def} is an analytic and bounded function of $k\in D_{\epsilon}(\mu)\setminus\Sigma^{(5)}$. The function $Y_{\mu}(x,t,k)$ defined in \eqref{Ymu-def} is an analytic and bounded function of $k\in D_{\epsilon}(\mu)$. Across $\Sigma^{(5)}\cap D_{\epsilon}(\mu)$, $m^{\mu}$ satisfies the jump condition $m_+^{\mu}=m_-^{\mu}v^{\mu}$, where the jump matrix $v^{\mu}$ satisfies
\[
v^{\mu}=\hat v^{(5)},\quad k\in\gamma_{(\bar\beta,\beta)}\cap D_{\epsilon}(\mu),
\]
and
\begin{equation}  \label{v5-vmu-norms}
\begin{cases}
\norm{\hat v^{(5)}-v^{\mu}}_{L^1(\C{X})}=\ord(t^{-1}\ln t),&\\
\norm{\hat v^{(5)}-v^{\mu}}_{L^2(\C{X})}=\ord(t^{-3/4}\ln t),&\\
\norm{\hat v^{(5)}-v^{\mu}}_{L^{\infty}(\C{X})}=\ord(t^{-1/2}\ln t),&
\end{cases}t\to\infty. 
\end{equation} 
Furthermore, on $\partial D_{\epsilon}(\mu)$ the quotient $m^{\model}(m^{\mu})^{-1}$ satisfies
\begin{equation}  \label{mmu-mmodel-sup-t-infty}
\norm{m^{\model}(m^{\mu})^{-1}-I}_{L^{\infty}(\partial D_{\epsilon}(\mu))}=\ord(t^{-1/2}),
\end{equation}
and
\begin{equation}  \label{mmu-mmodel-1-t-infty}
\frac{1}{2\pi\ii}\int_{\partial D_{\epsilon}(\mu)}\left(m^{\model}(m^{\mu})^{-1}-I\right)\dd k=\frac{Y_{\mu}(x,t,\mu)m_1^XY_{\mu}(x,t,\mu)^{-1}}{\sqrt{t}\psi_{\mu}(\mu)}+\ord(t^{-1}),
\end{equation}
as $t\to\infty$, where $m_1^X\equiv m_1^X(\xi)$ is defined by
\begin{equation}  \label{m1X-def}
m_1^X\coloneqq\begin{pmatrix}0&-\eul^{-\pi\nu}\beta^X(q)\\\eul^{\pi\nu}\overline{\beta^X(q)}&0\end{pmatrix}
\end{equation}
with $\beta^X(q)$ given by \eqref{betaX-def}.
\end{lemma}
%-------------------%

%-------------------%
\begin{proof}
It only remains to prove \eqref{v5-vmu-norms}-\eqref{m1X-def}. In the following $C\equiv C(\xi)$ denotes a constant independent of $k$ which may change within a computation.

Standard estimates show that
\[
\abs{\chi(k)-\chi(\mu)}\leq C\abs{k-\mu}(1+\abs{\ln\abs{k-\mu}}),\quad k\in\C{X},
\]
and
\[
\abs{\tilde\chi(k)-\tilde\chi(\mu)}\leq C\abs{k-\mu}(1+\abs{\ln\abs{k-\mu}}),\quad k\in\C{X}.
\]
Using that $\eul^{-\ii\nu\ln\frac{\psi_{\mu}(k)}{\psi_{\mu}(0)}}-1=\ord(k-\mu)$, this yields
\begin{align}  \label{delta-one-near-mu}
\abs{\delta_1(k)-1}
&=\bigl\lvert\eul^{-\ii\nu\ln\frac{\psi_{\mu}(k)}{\psi_{\mu}(\mu)}}\eul^{\chi(k)-\chi(\mu)}\eul^{\tilde\chi(k)-\tilde\chi(\mu)}-1\bigr\rvert\notag\\
&\leq C\abs{k-\mu}(1+\abs{\ln\abs{k-\mu}}),\quad k\in\C{X}.
\end{align}
On the other hand, the smoothness of $\hat r(k)$ implies that 
\begin{equation}   \label{r-near-mu}
\hat r(k)-q=\ord(k-\mu),\quad k\in\C{X}.
\end{equation}
Since
\[
v^{(\mu 2)}-v^X=
\begin{cases}
\begin{pmatrix}0&0\\(\hat r\delta_1^{-2}-q)\rho(z)^{-2}\eul^{2\ii z^2}&0\end{pmatrix},&\arg z=\frac{\pi}{4},\\ 
\begin{pmatrix}0&(\hat r^*\delta_1^2-\bar q)\rho(z)^2\eul^{-2\ii z^2}\\0&0\end{pmatrix},&\arg z=\frac{3\pi}{4},\\ 
\begin{pmatrix}0&0\\-\frac{\hat r(1+\hat r\hat r^*)\delta_1^{-2}-q(1+\abs{q}^2)}{(1+\abs{q}^2)^2}\rho(z)^{-2}\eul^{2\ii z^2}&0\end{pmatrix},&\arg z=\frac{5\pi}{4},\\ 
\begin{pmatrix}0&-\frac{\hat r^*(1+\hat r\hat r^*)\delta_1^2-\bar q(1+\abs{q}^2)}{(1+\abs{q}^2)^2}\rho(z)^2\eul^{-2\ii z^2}\\0&0\end{pmatrix},&\arg z=\frac{7\pi}{4}, 
\end{cases}
\]
equations \eqref{delta-one-near-mu} and \eqref{r-near-mu} imply that
\[
\abs{v^{(\mu 2)}-v^X}\leq C\abs{k-\mu}(1+\abs{\ln\abs{k-\mu}})\eul^{-ct\abs{k-\mu}^2},\quad k\in\C{X},
\]
where $c>0$. But
\[
\hat v^{(5)}-v^{\mu}=\begin{cases}
G_-(k)^{-1}(v^{(\mu 2)}-v^X)G_+(k),&k\in\C{X},\\
0,&k\in\gamma_{(\bar\beta,\beta)}\cap D_{\epsilon}(\mu),
\end{cases}
\]
where the function $G(k)\coloneqq\delta_0(t)^{-\sigma_3}\tilde\delta(k)^{\sigma_3}B(k)^{-1}\eul^{\ii h(k)\sigma_3}$ and its inverse are bounded for $k\in\Sigma^{(5)}\cap D_{\epsilon}(\mu)$, so we find
\[
\abs{\hat v^{(5)}-v^{\mu}}\leq C\abs{k-\mu}(1+\abs{\ln\abs{k-\mu}})\eul^{-ct\abs{k-\mu}^2},\quad k\in\C{X}.
\]
Thus
\[
\norm{\hat v^{(5)}-v^{\mu}}_{L^1(\C{X})}\leq C\int_0^{\epsilon}u(1+\abs{\ln u})\eul^{-ctu^2}\dd u\leq Ct^{-1}\ln t,\quad t\geq 0,
\]
and
\[
\norm{\hat v^{(5)}-v^{\mu}}_{L^{\infty}(\C{X})}\leq C\sup_{0\leq u\leq \epsilon}u(1+\abs{\ln u})\eul^{-ctu^2}\leq Ct^{-1/2}\ln t,\quad t\geq 0,
\]
which gives \eqref{v5-vmu-norms}.

Note that
\[
m^{\model}(m^{\mu})^{-1}-I=Y_{\mu}m^X(q,z(k))Y_{\mu}^{-1}-I.
\]
Since $\inf_{k\in\partial D_{\epsilon}(\mu)}\abs{\psi_{\mu}(k)}>0$, the variable $z=\ii\sqrt{t}(k-\mu)\psi_{\mu}(k)$ goes to infinity as $t\to\infty$ if $k\in\partial D_{\epsilon}(\mu)$. Thus equation \eqref{mX-at-z-infty} yields
\[
m^X(q,z(k))=I+\frac{m_1^X}{\sqrt{t}(k-\mu)\psi_{\mu}(k)}+\ord(t^{-1}),\quad t\to\infty,\quad k\in\partial D_{\epsilon}(\mu),\]
uniformly with respect to $k\in\partial D_{\epsilon}(\mu)$. Consequently,
\begin{equation}  \label{mmu-mmodel-t-infty}
m^{\model}(m^{\mu})^{-1}-I=\frac{Y_{\mu}m_1^XY_{\mu}^{-1}}{\sqrt{t}(k-\mu)\psi_{\mu}(k)}+\ord(t^{-1}),\quad t\to\infty,\quad k\in\partial D_{\epsilon}(\mu),
\end{equation}
uniformly with respect to $k\in\partial D_{\epsilon}(\mu)$. Since $Y_{\mu}^{\pm1}$ and $\psi_{\mu}^{-1}$ are bounded for $k\in\partial D_{\epsilon}(\mu)$, this proves \eqref{mmu-mmodel-sup-t-infty}. Using that the functions $Y_{\mu}^{\pm1}$ and $\psi_{\mu}^{-1}$ are analytic in $D_{\epsilon}(\mu)$, equation \eqref{mmu-mmodel-1-t-infty} follows from \eqref{mmu-mmodel-t-infty} and Cauchy's formula.
\end{proof}
%-------------------% 

%---------------------------------------------------------%
%:s.7
%---------------------------------------------------------%
\section{Final steps}    \label{sec:final}
%---------------------------------------------------------%
%:s.7.1
%---------------------------------------------------------%
\subsection{The approximate solution}

Define local solutions $m^{\bar\alpha}$ and $m^{\bar\beta}$ for $k\in D_{\epsilon}(\bar\alpha)$ and $k\in D_{\epsilon}(\bar\beta)$, respectively, by
\[
m^{\bar\alpha}(x,t,k)\coloneqq\sigma_3\sigma_1\overline{m^{\alpha}(x,t,\bar k)}\sigma_1\sigma_3,\qquad m^{\bar\beta}(x,t,k)\coloneqq\sigma_3\sigma_1\overline{m^{\beta}(x,t,\bar k)}\sigma_1\sigma_3.
\]
Let $\C{D}\subset\D{C}$ denote the union of the five open disks in \eqref{disks}, and $\partial\C{D}$ its boundary:
\[
\C{D}=D_\epsilon(\bar\alpha)\cup D_\epsilon(\bar\beta)\cup D_\epsilon(\mu)\cup D_\epsilon(\beta)\cup D_\epsilon(\alpha).
\]
Let $\C{X}$ denote the curved cross centered at $\mu$ defined in \eqref{cross}. The approximate solution $m^{\appr}$ defined by
\[
m^{\appr}\coloneqq\begin{cases}
m^{\alpha},&k\in D_{\epsilon}(\alpha),\\
m^{\beta},&k\in D_{\epsilon}(\beta),\\
m^{\bar\alpha},&k\in D_{\epsilon}(\bar\alpha),\\
m^{\bar\beta},&k\in D_{\epsilon}(\bar\beta),\\
m^{\mu},&k\in D_{\epsilon}(\mu),\\
m^{\model},&\text{elsewhere}
\end{cases}
\]
satisfies a RH problem with jump across the contour $\Sigma^{\appr}=\Sigma^{\model}\cup\partial\C{D}\cup\C{Y}\cup\C{Y}^*\cup\C{Z}\cup\C{Z}^*\cup\C{X}$ shown in Figure~\ref{fig:contour-approximate}, where $\C{Y}$, $\C{Y}^*$, $\C{Z}$, and $\C{Z}^*$ denote curved crosses centered at $\alpha$, $\bar\alpha$, $\beta$, and $\bar\beta$, respectively.
%-------------------%
%:fig 18 = 7.1
%-------------------%
\begin{figure}[ht]
\centering\includegraphics[scale=.8]{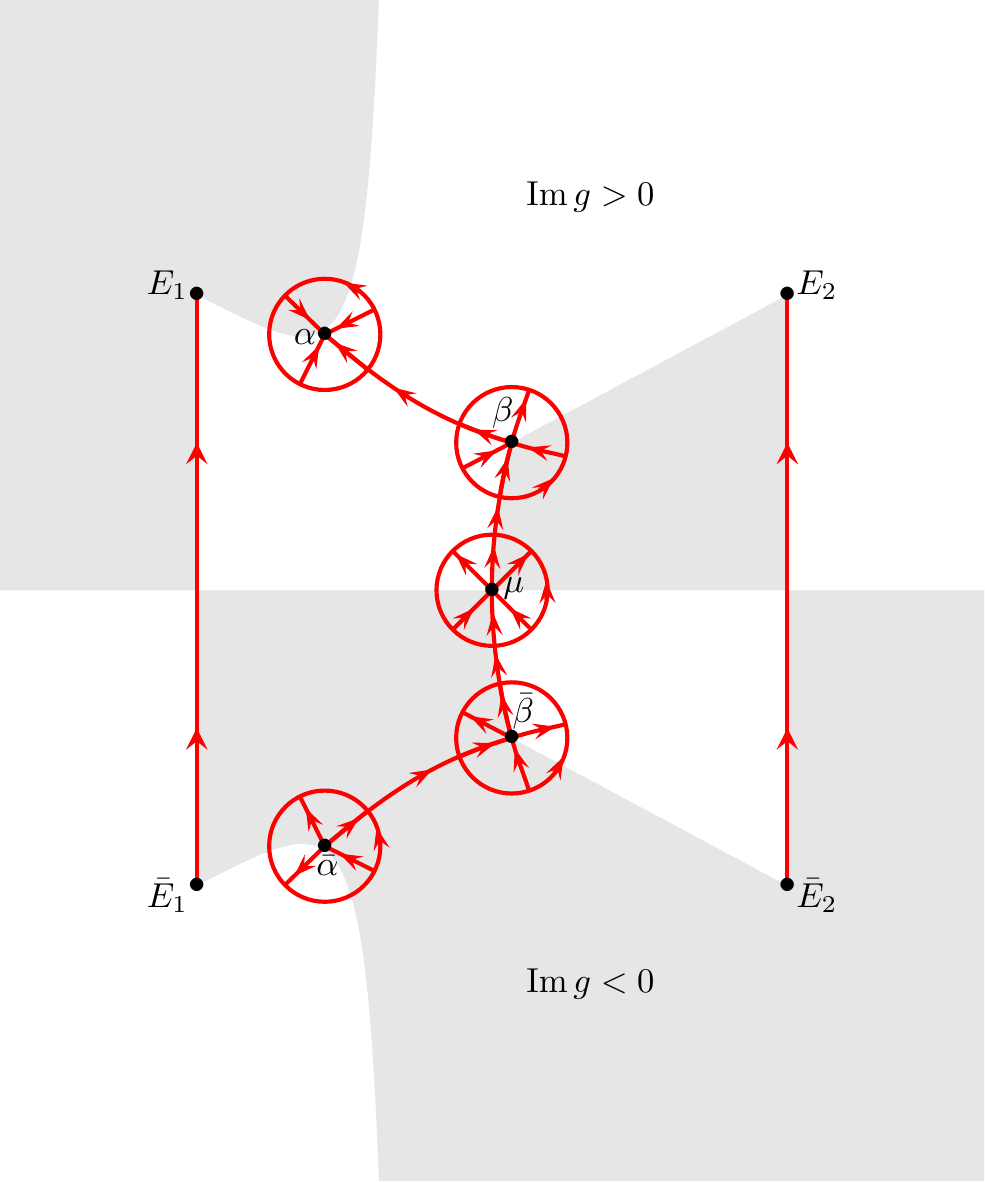}
\caption{The jump contour $\Sigma^{\appr}$.} 
\label{fig:contour-approximate}
\end{figure}
%-------------------%

%---------------------------------------------------------%
%:s.7.2
%---------------------------------------------------------%
\subsection{The solution $\BS{m^{\diff}}$}

%-------------------%
%:fig 19 = 7.2
%-------------------%
\begin{figure}[ht]
\centering\includegraphics[scale=.8]{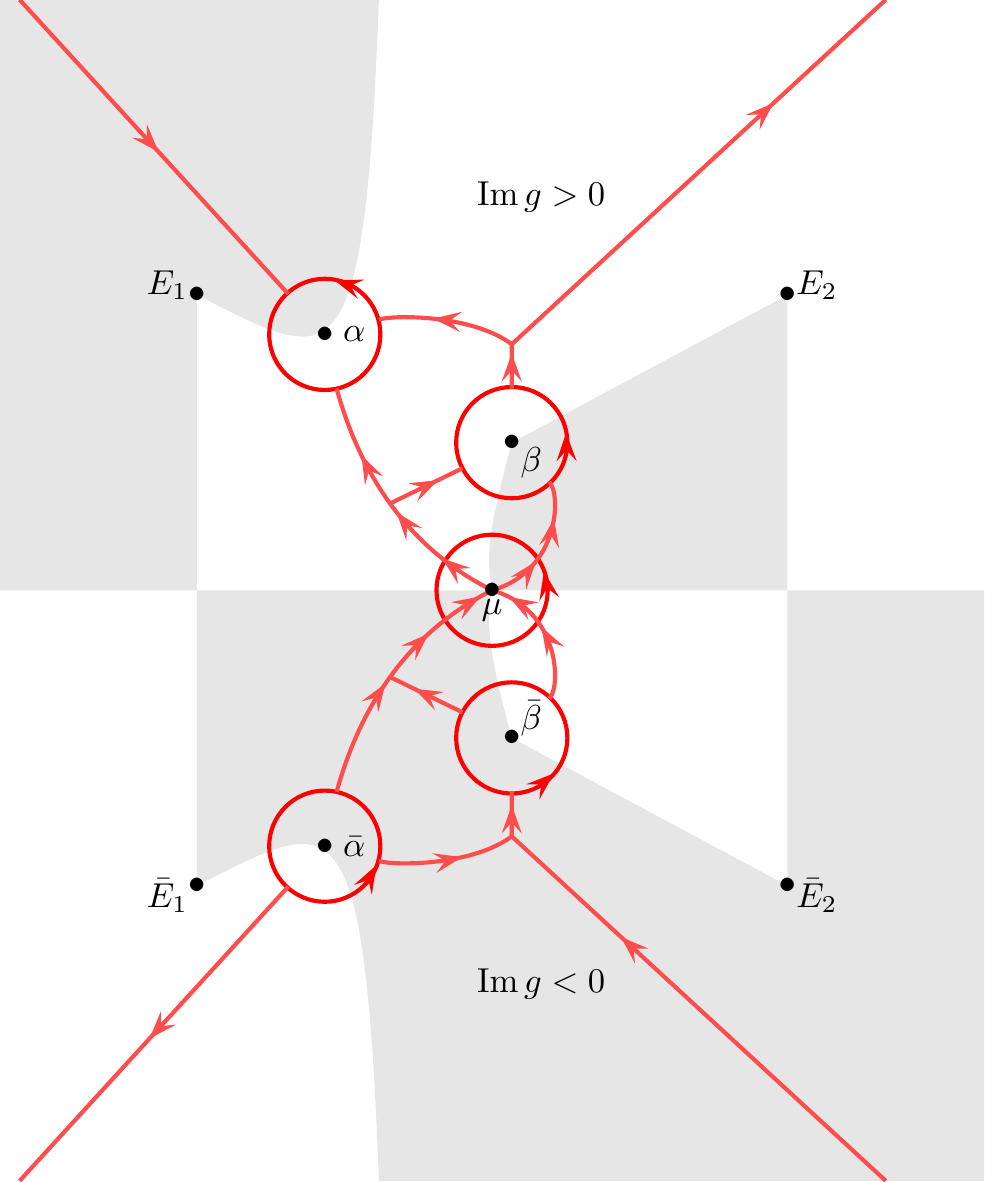}
\caption{The jump contour $\Sigma^{\diff}$.} 
\label{fig:contour-hat}
\end{figure}
%-------------------%

We will show that the function $m^{\diff}(x,t,k)$ defined by
\[
m^{\diff}\coloneqq\hat m^{(5)}(m^{\appr})^{-1}
\]
is such that $m^{\diff}-I$ is small for large $t$ and is solution of the RH problem
\begin{equation}  \label{rhp-mhat}
\begin{cases}
m^{\diff}(x,t,\,\cdot\,)\in I+\dot E^2(\D{C}\setminus\Sigma^{\diff}),&\\
m^{\diff}_+(x,t,k)=m^{\diff}_-(x,t,k)v^{\diff}(x,t,k)&\text{for a.e. }k\in\Sigma^{\diff},
\end{cases}
\end{equation}
where the contour $\Sigma^{\diff}=\left(\Sigma^{(5)}\setminus(\Sigma^{\model}\cup\bar{\C{D}})\right)\cup\partial\C{D}\cup\C{X}$ is displayed in Figure~\ref{fig:contour-hat} and the jump matrix $v^{\diff}$ is given by
\[
v^{\diff}\coloneqq\begin{cases}
m^{\model}\hat v^{(5)}(m^{\model})^{-1},&k\in\Sigma^{\diff}\setminus\bar{\C{D}},\\
m^{\model}(m^{\alpha})^{-1},&k\in\partial D_{\epsilon}(\alpha),\\
m^{\model}(m^{\beta})^{-1},&k\in\partial D_{\epsilon}(\beta),\\
m^{\model}(m^{\bar\alpha})^{-1},&k\in\partial D_{\epsilon}(\bar\alpha),\\
m^{\model}(m^{\bar\beta})^{-1},&k\in\partial D_{\epsilon}(\bar\beta),\\
m^{\model}(m^{\mu})^{-1},&k\in\partial D_{\epsilon}(\mu),\\
m_-^{\mu}\hat v^{(5)}(m_+^{\mu})^{-1}&k\in\C{X}.
\end{cases}
\]
Let $\hat w=v^{\diff}-I$. Away from the five critical points, $\hat v^{(5)}-I$ is exponentially small on $\Sigma^{(5)}\setminus\Sigma^{\model}$ as $t\to\infty$. Hence
\begin{subequations}   \label{what-t-infty}
\begin{equation}  \label{what-t-infty-1st}
\norm{\hat w}_{(L^1\cap L^2\cap L^{\infty})(\Sigma^{\diff}\setminus\bar{\C{D}})}=\ord(\eul^{-ct}),\quad t\to\infty,
\end{equation}
where $c\equiv c(\xi)>0$. On the other hand, equations \eqref{malpha-mmodel-t-infty} and \eqref{mbeta-mmodel-t-infty} imply that 
\begin{equation}   \label{what-t-infty-2nd}
\norm{\hat w}_{L^{\infty}(\partial D_{\epsilon}(\alpha)\cup\partial D_{\epsilon}(\beta)\cup\partial D_{\epsilon}(\bar\alpha)\cup\partial D_{\epsilon}(\bar\beta))}=\ord(t^{-N}),\quad t\to\infty,
\end{equation}
for any $N\geq 1$. According to equation \eqref{mmu-mmodel-sup-t-infty},
\begin{equation}\label{what-t-infty-3rd}
\norm{\hat w}_{L^{\infty}(\partial D_{\epsilon}(\mu))}=\ord(t^{-1/2}),\quad t\to\infty,
\end{equation}
and for $k\in\C{X}$, we have 
\[
\hat w=m_-^{\mu}(\hat v^{(5)}-v^{\mu})(m_+^{\mu})^{-1},
\]
so Lemma~\ref{mmu-properties} yields
\begin{equation}  \label{what-t-infty-4th}
\begin{cases}
\norm{\hat w}_{L^1(\C{X})}=\ord(t^{-1}\ln t),&\\
\norm{\hat w}_{L^2(\C{X})}=\ord(t^{-3/4}\ln t),&\\
\norm{\hat w}_{L^{\infty}(\C{X})}=\ord(t^{-1/2}\ln t),&
\end{cases}t\to\infty.
\end{equation}
\end{subequations}
Equations~\eqref{what-t-infty} show that
\begin{equation}
\begin{cases}
\norm{\hat w}_{(L^1\cap L^2)(\Sigma^{\diff})}=\ord(t^{-1/2}),&\\
\norm{\hat w}_{L^{\infty}(\Sigma^{\diff})}=\ord(t^{-1/2}\ln t),&
\end{cases}t\to\infty.
\end{equation}
Further on, we denote by $\hat{\C{C}}$ the Cauchy operator associated with $\Sigma^{\diff}$:
\[
(\hat{\C{C}}f)(k)=\frac{1}{2\pi\ii}\int_{\Sigma^{\diff}}\frac{f(s)}{s-k}\,\dd s,\quad k\in\D{C}\setminus\Sigma^{\diff},
\]
and by $\hat{\C{C}}_{\pm}f$ the nontangential boundary values of $\hat{\C{C}}f$ from the $\pm$ sides of $\Sigma^{\diff}$. We then define the operator $\hat{\C{C}}_{\hat w}\colon L^2(\Sigma^{\diff})\to L^2(\Sigma^{\diff})$ by $\hat{\C{C}}_{\hat w}f=\hat{\C{C}}_-(f\hat w)$ (see, e.g.,\cites{Le17,Le18}). Hence
\begin{equation}
\norm{\hat{\C{C}}_{\hat w}}_{\C{B}(L^2(\Sigma^{\diff}))}\leq C\norm{\hat w}_{L^{\infty}(\Sigma^{\diff})}=\ord(t^{-1/2}\ln t),\quad t\to\infty.
\end{equation}
In particular, $I-\hat{\C{C}}_{\hat w(x,t,\,\cdot\,)}\in\C{B}(L^2(\Sigma^{\diff}))$ is invertible for all large enough $t$. We define $\hat\mu(x,t,k)\in I+L^2(\Sigma^{\diff})$ for large $t$ by
\begin{equation}
\hat\mu\coloneqq I+(I-\hat{\C{C}}_{\hat w})^{-1}\hat{\C{C}}_{\hat w}.
\end{equation}
Standard estimates using the Neumann series show that
\[
\norm{\hat\mu-I}_{L^2(\Sigma^{\diff})}\leq C\frac{\norm{\hat w}_{L^2(\Sigma^{\diff})}}{1-\norm{\hat{\C{C}}_{\hat w}}_{\C{B}(L^2(\Sigma^{\diff}))}}.
\]
Thus,
\begin{equation}   \label{muhat-t-infty}
\norm{\hat\mu(x,t,\,\cdot\,)-I}_{L^2(\Sigma^{\diff})}=\ord(t^{-1/2}),\quad t\to\infty.
\end{equation}
It follows that there exists a unique solution $m^{\diff}\in I+\dot E^2(\hat{\D{C}}\setminus\Sigma^{\diff})$ of the RH problem
\eqref{rhp-mhat} for all sufficiently large $t$. This solution is given by
\begin{equation}
m^{\diff}(x,t,k)=I+\hat{\C{C}}(\hat\mu\hat w)=I+\frac{1}{2\pi\ii}\int_{\Sigma^{\diff}}\hat\mu(x,t,s)\hat w(x,t,s)\frac{\dd s}{s-k}.
\end{equation}

%---------------------------------------------------------%
%:s.7.3
%---------------------------------------------------------%
\subsection{Asymptotics of $\BS{m^{\diff}}$}

The exponential decay of $\hat w$ as $k\to\infty$, shows that the following limit exists uniformly with respect to $\arg k\in\croch{0,2\pi}$:
\begin{equation}   \label{mhat-k-infty}
\lim_{k\to\infty}k(m^{\diff}(x,t,k)-I)=-\frac{1}{2\pi\ii}\int_{\Sigma^{\diff}}\hat\mu(x,t,s)\hat w(x,t,s)\dd s.
\end{equation}
By \eqref{what-t-infty-1st} and \eqref{muhat-t-infty},
\begin{align*}  
\int_{\Sigma'}\hat\mu(x,t,s)\hat w(x,t,s)\dd s 
&=\int_{\Sigma'}\hat w(x,t,s)\dd s+\int_{\Sigma'}(\hat\mu(x,t,s)-I)\hat w(x,t,s)\dd s\\
&=\ord(\norm{\hat w}_{L^1(\Sigma')})+\ord(\norm{\hat\mu-I}_{L^2(\Sigma')}\norm{\hat w}_{L^2(\Sigma')})\\
&=\ord(\eul^{-ct}),\quad t\to\infty,
\end{align*}
where $\Sigma'\coloneqq\Sigma^{\diff}\setminus\bar{\C{D}}$. Hence the contribution to the integral in \eqref{mhat-k-infty} from $\Sigma^{\diff}\setminus\bar{\C{D}}$ is $\ord(\eul^{-ct})$. Similarly, the contribution from $\partial D_{\epsilon}(\alpha)\cup\partial D_{\epsilon}(\beta)\cup\partial D_{\epsilon}(\bar\alpha)\cup\partial D_{\epsilon}(\bar\beta)$ is $\ord(t^{-N})$ for any $N\geq 1$. By \eqref{mmu-mmodel-1-t-infty}, \eqref{what-t-infty-3rd}, and \eqref{muhat-t-infty}, the contribution from $\partial D_{\epsilon}(\mu)$ to the right-hand side of \eqref{mhat-k-infty} is
\begin{align*}
&-\frac{1}{2\pi\ii}\int_{\partial D_{\epsilon}(\mu)}\hat w(x,t,k)\dd k-\frac{1}{2\pi\ii}\int_{\partial D_{\epsilon}(\mu)}(\hat\mu(x,t,k)-I)\hat w(x,t,k)\dd k\\
&=-\frac{1}{2\pi\ii}\int_{\partial D_{\epsilon}(\mu)}\left(m^{\model}(m^{\mu})^{-1}-I\right)\dd k+\ord(\norm{\hat\mu-I}_{L^2(\partial D_{\epsilon}(\mu))}\norm{\hat w}_{L^2(\partial D_{\epsilon}(\mu))})\\
&=-\frac{Y_{\mu}(x,t,\mu)m_1^XY_{\mu}(x,t,\mu)^{-1}}{\sqrt{t}\psi_{\mu}(\mu)}+\ord(t^{-1}),\quad t\to\infty.
\end{align*}
Finally, by \eqref{what-t-infty-4th} and \eqref{muhat-t-infty}, the contribution from $\C{X}$ to the right-hand side of \eqref{mhat-k-infty} is
\[
\ord(\norm{\hat w}_{L^1(\C{X})})+\ord(\norm{\hat\mu-I}_{L^2(\C{X})}\norm{\hat w}_{L^2(\C{X})})=\ord(t^{-1}\ln t),\quad t\to\infty.
\]
Collecting the above contributions, we find from \eqref{mhat-k-infty} that
\begin{equation}   \label{mhat-t-infty}
\lim_{k\to\infty}k(m^{\diff}(x,t,k)-I)=-\frac{Y_{\mu}(x,t,\mu)m_1^XY_{\mu}(x,t,\mu)^{-1}}{\sqrt{t}\psi_{\mu}(\mu)}+\ord(t^{-1}\ln t),\quad t\to\infty.
\end{equation}

%---------------------------------------------------------%
%:s.7.4
%---------------------------------------------------------%
\subsection{Asymptotics of $\BS{q}$}

Recalling the many transformations of Section~\ref{sec:transfos}, we find
\[
\hat m=\eul^{\ii tg^{(0)}\sigma_3}\eul^{\ii h(\infty)\sigma_3}m^{\diff}m^{\model}\eul^{-\ii h\sigma_3}\delta^{\sigma_3}\eul^{-\ii t(g(k)-\theta(k))\sigma_3}
\]
for all large $k$ in $U_2$. It follows that
\begin{align*}
\lim_{k\to\infty}k(\hat m(x,t,k)-I) 
&=\eul^{\ii tg^{(0)}\sigma_3}\eul^{\ii h(\infty)\sigma_3}\lim_{k\to\infty}k(m^{\model}-I+(m^{\diff}-I)m^{\model})\eul^{-\ii h(\infty)\sigma_3}\eul^{-\ii tg^{(0)}\sigma_3}\\
&=\eul^{\ii(tg^{(0)}+h(\infty))\hat\sigma_3}\left(\lim_{k\to\infty}k(m^{\model}-I)+\lim_{k\to\infty}k(m^{\diff}-I)\right).
\end{align*}
Hence
\begin{align*}
q(x,t)&=2\ii\lim_{k\to\infty}(\hat m(x,t,k))_{12}\\
&=2\ii\eul^{2\ii(tg^{(0)}+h(\infty))}\left(\lim_{k\to\infty}k(m^{\model}(x,t,k))_{12}+\lim_{k\to\infty}k(m^{\diff}(x,t,k))_{12}\right).
\end{align*}
In view of \eqref{mmodel-at-infty} and \eqref{mhat-t-infty}, this yields \eqref{main-asymptotics} and completes the proof of Theorem~\ref{main-thm}.
%---------------------------------------------------------%
%:app.A
%---------------------------------------------------------%
\appendix
\section{Exact solution in terms of parabolic cylinder functions}\label{sec:A}

Let $X$ denote the cross $X\coloneqq X_1\cup\dots\cup X_4\subset\D{C}$, where the rays 
\begin{alignat}{2}  \label{Xrays-def}
&X_1\coloneqq\accol{s\eul^{\frac{\ii\pi}{4}}\mid 0\leq s<\infty},&\quad&X_2\coloneqq\accol{s\eul^{\frac{3\ii\pi}{4}}\mid 0\leq s<\infty},\notag\\
&X_3\coloneqq\accol{s\eul^{-\frac{3\ii\pi}{4}}\mid 0\leq s<\infty},&&X_4\coloneqq\accol{s\eul^{-\frac{\ii\pi}{4}}\mid 0\leq s<\infty}
\end{alignat}
are oriented toward the origin as in Figure~\ref{fig:contour-X}. 
%-------------------%
%:fig 20 = A.1
%-------------------%
\begin{figure}[ht]
\centering\includegraphics[scale=.9]{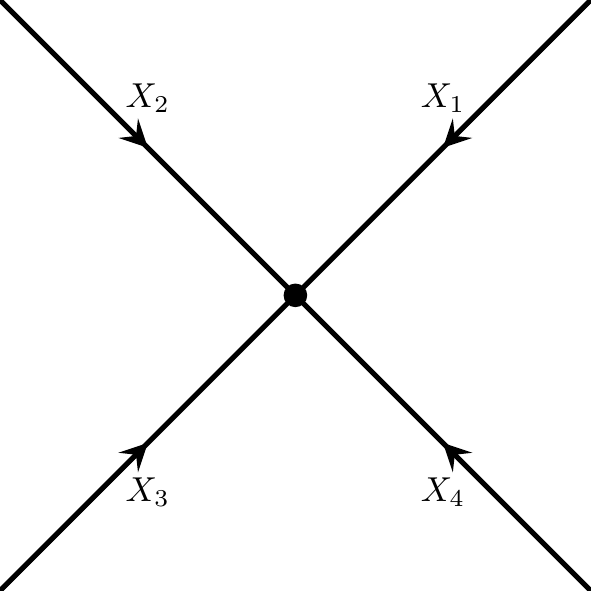}
\caption{The contour $X=X_1\cup\dots\cup X_4$.} 
\label{fig:contour-X}
\end{figure}
%-------------------%

Define the function $\nu\colon\D{C}\to\lbrack 0,\infty)$ by $\nu(q)\coloneqq\frac{1}{2\pi}\ln(1+\abs{q}^2)$. Let $\rho(q,z)\coloneqq\eul^{\ii\nu(q)\ln_{-\pi/2}z}$, i.e., $\rho(q,z)=z^{\ii\nu(q)}$ with the branch cut along the negative imaginary axis.

We consider the following family of RH problems parametrized by $q\in\D{C}$:
\begin{equation}  \label{rhp-A}
\begin{cases}
m^X(q,\,\cdot\,)\in I+\dot E^2(\D{C}\setminus X),&\\
m_+^X(q,z)=m_-^X(q,z)v^X(q,z)&\text{for a.e. }z\in X,
\end{cases}
\end{equation}
where the jump matrix $v^X(q,z)$ is defined by 
\begin{equation}    \label{jump-A}
v^X(q,z)\coloneqq\begin{cases}
\begin{pmatrix}1&0\\q\rho(q,z)^{-2}\eul^{2\ii z^2}&1\end{pmatrix},&k\in X_1,\\
\begin{pmatrix}1&\bar q\rho(q,z)^2\eul^{-2\ii z^2}\\0&1\end{pmatrix},&k\in X_2,\\
\begin{pmatrix}1&0\\-\frac{q}{1+\abs{q}^2}\rho(q,z)^{-2}\eul^{2\ii z^2}&1\end{pmatrix},&k\in X_3,\\
\begin{pmatrix}1&-\frac{\bar q}{1+\abs{q}^2}\rho(q,z)^2\eul^{-2\ii z^2}\\0&1\end{pmatrix},&k\in X_4.
\end{cases}
\end{equation}
The matrix $v^X$ has entries that oscillate rapidly as $z\to 0$ and $v^X$ is not continuous at $z=0$; however $v^X(q,\,\cdot\,)-I\in L^2(X)\cap L^{\infty}(X)$. The RH problem \eqref{rhp-A} can be solved explicitly in terms of parabolic cylinder functions \cite{I81}.

%-------------------%
%:thm A.1
%-------------------%
\begin{theorem}   \label{thm-A}
The RH problem \eqref{rhp-A} has a unique solution $m^X(q,z)$ for each $q\in\D{C}$. This solution satisfies
\begin{equation}   \label{mX-at-z-infty}
m^X(q,z)=I+\frac{\ii}{z}
\begin{pmatrix}
0&-\eul^{-\pi\nu}\beta^X(q)\\\eul^{\pi\nu}\overline{\beta^X(q)}&0\end{pmatrix}+\ord\left(\frac{1}{z^2}\right),\quad z\to\infty,\quad q\in\D{C},
\end{equation}
where the error term is uniform with respect to $\arg z\in\croch{0,2\pi}$ and $q$ in compact subsets of $\D{C}$, and the function $\beta^X(q)$ is defined by
\begin{equation}    \label{betaX-def}
\beta^X(q)\coloneqq\frac{\sqrt{\nu(q)}}{2}\eul^{\ii\left(-\frac{3\pi}{4}-2\nu(q)\ln 2-\arg q+\arg\Gamma(\ii\nu(q))\right)},\quad q\in\D{C}.
\end{equation}
Moreover, for each compact subset $K\subset\D{C}$,
\begin{equation}   \label{mX-q-K-bounded}
\sup_{q\in K}\sup_{z\in\D{C}\setminus X}\abs{m^X(q,z)}<\infty.  
\end{equation}
\end{theorem}
%-------------------%

%-------------------%
\begin{proof}
Uniqueness follows because $\det v^X=1$. Define $m^{(X1)}$ by 
\begin{equation}  \label{mX1-def}
m^{(X1)}(q,z)\coloneqq m^X(q,z)\rho(q,z)^{\sigma_3}\eul^{-\ii z^2\sigma_3}.
\end{equation}
Using that $\rho(q,\,\cdot\,)$ has a branch cut along the negative imaginary axis, we see that $m^X$ satisfies the jump condition in \eqref{rhp-A} iff $m^{(X1)}$ satisfies
\[
m_+^{(X1)}(q,z)=m_-^{(X1)}(q,z)v^{(X1)}(q,z)\quad\text{for a.e. }z\in X \cup\ii\,\D{R}_-,
\]
where
\[
v^{(X1)}\coloneqq\begin{cases}
\begin{pmatrix}1&0\\q&1\end{pmatrix},&z\in X_1,\\
\begin{pmatrix}1&\bar q\\0&1\end{pmatrix},&z\in X_2,\\
\begin{pmatrix}1&0\\-\frac{q}{1+\abs{q}^2}&1\end{pmatrix},&z\in X_3,\\
\left(\frac{\rho_+(q,z)}{\rho_-(q,z)}\right)^{\sigma_3}=\begin{pmatrix}\frac{1}{1+\abs{q}^2}&0\\0&1+\abs{q}^2\end{pmatrix},&z\in\ii\,\D{R}_-,\\
\begin{pmatrix}1&-\frac{\bar q}{1+\abs{q}^2}\\0&1\end{pmatrix},&z\in X_4,
\end{cases}
\]
with all contours oriented toward the origin. We next merge the contours $X_1$ and $X_2$ along $\ii\,\D{R}_+$ and the contours $X_3$ and $X_4$ along $\ii\,\D{R}_-$. Thus we let
\begin{equation}   \label{psi-def}
\psi(q,z)=m^{(X1)}(q,z)B(q,z),
\end{equation}
where
\[
B(z)\coloneqq\begin{cases}
\begin{pmatrix}1&0\\q&1\end{pmatrix},&\arg z\in\left(\frac{\pi}{4},\frac{\pi}{2}\right),\\
\begin{pmatrix}1&-\bar q\\0&1\end{pmatrix},&\arg z\in\left(\frac{\pi}{2},\frac{3\pi}{4}\right),\\
\begin{pmatrix}1&0\\-\frac{q}{1+\abs{q}^2}&1\end{pmatrix},&\arg z\in\left(\frac{5\pi}{4},\frac{3\pi}{2}\right),\\
\begin{pmatrix}1&\frac{\bar q}{1+\abs{q}^2}\\0&1\end{pmatrix},&\arg z\in\left(\frac{3\pi}{2},\frac{7\pi}{4}\right),\\
\,I,&\text{else}.
\end{cases}
\]
Then $m^X$ satisfies the jump condition in \eqref{rhp-A} iff $\psi$ satisfies
\begin{equation}   \label{psi-jump}
\psi_+(q,z)=\psi_-(q,z)v^{\psi}(q)\quad\text{for a.e. }z\in\ii\,\D{R},
\end{equation}
where the constant matrix $v^{\psi}$ is defined by
\begin{equation}  \label{vpsi-def}
v^{\psi}(q)\coloneqq\begin{pmatrix}1+\abs{q}^2&\bar q\\q&1\end{pmatrix},\quad k\in\ii\,\D{R},
\end{equation}
and the contour $\ii\,\D{R}$ is oriented downward.

Let
\begin{equation}  \label{betaX-def-bis}
\beta^X\coloneqq -\frac{\eul^{\frac{\pi\ii}{4}}2^{-2\ii\nu}\eul^{\frac{\pi\nu}{2}}\sqrt{\pi}}{q\sqrt{2}\,\Gamma(-\ii\nu)}.
\end{equation}
Then $\beta^X$ can be written as in \eqref{betaX-def}. Define a sectionally analytic function $\psi(q,z)$ by
\begin{equation}  \label{psi-def-bis}
\psi(q,z)\coloneqq\begin{pmatrix}
\psi_{11}(q,z)&\frac{\left(\frac{\dd}{\dd z}-2\ii z\right)\psi_{22}(q,z)}{4\eul^{\pi\nu}\overline{\beta^X(q)}}\\
\frac{\left(\frac{\dd}{\dd z}+2\ii z\right)\psi_{11}(q,z)}{4\eul^{-\pi\nu}\beta^X(q)}&\psi_{22}(q,z)\end{pmatrix},\quad q\in\D{C},\quad z\in\D{C}\setminus\ii\,\D{R},
\end{equation}
where the functions $\psi_{11}$ and $\psi_{22}$ are defined by
\begin{subequations}  \label{psi-diag}
\begin{align}  \label{psi-11}
\psi_{11}(q,z)&\coloneqq\begin{cases}
2^{-\ii\nu}\eul^{-\frac{3\pi\nu}{4}}D_{\ii\nu}(2\eul^{-\frac{3\pi\ii}{4}}z),&\Re z<0,\\
2^{-\ii\nu}\eul^{\frac{\pi\nu}{4}}D_{\ii\nu}(2\eul^{\frac{\pi\ii}{4}}z),&\Re z>0,
\end{cases}\\  \label{psi-22}
\psi_{22}(q,z)&\coloneqq\begin{cases}
2^{\ii\nu}\eul^{\frac{5\pi\nu}{4}}D_{-\ii\nu}(2\eul^{\frac{3\pi\ii}{4}}z),&\Re z<0,\\
2^{\ii\nu}\eul^{\frac{\pi\nu}{4}}D_{-\ii\nu}(2\eul^{-\frac{\pi\ii}{4}}z),&\Re z>0,
\end{cases}
\end{align}
\end{subequations}      
and $D_a(z)$ denotes the parabolic cylinder function.

Since $D_a(z)$ is an entire function of both $a$ and $z$, $\psi(q,z)$ is analytic in the left and right halves of the complex $z$-plane with a jump across the imaginary axis. The function $\psi$ satisfies
\[
\partial_z\psi+2\ii z\sigma_3\psi=4\begin{pmatrix}0&\beta^X\eul^{-\pi\nu}\\\bar\beta^X\eul^{\pi\nu}&0\end{pmatrix}\psi,\quad q\in\D{C},\quad z\in\D{C}\setminus\ii\,\D{R}.
\]
Since $\psi_+$ and $\psi_-$ solve the same second order ODE, there exists a function $v^{\psi}$ independent of $z$ such that \eqref{psi-jump} holds. Setting $z=0$ and using that
\[
D_{\ii\nu}(0)=\frac{\sqrt{\pi}\,2^{\frac{\ii\nu}{2}}}{\Gamma\left(\frac{1}{2}(1-\ii\nu)\right)},\quad D_{\ii\nu}'(0)=-\frac{\sqrt{\pi}\,2^{\frac{1}{2}(1+\ii\nu)}}{\Gamma\left(-\frac{\ii\nu}{2}\right)}\,,
\]
we find
\[
v^{\psi}(q)=\psi_-(q,0)^{-1}\psi_+(q,0)=\begin{pmatrix}1+\abs{q}^2&\bar q\\q&1\end{pmatrix}.
\]
Hence $\psi$ satisfies \eqref{psi-jump}. This shows that $m^X$ satisfies the jump condition in \eqref{rhp-A}.

For each $\delta>0$, the parabolic cylinder function satisfies the asymptotic formula \cite{Ol10}
\begin{align*}
D_a(z)&=z^a\eul^{-\frac{z^2}{4}}\left(1-\frac{a(a-1)}{2z^2}+\ord(z^{-4})\right)\\
&\quad -\frac{\sqrt{2\pi}\eul^{\frac{z^2}{4}}z^{-a-1}}{\Gamma(-a)}\left(1+\frac{(a+1)(a+2)}{2z^2}+\ord(z^{-4})\right)\\
&\qquad\times\begin{cases}
0,&\arg z\in\croch{-\frac{3\pi}{4}+\delta,\frac{3\pi}{4}-\delta},\\
\eul^{\ii\pi a},&\arg z\in\croch{\frac{\pi}{4}+\delta,\frac{5\pi}{4}-\delta},\\
\eul^{-\ii\pi a},&\arg z\in\croch{-\frac{5\pi}{4}+\delta,-\frac{\pi}{4}-\delta},
\end{cases}\quad z\to\infty,\quad a\in\D{C},
\end{align*}
where the error terms are uniform with respect to $a$ in compact subsets and $\arg z$ in
the given ranges. Using this formula and the identity
\[
\frac{\dd}{\dd z}D_a(z)=\frac{z}{2}D_a(z)-D_{a+1}(z),
\]
the asymptotic equation \eqref{mX-at-z-infty} follows from a tedious but straightforward computation. The boundedness in \eqref{mX-q-K-bounded} is a consequence of \eqref{mX-at-z-infty} and the definition \eqref{psi-def-bis} of $\psi(q,z)$.
\end{proof}
%-------------------%

%-------------------%
%:rem A.2
%-------------------%
\begin{remark}
The definition \eqref{psi-def-bis} of $\psi$ can be motivated as follows. Since $(\partial_z+2\ii z\sigma_3)\psi$ and $\psi$ have the same jump across $\ii\,\D{R}$, the function $(\partial_z\psi+2\ii z\sigma_3\psi)\psi^{-1}$ is entire. Suppose $m^X=1+m_1^X(q)z^{-1}+\ord(z^{-2})$ as $z\to\infty$, where the matrix $m_1^X(q)$ is independent of $z$. Then we expect
\begin{equation}   \label{psi-at-z-infty}
\psi=\left(I+\frac{m_1^X}{z}+\ord(z^{-2})\right)\rho^{\sigma_3}\eul^{-\ii z^2\sigma_3},\quad z\to\infty,
\end{equation}
which suggests that
\begin{equation}\label{psi-asymptot-eq}
(\partial_z\psi+2\ii z\sigma_3\psi)\psi^{-1}=2\ii\croch{\sigma_3,m_1^X}+\ord(z^{-1}),\quad z\to\infty.
\end{equation}
Strictly speaking, due to the factor $B(q,z)$ in \eqref{psi-def}, equation \eqref{psi-at-z-infty} is not valid for $z$ close to $\ii\,\D{R}$. Equation \eqref{psi-asymptot-eq} implies that $(\partial_z\psi+2\ii z\sigma_3\psi)\psi^{-1}$ is bounded; hence a constant. Thus
\[
\partial_z\psi+2\ii z\sigma_3\psi=2\ii\croch{\sigma_3,m_1^X}\psi.
\]
Letting
\[
\beta_{12}=4\ii(m_1^X)_{12},\qquad\beta_{21}=-4\ii(m_1^X)_{21},
\]
we find that the $(11)$ and $(22)$ entries of $\psi$ satisfy the equations
\[
\begin{cases}
\partial_z^2\psi_{11}+(4z^2+2\ii-\beta_{12}\beta_{21})\psi_{11}=0,&\\ \partial_z^2\psi_{22}+(4z^2-2\ii-\beta_{12}\beta_{21})\psi_{22}=0,&
\end{cases}z\in\D{C}\setminus\ii\,\D{R}.
\] 
Introducing the new variable $\zeta$ by $\zeta\coloneqq 2\eul^{-\frac{3\pi\ii}{4}}z$, we find that $f(\zeta)\coloneqq\psi_{11}(z)$ satisfies the parabolic cylinder equation
\[
\partial_{\zeta}^2f+\left(\frac{1}{2}-\frac{\zeta^2}{4}+a\right)f=0
\]
for $a=\frac{\ii}{4}\beta_{12}\beta_{21}$. This means that
\[
\psi_{11}(z)=\begin{cases}
c_1D_a(2\eul^{-\frac{3\pi\ii}{4}}z)+c_2D_a(2\eul^{\frac{\pi\ii}{4}}z),&\Re z>0,\\
c_3D_a(2\eul^{-\frac{3\pi\ii}{4}}z)+c_4D_a(2\eul^{\frac{\pi\ii}{4}}z),&\Re z<0,
\end{cases}
\]
for some constants $\accol{c_j}_1^4$. Since
\[
D_a(\zeta)=\zeta^a\eul^{-\frac{\zeta^2}{4}}\left(1+\ord(\zeta^{-2})\right),\quad\zeta\to\infty,\quad\abs{\arg\zeta}<\frac{3\pi}{4}-\delta,
\]
the sought-after asymptotics \eqref{psi-at-z-infty} of $\psi$ can be obtained by choosing
\[
a=\ii\nu,\quad c_1=0,\quad c_2=2^{-\ii\nu}\eul^{\frac{\pi\nu}{4}},\quad c_3=2^{-\ii\nu}\eul^{-\frac{3\pi\nu}{4}},\quad c_4=0.
\]
This yields the expression \eqref{psi-11} for $\psi_{11}$; the expression \eqref{psi-22} for $\psi_{22}$ is derived in a similar way. The functions $\psi_{12}$ and $\psi_{21}$ can be obtained from the equations
\[
\psi_{21}=\beta_{12}^{-1}(\psi_{11}'(z)+2\ii z\psi_{11}(z)),\qquad\psi_{12}=\beta_{21}^{-1}(\psi_{22}'(z)-2\ii z\psi_{22}(z)).
\]
Using that $\det\psi=1$, the $(21)$ and $(12)$ entries of the jump condition $(\psi_-)^{-1}\psi_+=v^{\psi}$ evaluated at $z=0$ then yield
\[
v_{21}^{\psi}(q)=-\frac{\sqrt{2\pi}\eul^{\frac{\pi\ii}{4}} 2^{-2\ii\nu+1}\eul^{-\frac{\pi\nu}{2}}}{\beta_{12}\Gamma(-\ii\nu)},\qquad v_{12}^{\psi}(q)=-\frac{\sqrt{2\pi}\eul^{-\frac{\pi\ii}{4}} 2^{2\ii\nu+1}\eul^{\frac{3\pi\nu}{2}}}{\beta_{21}\Gamma(\ii\nu)}.
\]
Thus $v^{\psi}$ has the form \eqref{vpsi-def} provided that $\beta_{12}=4\eul^{-\pi\nu}\beta^X$ and $\beta_{21}=4\eul^{\pi\nu}\bar\beta^X$ with $\beta^X$ given by \eqref{betaX-def-bis}. This motivates the form of equation \eqref{psi-def-bis}.
\end{remark}
%-------------------%

%---------------------------------------------------------%
%:app.B
%---------------------------------------------------------%
\section{Exact solution in terms of Airy functions}\label{sec:B}

Let $Y\coloneqq Y_1\cup\dots\cup Y_4\subset\D{C}$ denote the union of the four rays
\begin{alignat}{2}  \label{Yrays-def}
&Y_1\coloneqq\accol{s\mid 0\leq s<\infty},&\qquad&Y_2\coloneqq\accol{s\eul^{\frac{2\ii\pi}{3}}\mid 0\leq s<\infty},\notag\\
&Y_3\coloneqq\accol{-s\mid 0\leq s<\infty},&&Y_4\coloneqq\accol{s\eul^{-\frac{2\ii\pi}{3}}\mid 0\leq s<\infty},
\end{alignat}
oriented toward the origin as in Figure~\ref{fig:rays-Y}.
%-------------------%
%:fig 21 = B.1
%-------------------%
\begin{figure}[ht]
\centering\includegraphics[scale=1]{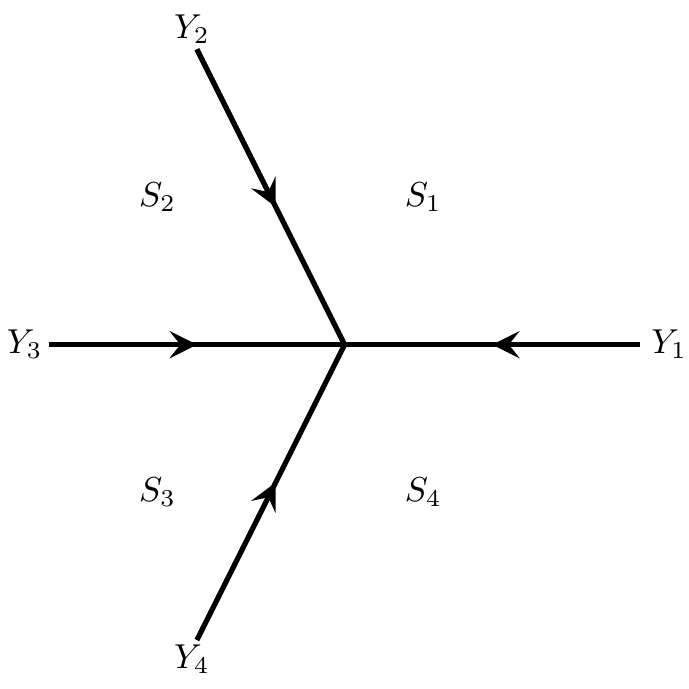}
\caption{The sectors $S_j$, $j=1,\dots,4$.} 
\label{fig:rays-Y}
\end{figure}
%-------------------%
Define the open sectors $\accol{S_j}_1^4$ by 
\begin{alignat*}{2}
&S_1\coloneqq\accol{\arg k\in(0,\tfrac{2\pi}{3})},&\qquad&S_2\coloneqq\accol{\arg k\in(\tfrac{2\pi}{3},\pi)},\\
&S_3\coloneqq\accol{\arg k\in(\pi,\tfrac{4\pi}{3})},&&S_4\coloneqq\accol{\arg k\in(\tfrac{4\pi}{3},2\pi)}.
\end{alignat*}
Let $\omega\coloneqq\eul^{\frac{2\pi\ii}{3}}$. Define the function $m^{\Airy}(\zeta)$ for $\zeta\in\D{C}\setminus Y$ by
\begin{equation}  \label{mAiry-def}
m^{\Airy}(\zeta)\coloneqq\Psi(\zeta)\times\begin{cases}
\,\eul^{\frac{2}{3}\zeta^{3/2}\sigma_3},&\zeta\in S_1\cup S_4,\\
\begin{pmatrix}1&0\\-1&1\end{pmatrix}\eul^{\frac{2}{3}\zeta^{3/2}\sigma_3},&\zeta\in S_2,\\
\begin{pmatrix}1&0\\1&1\end{pmatrix}\eul^{\frac{2}{3}\zeta^{3/2}\sigma_3},&\zeta\in S_3,
\end{cases}
\end{equation}
where
\[
\Psi(\zeta)\coloneqq\begin{cases}
\begin{pmatrix}\Airy(\zeta)&\Airy(\omega^2\zeta)\\\Airy'(\zeta)&\omega^2\Airy'(\omega^2\zeta)\end{pmatrix}\eul^{-\frac{\pi\ii}{6}\sigma_3},&\zeta\in\D{C}^+,\\ 
\begin{pmatrix}\Airy(\zeta)&-\omega^2\Airy(\omega\zeta)\\\Airy'(\zeta)&-\Airy'(\omega\zeta)\end{pmatrix}\eul^{-\frac{\pi\ii}{6}\sigma_3},&\zeta\in\D{C}^-.
\end{cases}
\]
Note that $\det m^{\Airy}(\zeta)=\eul^{\pi\ii/6}/(2\pi)$ is constant and nonzero. For each integer $N\geq 0$, define the asymptotic approximations $m_{\asympt,N}^{\Airy}(\zeta)$ and $m_{\asympt,N}^{\Airy,\inv}(\zeta)$ of $m^{\Airy}(\zeta)$ and $m^{\Airy}(\zeta)^{-1}$, respectively, by
\begin{subequations}  \label{mAiry-approx-def}
\begin{alignat}{2}
m_{\asympt,N}^{\Airy}(\zeta)
&\coloneqq\frac{\eul^{\frac{\ii\pi}{12}}}{2\sqrt{\pi}}\sum_{k=0}^N\frac{1}{\left(\frac{2}{3}\zeta^{3/2}\right)^k}\,\zeta^{-\frac{1}{4}\sigma_3}\!\begin{pmatrix}(-1)^ku_k&u_k\\-(-1)^k\nu_k&\nu_k\end{pmatrix}\eul^{-\frac{\pi\ii}{4}\sigma_3},&\quad&\zeta\in\D{C}\setminus Y,\\
\label{mAiry-approx-defb}
m_{\asympt,N}^{\Airy,\inv}(\zeta)
&\coloneqq\sqrt{\pi}\eul^{-\frac{\ii\pi}{12}}\sum_{k=0}^N\frac{1}{\left(\frac{2}{3}\zeta^{3/2}\right)^k}\eul^{\frac{\pi\ii}{4}\sigma_3}\begin{pmatrix}\nu_k & -u_k\\(-1)^k\nu_k &(-1)^k u_k\end{pmatrix}\zeta^{\frac{1}{4}\sigma_3},&&\zeta\in\D{C}\setminus Y,
\end{alignat}
\end{subequations}
where the real constants $\accol{u_j,\nu_j}_0^{\infty}$ are defined by $u_0\coloneqq\nu_0\coloneqq 1$ and
\[
u_k\coloneqq\frac{(2k+1)(2k+3)\dots(6k-1)}{(216)^kk!}\,,\quad\nu_k\coloneqq\frac{6k+1}{1-6k}u_k,\quad k=1,2,\dots
\]
For $N=0$, we have
\[
m_{\asympt,0}^{\Airy}(\zeta)=m_{\asympt,0}^{\Airy,\inv}(\zeta)^{-1}=\frac{\eul^{\frac{\ii\pi}{12}}}{2\sqrt{\pi}}\begin{pmatrix}\zeta^{-\frac{1}{4}}&0\\0&\zeta^{\frac{1}{4}}\end{pmatrix}\begin{pmatrix}1&1\\-1&1\end{pmatrix}\eul^{-\frac{\pi\ii}{4}\sigma_3}.
\]

%-------------------%
%:thm B.1
%-------------------%
\begin{theorem}   \label{thm-B}
\emph{(a)} The function $m^{\Airy}(\zeta)$ defined in \eqref{mAiry-def} is analytic for $\zeta\in\D{C}\setminus Y$ and satisfies the jump condition
\begin{equation}  \label{mAiry-jump}
m_+^{\Airy}(\zeta)=m_-^{\Airy}(\zeta)v^{\Airy}(\zeta),\quad\zeta\in Y\setminus\accol{0},
\end{equation}
where the jump matrix $v^{\Airy}$ is defined by
\[
v^{\Airy}(\zeta)\coloneqq\begin{cases}
\begin{pmatrix}1&-\eul^{-\frac{4}{3}\zeta^{3/2}}\\0&1\end{pmatrix},&\zeta\in Y_1,\\
\begin{pmatrix}1&0\\\eul^{\frac{4}{3}\zeta^{3/2}}&1\end{pmatrix},&\zeta\in Y_2\cup Y_4,\\
\begin{pmatrix}0&1\\-1&0\end{pmatrix},&\zeta\in Y_3.
\end{cases}
\]

\emph{(b)}
For each integer $N\geq 0$, the functions $m_{\asympt,N}^{\Airy}(\zeta)$ and $m_{\asympt,N}^{\Airy,\inv}(\zeta)$ are analytic for $\zeta\in\D{C}\setminus(-\infty,0\rbrack$ and satisfy the following jump relations on the negative real axis:
\begin{subequations}  \label{mAiry-approx-jump}
\begin{alignat}{2}
m_{\asympt,N+}^{\Airy}(\zeta)&=m_{\asympt,N-}^{\Airy}(\zeta)\begin{pmatrix}0&1\\-1&0\end{pmatrix},&\quad&\zeta<0,\\
\label{mAiry-approx-jumpb}
m_{\asympt,N+}^{\Airy,\inv}(\zeta)&=\begin{pmatrix}0&-1\\1&0\end{pmatrix}m_{\asympt,N-}^{\Airy,\inv}(\zeta),&&\zeta<0.
\end{alignat}
\end{subequations}

\emph{(c)}
The functions $m_{\asympt,N}^{\Airy}$ and $m_{\asympt,N}^{\Airy,\inv}$ approximate $m^{\Airy}$ and its inverse as $\zeta\to\infty$ in the sense that
\begin{subequations}   \label{mAiry-approx-mAiry}
\begin{alignat}{2}
m^{\Airy}(\zeta)^{-1} m_{\asympt,N}^{\Airy}(\zeta)&=I+\ord(\zeta^{-\frac{3(N+1)}{2}}),&\quad&\zeta\to\infty,\\
\label{mAiry-approx-mAiryb}
m_{\asympt,N}^{\Airy,\inv}(\zeta)m^{\Airy}(\zeta)&=I+\ord(\zeta^{-\frac{3(N+1)}{2}}),&&\zeta\to\infty,
\end{alignat}
\end{subequations}
where the error terms are uniform with respect to $\arg\zeta\in\croch{0,2\pi}$.
\end{theorem}
%-------------------%

%-------------------%
\begin{proof}
The analyticity of $m^{\Airy}$ is a direct consequence of the Airy function $\Airy(\zeta)$ being entire. The jump condition \eqref{mAiry-jump} can be verified by means of the identity
\[
\Airy(\zeta)+\omega\Airy(\omega\zeta)+\omega^2\Airy(\omega^2\zeta)=0,\quad\zeta\in\D{C}.
\]
On the other hand, the analyticity of $m_{\asympt,N}^{\Airy}(\zeta)$ and  $m_{\asympt,N}^{\Airy,\inv}(\zeta)$ is immediate from \eqref{mAiry-approx-def}. Using that $(\zeta^{3/2})_{\pm}=\mp\ii\abs{\zeta}^{3/2}$ and $(\zeta^{1/4})_{\pm}=\eul^{\pm\frac{\pi\ii}{4}}\abs{\zeta}^{1/4}$ for $\zeta\in(-\infty,0)$, a straightforward computation gives \eqref{mAiry-approx-jump}. This proves (a) and (b).

In order to prove (c), we note that, for each small number $\delta>0$, the Airy function satisfies the following asymptotic expansions uniformly in the stated sectors, see \cite{Ol10}:
\begin{equation}  \label{Airy-1st-asymptot-expans}
\begin{cases}
\Airy(z)=\frac{\eul^{-\frac{2}{3}z^{3/2}}}{2\sqrt{\pi}z^{1/4}}\sum\limits_{k=0}^{\infty}\frac{(-1)^ku_k}{\left(\frac{2}{3}z^{3/2}\right)^k},&\\[1mm]
\Airy'(z)=-\frac{z^{1/4}\eul^{-\frac{2}{3}z^{3/2}}}{2\sqrt{\pi}}\sum\limits_{k=0}^{\infty}\frac{(-1)^k(6k+1)u_k}{(1-6k)\left(\frac{2}{3}z^{3/2}\right)^k},&
\end{cases}\qquad z\to\infty,\quad\abs{\arg z}\leq\pi-\delta,
\end{equation}
and
\begin{align}  \label{Airy-2nd-asymptot-expans}
&\begin{cases}
\Airy(z)=\frac{(-z)^{-1/4}}{\sqrt{\pi}}\sum\limits_{k=0}^{\infty}(-1)^k\left(\frac{u_{2k}\cos\left(\frac{2}{3}(-z)^{3/2}-\frac{\pi}{4}\right)}{\left(\frac{2}{3}(-z)^{3/2}\right)^{2k}}+\frac{u_{2k+1}\sin\left(\frac{2}{3}(-z)^{3/2}-\frac{\pi}{4}\right)}{\left(\frac{2}{3}(-z)^{3/2}\right)^{2k+1}}\right),&\\[1mm]
\Airy'(z)=\frac{(-z)^{1/4}}{\sqrt{\pi}}\sum\limits_{k=0}^{\infty}(-1)^k\left(\frac{\nu_{2k}\sin\left(\frac{2}{3}(-z)^{3/2}-\frac{\pi}{4}\right)}{\left(\frac{2}{3}(-z)^{3/2}\right)^{2k}}-\frac{\nu_{2k+1}\cos\left(\frac{2}{3}(-z)^{3/2}-\frac{\pi}{4}\right)}{\left(\frac{2}{3}(-z)^{3/2}\right)^{2k+1}}\right),&
\end{cases}\notag\\
&\hspace{76mm} z\to\infty,\quad\frac{\pi}{3}+\delta\leq\arg z\leq \frac{5\pi}{3}-\delta.
\end{align}
The function $m_{\asympt,N}^{\Airy}(\zeta)$ is defined by the expression obtained by substituting the asymptotic sums in \eqref{Airy-1st-asymptot-expans} from $k=0$ to $k=N$ into the definition of $m^{\Airy}(\zeta)$ for $\zeta\in S_1$. Similarly, $m_{\asympt,N}^{\Airy,\inv}(\zeta)$ is defined by the expression obtained by substituting the asymptotic sums in \eqref{Airy-1st-asymptot-expans} from $k=0$ to $k=N$ into the following expression for $m^{\Airy}(\zeta)^{-1}$ which is valid for $\zeta\in S_1$:
\[
m^{\Airy}(\zeta)^{-1}=2\pi\ii\omega^2\eul^{-\frac{2}{3}\zeta^{3/2} \sigma_3}\eul^{\frac{\pi\ii}{6}\sigma_3} 
\begin{pmatrix}\omega^2 \Airy'(\omega^2\zeta)&-\Airy(\omega^2\zeta)\\-\Airy'(\zeta) & \Airy(\zeta)\end{pmatrix}.
\]
It follows that the estimates in \eqref{mAiry-approx-mAiry} hold as $\zeta\to\infty$ in $\bar S_1$. Long but straightforward computations using \eqref{Airy-1st-asymptot-expans} and \eqref{Airy-2nd-asymptot-expans} show that the estimates in \eqref{mAiry-approx-mAiry} hold as $\zeta\to\infty$ in $\bar S_2\cup\bar S_3\cup\bar S_4$ as well.
\end{proof}
%-------------------%

%---------------------------------------------------------%
%:app.C
%---------------------------------------------------------%
\section{Endpoint behavior of a Cauchy integral}\label{sec:C}

Let $\gamma\colon\croch{\alpha,\beta}\to\D{C}$ be a smooth simple contour from $a=\gamma(\alpha)$ to $b=\gamma(\beta)$ with $a\neq b$. By a slight abuse of notation, we let $\gamma$ denote both the map $\croch{\alpha,\beta}\to\D{C}$ and its image $\gamma\coloneqq\gamma(\croch{\alpha,\beta})$ as a subset of $\D{C}$.

%-------------------%
%:lem C.1
%-------------------%
\begin{lemma}   \label{lem-C}
Let $f\colon\gamma\to\D{C}$ be such that $t\mapsto f(\gamma(t))$ is $C^1$ on $\croch{\alpha,\beta}$. Suppose $g(k)$ is a $C^1$-function of $k\in\D{C}$ and define the function $h(k)$ by
\[
h(k)\coloneqq\frac{g(k)}{2\pi\ii}\int_{\gamma}\frac{f(s)\dd s}{s-k},\quad k\in\D{C}\setminus\gamma.
\]
Then the functions $\eul^{\ii h(k)}$ and $\eul^{-\ii h(k)}$ are bounded as $k\in\D{C}\setminus\gamma$ approaches $b$ if and only if the product $g(b)f(b)$ is purely imaginary.
\end{lemma}
%-------------------%

%-------------------%
\begin{proof}
We know (see \cite{Mu92}*{Appendix A2}) that the endpoint behavior of the Cauchy integral (here near $k=b$) is given by
\[
\frac{1}{2\pi\ii}\int_a^b\frac{f(s)}{s-k}\,\dd s=\frac{f(b)}{2\pi\ii}\ln(k-b)+\Phi_b(k),
\]
where $\Phi_b(k)$ is bounded near $k=b$. We thus have $h(k)=\frac{g(k)}{2\pi\ii}f(b)\ln(k-b)+\Psi_b(k)$ where $\Psi_b(k)$ is bounded near $k=b$. Moreover,
\[
\frac{g(k)}{2\pi\ii}f(b)\ln(k-b)=\frac{g(k)-g(b)}{2\pi\ii}f(b)\ln(k-b)+\frac{g(b)}{2\pi\ii}f(b)\ln(k-b),
\]
where the first term on the right-hand side is clearly bounded as $k \to b$. The term $\frac{g(b)}{2\pi\ii}f(b)\ln(k-b)$ is unbounded as $k\to b$, but its contribution to $\eul^{\pm\ii h(k)}$ is bounded iff the product $g(b)f(b)$ is purely imaginary.
\end{proof}
%-------------------%

%-------------------%
\begin{acknowledgements*}
J.~Lenells acknowledges support from the G\"oran Gustafsson Foundation, the Ruth and Nils-Erik Stenb\"ack Foundation, the Swedish Research Council, Grant No.~2015-05430, and the European Research Council, Grant Agreement No.~682537.
\end{acknowledgements*}
%---------------------------------------------------------%
%:bib
%---------------------------------------------------------%

%---------------------------------------------------------%
\end{document}